\documentclass{article}
\usepackage[a4paper,textwidth=15cm,textheight=21.5cm,centering]{geometry}
\usepackage[T1]{fontenc}
\usepackage[utf8]{inputenc} 
\usepackage{lmodern}
\usepackage{mathtools}   
\usepackage{amssymb, amsfonts}
\usepackage{amsthm}
\usepackage{bm}
\usepackage{graphicx}
\usepackage{booktabs, dcolumn}
\usepackage[tableposition=top]{caption}
\usepackage{tikz}
\usetikzlibrary{arrows.meta,calc,positioning}
\usepackage{xcolor}
\usepackage{hyperref}
\hypersetup{
    colorlinks=true,
    citecolor=blue,
    linkcolor=blue,
    filecolor=magenta,      
    urlcolor=cyan,
}
\usepackage{bookmark}   
\usepackage{enumerate}
\usepackage[shortlabels]{enumitem}
\numberwithin{equation}{section}
\theoremstyle{plain}
\newtheorem{theorem}{Theorem}[section]
\newtheorem{proposition}[theorem]{Proposition}
\newtheorem{lemma}[theorem]{Lemma}
\newtheorem{corollary}[theorem]{Corollary}
\newtheorem{conjecture}[theorem]{Conjecture}
\newtheorem{claim}[theorem]{Claim}

\theoremstyle{definition}
\newtheorem{definition}[theorem]{Definition}

\theoremstyle{remark}
\newtheorem{remark}[theorem]{Remark}
\newcommand{\Om}{\Omega}
\newcommand{\C}{\mathbb{C}}
\newcommand{\R}{\mathbb{R}}

\newcommand{\G}{\Gamma}
\newcommand{\g}{\gamma}
\newcommand{\Z}{\mathbb{Z}}

\newcommand{\sph}{\mathbb{S}}
\newcommand{\sub}{\subset}

\newcommand{\vp}{\varphi}
\newcommand{\p}{\partial}
\newcommand{\Ric}{\operatorname{Ric}}
\newcommand{\Bry}{\operatorname{Bry}}
\begin{document}
\title{Asymptotic Geometry of Four-Dimensional Steady Solitons}
\author{Aprameya Girish Hebbar\thanks{Department of Mathematics, Rutgers University, Piscataway, NJ 08854. Emails: \texttt{ah1531@math.rutgers.edu}, \texttt{natasas@math.rutgers.edu}.}
\and Nata\v{s}a \v{S}e\v{s}um\footnotemark[1]
}
\date{}
\maketitle
\begin{abstract}
In this paper we study the behavior of the scalar curvature at infinity on complete noncompact steady gradient Ricci solitons. In dimension four, we assume that the canonical Ricci flow induced by the soliton is a weak $\kappa$-solution and that the soliton is not isometric to the Bryant soliton. In this setting, we identify the two edges of the soliton and prove that the scalar curvature decays at a linear rate away from these edges. Moreover, if the scalar curvature vanishes at infinity, then a stronger inequality holds and the asymptotic cone is a ray. In particular, our results apply to the four-dimensional steady solitons constructed by Lai.
\end{abstract}

\section{Introduction}
\label{sec:Introduction}
A smooth one-parameter family of Riemannian metrics $(g(t))_{t\in[0,T]}$ on a smooth manifold $M^n$ (without boundary) is said to evolve by the Ricci flow if
$$\partial_t g(t) = -2\,\Ric_{g(t)} \qquad \text{on } M \times [0,T].$$
Since Hamilton's foundational work \cite{Ham82}, Ricci flow has become an important tool in geometric analysis, especially in Perelman's proof of the Poincar\'e and Geometrization conjectures \cite{Per02}. 

Ricci solitons give rise to self-similar solutions to the Ricci flow, evolving only by diffeomorphisms and scaling. They play a fundamental role in \textit{singularity analysis}: performing parabolic rescalings about spacetime points of large curvature yields, after passing to a subsequence, pointed Cheeger--Gromov limits that are ancient or eternal solutions. In many important settings, these limits are gradient shrinking solitons or gradient steady solitons. Accordingly, describing Ricci solitons is a key step toward understanding possible singularity formation in Ricci flow. In this paper, we focus on gradient steady Ricci solitons. 

Let $(M^n,g)$ be a Riemannian manifold and let $f\in C^\infty(M)$. The triple $(M^n,g,f)$ is called a gradient steady Ricci soliton if
$$\Ric_g + \nabla^2 f = 0 \qquad \text{on } M.$$
Equivalently, the Bakry--\'Emery Ricci tensor $\Ric_f := \Ric+\nabla^2 f$ vanishes; in particular, Ricci-flat manifolds are precisely the steady solitons with $f$ constant. When $(M,g)$ is complete, it gives rise to a Ricci flow solution that evolves only by diffeomorphisms. Let $\Phi_t$ denote the one-parameter family of diffeomorphisms generated by $\nabla f$ with $\Phi_0=\operatorname{id}_M$ and set $g(t):=\Phi_t^* g$. Then, $(M,g(t))_{t\in (-\infty,\infty)}$ is an \textit{eternal} solution to the Ricci flow, called the \textit{canonical Ricci flow} induced by $g$. 

In dimension two, Hamilton constructed the \textit{cigar soliton}, a nonflat collapsed rotationally symmetric gradient steady Ricci soliton. Moreover, any complete two-dimensional gradient steady Ricci soliton is either flat or isometric, up to scaling, to the cigar soliton \cite[Theorem~3.11]{Cho23}. In dimension three, Bryant constructed a complete nonflat noncompact rotationally symmetric steady gradient Ricci soliton, and within the rotationally symmetric class showed that it is unique up to scaling \cite{Bry05}. Bryant's construction extends to all dimensions and yields, for each $n\ge 3$, a complete nonflat asymptotically cylindrical $O(n)$-symmetric steady gradient soliton on $\mathbb{R}^n$, unique within the rotationally symmetric class up to scaling (see \cite[Chapter~6]{Cho23}). These are called the $n$-dimensional \emph{Bryant soliton}. 

For a long time, much of the classification theory aimed to identify geometric hypotheses forcing a steady soliton to be rotationally symmetric, and hence isometric (up to scaling) to the Bryant soliton. A landmark result in this direction is Brendle's theorem that any complete nonflat $\kappa$-noncollapsed three-dimensional steady gradient Ricci soliton is isometric, up to scaling, to the Bryant soliton \cite{Bre13}. In higher dimensions, Brendle proved that any $\kappa$-noncollapsed asymptotically cylindrical steady soliton with positive sectional curvature is also isometric, up to scaling, to the Bryant soliton \cite{Bre14}. Deng--Zhu \cite{DZ20} proved rigidity results under the hypothesis of linear curvature decay, while Zhao--Zhu \cite{ZZ22} obtained rigidity for the Bryant soliton under curvature pinching assumptions. Law \cite{Law25} studied rigidity in the asymptotically cylindrical setting without assuming any curvature positivity.

Lai constructed the \emph{flying wing} steady solitons, a family of examples that are not asymptotically cylindrical. In dimension three, Lai confirmed a conjecture of Hamilton by producing a family of $\mathbb{Z}_2\times O(2)$-symmetric steady solitons whose asymptotic cone is a sector of opening angle $\alpha\in(0,\pi)$ \cite{Lai24}. The same work also yields, for every $n\ge 4$, a family of $\mathbb{Z}_2\times O(n-1)$-symmetric but non-rotationally symmetric steady solitons with positive curvature operator; unlike the three-dimensional flying wings, these higher-dimensional examples are $\kappa$-noncollapsed \cite{Lai24}. 

The three-dimensional steady solitons also have the following asymptotic behavior. In the non-Bryant case, the asymptotic cone is a sector of positive opening angle, rather than a ray. Along the ``edge'' of the wing, the scalar curvature converges to a positive limit depending on the opening angle \cite{Lai25}. 

Motivated by these developments, we study four-dimensional steady gradient Ricci solitons, with particular emphasis on the asymptotic geometry of the flying wings. Rather than restricting to the specific examples constructed by Lai, we work in a broader class of $\kappa$-noncollapsed steady solitons whose canonical Ricci flows (viewed as ancient solutions) are $\kappa$-solutions. 

As introduced by Perelman \cite{Per02}, a $\kappa$-solution is a complete $\kappa$-noncollapsed ancient solution of the Ricci flow, with bounded nonnegative curvature operator and positive scalar curvature. $\kappa$-solutions often arise as singularity models near cylindrical singularities. In dimension three, the analysis of $\kappa$-solutions played a crucial role in Perelman's work. In dimension four, the asymptotic shrinker of a noncompact $\kappa$-solution is a noncompact nonflat gradient shrinker with nonnegative curvature operator. By Munteanu--Wang \cite{MW17}, they must be either $\sph^3\times \R$ or $\sph^2\times \R^2$, up to finite quotients. The case when $\sph^3\times \R$ occurs as an asymptotic shrinker has been studied in \cite{LZ22,Heb26}. The case corresponding to $\sph^2\times\R^2$ remains largely open and is relevant to four-dimensional flying wings.
 
A conjectural picture for $\kappa$-solutions in dimension four has been proposed by Haslhofer \cite{Has24}; when specialized to the steady case, it predicts that no further steady solitons that are also $\kappa$-solutions exist beyond the known ones. 
\begin{conjecture}
\label{conj:Haslhofer-steady}
Any $4$D $\kappa$-noncollapsed steady gradient Ricci soliton with nonnegative curvature operator and positive scalar curvature is, up to scaling and finite quotients, one of the following: the $4$D Bryant soliton; the product of the $3$D Bryant soliton with a line; or an element of the one-parameter family of $\mathbb{Z}_2\times O(3)$-symmetric steady solitons constructed by Lai \textup{\cite{Lai24}}. 
\end{conjecture}
For four-dimensional steady solitons with nonnegative sectional curvature that are $\kappa$~noncollapsed, Chan--Ma--Zhang \cite{CMZ25b} proved that the tangent flow at infinity is either $\sph^3\times\R$ or $\sph^2\times\R^2$. In the $\sph^2\times\R^2$ case, they further showed that the manifold dimension reduces at infinity to $\Bry^3$ or $\sph^2\times\R$. In \cite{MMS26}, Ma, Mahmoudian, and the second named author studied the corresponding asymptotics under ${O}(3)$-symmetry together with additional curvature decay hypotheses (for example, $R(x)\sim d(x,o)^{-\eta}$). While such decay is expected, one goal of the present paper is to derive information about the manifold without imposing \emph{a priori} curvature decay or symmetry assumptions. 
\begin{remark}
If one moves beyond the positively curved, $\kappa$-noncollapsed setting, the landscape of steady solitons in dimension four becomes considerably richer, and a general classification appears extremely difficult.

Collapsed examples include Cao's $U(2)$-invariant steady K\"ahler--Ricci soliton on $\C^2$ \cite{Cao96};
the Koiso-type steady K\"ahler--Ricci solitons of Yang \cite{Yan12};
the non-K\"ahler steady solitons of Buzano--Dancer--Wang \cite{BDW15};
the product of Lai's three-dimensional flying wing \cite{Lai24} with a line;
the steady K\"ahler--Ricci solitons of Biquard--Macbeth \cite{BM24} on crepant resolutions of finite quotients of $\C^2$;
the $U(1)\times U(1)$-invariant steady K\"ahler--Ricci solitons on $\C^2$ of Apostolov--Cifarelli \cite{AC25};
the steady K\"ahler--Ricci solitons of Conlon--Deruelle \cite{CD25};
the $U(1)\times U(1)$-invariant K\"ahler flying wings of Chan--Conlon--Lai \cite{CCL24};
and the $\Z_2^2\times O(2)$- and $O(2)\times O(2)$-symmetric flying wing constructions of Chan--Lai--Lee \cite{CLL25} and Lavoyer--Peachey \cite{LP25}.

There are also examples that do not have nonnegative sectional curvature, such as steady solitons on $\R^2\times \sph^2$ of Ivey \cite{Ive94}, the $\kappa$-noncollapsed steady solitons of Dancer--Wang \cite{DW09}, and Appleton's $\kappa$-noncollapsed steady soliton \cite{App17}, all of which have linear scalar curvature decay at infinity. 

Other examples include Sch\"afer's asymptotically cylindrical steady solitons \cite{Sch21}, Stolarski's $U(1)$-invariant steady solitons on complex line bundles over $\C\mathbb{P}^1$ \cite{Sto24}, and Sch\"afer's $\sph^1$-invariant steady K\"ahler--Ricci solitons \cite{Sch23}; none of which has positive Ricci curvature. 
\end{remark}

\subsection{Setting}
\label{subsec:Setting}
Guided by Conjecture \ref{conj:Haslhofer-steady}, we study the geometry at infinity of four-dimensional steady solitons which are $\kappa$-noncollapsed and nonnegatively curved. We begin by introducing our basic setting and standing assumptions.

\medskip 

\noindent We will consider $(M^4,g,f)$, a complete noncompact gradient steady Ricci soliton: 
$$\Ric_g+\nabla^2 f=0\qquad \text{on }M.$$
\noindent \textbf{Assumptions. }We will impose some or all of the following assumptions:
\begin{enumerate}[label=({\bf A\arabic*}), ref=A\arabic*]
    \item\label{assumption:A1} $\Ric_g > 0$ on $M$. 
    \item  \label{assumption:A2} The soliton potential function $f$ has a critical point at $o\in M$. 
    \item \label{assumption:A3} The canonical Ricci flow $(M^4,g(t))_{t\in (-\infty,1]}$ induced by $g$ has nonnegative sectional curvature and is $\kappa$-noncollapsed on all scales. 
\end{enumerate}

\begin{remark}
Our setting is inspired by the approach of \cite{MMS26}. Note that any compact gradient steady Ricci soliton is Ricci-flat; in particular, under \hyperref[assumption:A1]{(A1)} the soliton is necessarily noncompact and nonflat. Under \hyperref[assumption:A1]{(A1)} and \hyperref[assumption:A3]{(A3)}, since the scalar curvature is bounded above on a steady soliton, one obtains bounded curvature on spacetime for the canonical Ricci flow. Consequently, under \hyperref[assumption:A1]{(A1)}--\hyperref[assumption:A3]{(A3)}, $(M^4,g(t))_{t\in(-\infty,1]}$ is a \textbf{weak $\kappa$-solution} in the sense of \cite{MMS26}, i.e.\ an ancient, complete, $\kappa$-noncollapsed Ricci flow with nonnegative bounded sectional curvature and positive scalar curvature. Moreover, by \cite[Theorem~1.13]{MZ21}, $\kappa$-noncollapsing is equivalent to a lower bound for Perelman's entropy.
\end{remark}

\noindent Known nonflat four-dimensional steady solitons satisfying \hyperref[assumption:A1]{(A1)}--\hyperref[assumption:A3]{(A3)} include:
\begin{enumerate}
\item  $O(4)$-symmetric Bryant soliton $\Bry^4$, which has positive curvature operator and is asymptotically cylindrical (in the sense of \cite{Bre14}), with tangent flow at infinity $\sph^3\times\R$. Moreover, its scalar curvature $R(x)$ decays linearly: there exist $c_1,c_2>0$ such that
$$\frac{c_1}{1+d(x,o)}\leq R(x)\leq \frac{c_2}{1+d(x,o)}\qquad \text{for all }x\in \Bry^4.$$
\item The one-parameter family of $\Z_2\times O(3)$-symmetric steady solitons constructed by Lai \cite{Lai24}. Each of these solitons has positive curvature operator and admits $\sph^2\times\R^2$ as its unique tangent flow at infinity. 
\end{enumerate}
By \cite[Corollary~5.4]{CMZ25b}, if $(M^4,g,f)$ satisfies \hyperref[assumption:A1]{(A1)} and \hyperref[assumption:A3]{(A3)}, then the tangent flow at infinity is unique and equal to either $\sph^2\times\R^2$ or $\sph^3\times\R$. Moreover, if the tangent flow at infinity is $\sph^3\times \R$, then $(M^4,g)$ is isometric to $\Bry^4$ with its soliton metric; see \cite[Theorem~5.2]{CMZ25b}. To exclude the Bryant soliton, we impose the following additional assumption:
\begin{enumerate}[label=({\bf A\arabic*}), ref=A\arabic*, start=4]
    \item \label{assumption:A4} The tangent flow at infinity of $(M^4,g)$ is $\sph^2\times\R^2$.
\end{enumerate}
Under \hyperref[assumption:A1]{(A1)} and \hyperref[assumption:A3]{(A3)}, Assumption \hyperref[assumption:A4]{(A4)} is equivalent to the assertion that $(M^4,g)$ is not isometric to the Bryant soliton. It is also equivalent to the statement that Perelman's asymptotic shrinker of the canonical Ricci flow $(M^4,g(t))_{t\leq 0}$ is $\sph^2\times \R^2$; see \cite[Theorem~1.3]{CMZ25a}. 

\subsection{Main Results}
\label{subsec:Main Results}
Under Assumptions \hyperref[assumption:A1]{(A1)}--\hyperref[assumption:A4]{(A4)}, the geometry at infinity of $(M^4,g)$ is not yet fully understood. By \cite{DZ20}, it is known that the scalar curvature cannot decay linearly, equivalently, we have $\limsup_{d(x,o)\to \infty}R(x)d(x,o)=+\infty$. It remains open whether the scalar curvature satisfies $R(x)\to 0$ as $d(x,o)\to\infty$. We obtain partial results in this direction by identifying two edges of the soliton and proving quantitative decay of the scalar curvature away from these edges. 
\begin{theorem}
\label{Full-linear-bound}
Let $(M^4,g,f)$ be a complete steady soliton satisfying \textup{\hyperref[assumption:A1]{(A1)}, \hyperref[assumption:A2]{(A2)}, \hyperref[assumption:A3]{(A3)},} and \textup{\hyperref[assumption:A4]{(A4)}}. Then there exist two curves $\G_1,\G_2:[0,\infty)\to M$ with $\G_i(0)=o$ such that $\G_i((0,\infty))$ is an integral curve of $-\nabla f/|\nabla f|$, for $i=1,2$. If $\G:=\G_1([0,\infty))\cup \G_2([0,\infty))$, then the scalar curvature satisfies
$$R(x)\leq \frac{C}{d_g(x,\G)}\qquad \text{for all }x\in M\setminus \G,$$
where $C$ depends on $\kappa$ and $(M,g)$.
\end{theorem}
Here $R$ denotes the scalar curvature of $g$ and $d_g(x,\G)$ denotes the distance from $x$ to $\G$. We refer to the curves $\G_1$ and $\G_2$ as the \emph{edges} of the soliton (see Definition \ref{defn:edges-of-soliton} and Theorem \ref{existence-of-x-+-} for properties of $\G_1,\G_2$). 

\begin{figure}[ht]
\centering
\begin{tikzpicture}[
    x=1.15cm,y=1.15cm,
    >={Latex[length=2.4mm]},
    line cap=round,
    line join=round,
    every node/.style={font=\small}
]

\definecolor{myblue}{RGB}{60,100,190}
\definecolor{lightblue}{RGB}{232,239,255}
\definecolor{myred}{RGB}{185,50,50}
\definecolor{levelset}{RGB}{35,35,35}

\pgfmathsetmacro{\xmax}{9.6}
\pgfmathsetmacro{\xc}{4.0}         
\pgfmathsetmacro{\aV}{0.494}        
\pgfmathsetmacro{\bV}{1.0}          

\path[fill=lightblue]
plot[domain=0:\xmax,samples=160,smooth] (\x,{0.5*sqrt(\x)})
-- plot[domain=\xmax:0,samples=160,smooth] (\x,{-0.5*sqrt(\x)})
-- cycle;

\draw[myblue, line width=1pt]
plot[domain=0:\xmax,samples=160,smooth] (\x,{0.5*sqrt(\x)});

\draw[myblue, line width=1pt]
plot[domain=0:\xmax,samples=160,smooth] (\x,{-0.5*sqrt(\x)});

\draw[myred, very thick]
plot[domain=0:\xmax,samples=160,smooth] (\x,{0.5*sqrt(\x)});

\draw[myred, very thick]
plot[domain=0:\xmax,samples=160,smooth] (\x,{-0.5*sqrt(\x)});


\fill (0,0) circle (1.4pt);
\node[left] at (-0.08,0) {$o$};

\node[above] at (6.8,1.42) {$\Gamma_1$};
\node[below] at (6.8,-1.42) {$\Gamma_2$};

\node at (7.1,0.00) {$\sph^2\times\mathbb{R}^2$};

\node[font=\scriptsize, anchor=west] at (8.0,1.20)
{$\mathrm{Bry}^3\times\mathbb{R}$};

\node[font=\scriptsize, anchor=west] at (8.0,-1.20)
{$\mathrm{Bry}^3\times\mathbb{R}$};

\draw[levelset, line width=1.3pt]
(\xc,0) ellipse[x radius=\aV, y radius=\bV];

\fill[levelset] (\xc,\bV) circle (1.1pt);
\fill[levelset] (\xc,-\bV) circle (1.1pt);

\node[above right] at (\xc,\bV) {$x_{+}$};
\node[below right] at (\xc,-\bV) {$x_{-}$};

\foreach \yy/\xr in {
    -0.72/0.34,
    -0.38/0.46,
     0.00/0.50,
     0.38/0.46,
     0.72/0.34
}{
    \draw[levelset!75, line width=0.85pt]
    (\xc,\yy) ellipse[x radius=\xr, y radius=0.075];
}

\node[fill=white, inner sep=1.5pt] (sigmalabel) at (2.55,1.42)
{$\Sigma=f^{-1}(s_0)$};

\draw[->, levelset]
(sigmalabel.east) to[out=0,in=145] (3.72,0.82);

\node[font=\scriptsize, fill=white, inner sep=1pt] at (\xc-0.7,0.22) {$\sph^3$};

\end{tikzpicture}
\caption{Structure of $M$ at infinity.}
\label{fig:a-picture-of-the-soliton-with-x-pm-Gamma}
\end{figure}

Lai's $4$D solitons are constructed indirectly (as limit of a family of Ricci expanders) \cite{Lai24}. Although these examples are known to have positive curvature operator, their asymptotic geometry is not yet well understood. Our result above gives a curvature estimate for the $4$D flying wings. 
\begin{corollary}
Theorem \textup{\ref{Full-linear-bound}} applies to the $\Z_2\times O(3)$-symmetric $4$D steady solitons constructed by Lai \textup{\cite{Lai24}}. 
\end{corollary}
In dimension three, Lai obtained polynomial decay of the scalar curvature on non-Bryant steady solitons, away from two edges: for every $k\ge 1$ there exists $C_k>0$ such that $R(x)\le C_k\, d(x,\text{edges})^{-k}$  \cite{Lai25}. In our noncollapsed setting, such higher-order decay is not expected (see Remark~\ref{rem:polynomial-decay-not-possible}); nevertheless one can ask whether the stronger condition $R(x)\,d(x,\G)\to 0$ holds as $d(x,\G)\to\infty$. Our next result shows that if $R(x)\to 0$ at infinity, then the decay away from the edges improves and the asymptotic cone is a ray. 
\begin{theorem}
\label{thm:main-theorem-2-stronger-scalar-bound}
Let $(M^4,g,f)$ be a complete steady soliton satisfying \textup{\hyperref[assumption:A1]{(A1)}, \hyperref[assumption:A2]{(A2)}, \hyperref[assumption:A3]{(A3)},} and \textup{\hyperref[assumption:A4]{(A4)}}. Suppose $\lim_{d(x,o)\to \infty}R(x)=0$. Then,  
$$\lim_{d(x,o)\to \infty}R(x)d(x,\G)=0,$$
and the asymptotic cone of $(M,g)$ is a ray. 
\end{theorem}
The next theorem gives a converse to Theorem~\ref{thm:main-theorem-2-stronger-scalar-bound} under weaker assumptions. 
\begin{theorem}
\label{thm:converse-to-main-theorem-2-stronger-scalar-bound-4d}
Let $(M^4,g,f)$ satisfy \textup{\hyperref[assumption:A1]{(A1)}, \hyperref[assumption:A2]{(A2)},} and $\sec\geq 0$. If the asymptotic cone of $(M,g)$ is a ray, then
$$
\lim_{d(x,o)\to \infty}R(x)=0.
$$
\end{theorem}
\begin{remark}
Although we state Theorem~\ref{thm:converse-to-main-theorem-2-stronger-scalar-bound-4d} in dimension four to match the setting of this paper, the proof applies verbatim in every dimension (Theorem~\ref{thm:dimension-free-converse-to-main-theorem-2}). 
\end{remark}
Combining Theorems~\ref{thm:main-theorem-2-stronger-scalar-bound} and \ref{thm:converse-to-main-theorem-2-stronger-scalar-bound-4d}, we obtain the following equivalence in our setting. 
\begin{corollary}
\label{cor:ray-iff-R-vanishes}
Let $(M^4,g,f)$ satisfy \textup{\hyperref[assumption:A1]{(A1)}, \hyperref[assumption:A2]{(A2)}, \hyperref[assumption:A3]{(A3)},} and \textup{\hyperref[assumption:A4]{(A4)}}. Then the following are equivalent:
\begin{enumerate}
    \item $\lim_{d(x,o)\to\infty}R(x)=0$;
    \item $\lim_{d(x,o)\to\infty}R(x)d(x,\G)=0$;
    \item the asymptotic cone of $(M,g)$ is a ray.
\end{enumerate}
\end{corollary}

\subsection{Outline of the paper}

We now outline the organization of the paper and the main ideas. Unless otherwise stated, throughout this discussion $(M^4,g,f)$ denotes a steady soliton satisfying \hyperref[assumption:A1]{(A1)}--\hyperref[assumption:A4]{(A4)}, and $g(t)$ denotes the canonical Ricci flow induced by $g$ with $g(0)=g$.

In Section \ref{sec:Preliminaries} we collect background material and fix notation, including the foliation of $M\setminus \{o\}$ by level sets $\Sigma_s:=f^{-1}(s),s<f(o)$ (each diffeomorphic to $\sph^3$). To study scalar curvature at infinity, we introduce a function $G$ on a fixed level set $\Sigma_{s_0}$ that records the limit of the scalar curvature along the integral curve of $-\nabla f/|\nabla f|$ starting at each $q\in \Sigma_{s_0}$.

In Section \ref{sec:Structure-of-level-sets} we analyze the geometry of the far-out level sets $\Sigma_s$ via dimension reduction on $M$. Using an argument in the spirit of \cite{BDS21}, we show that for $s\ll 0$ each $\Sigma_s$ admits a neck-cap decomposition. We then prove that each far-out level set has exactly two tips. This implies that the set where the function $G$ is nonzero consists of at most two points on $\Sigma_{s_0}$. Since Lai's $4$D steady solitons satisfy our standing assumptions, all of our intermediate results apply to them as well. 

In Section \ref{sec:Stability-of-e-2-centers} we prove a stability statement for points near which the manifold is close to $\sph^2\times\R^2$: closeness persists when flowing backward along the vector field $-\nabla f$. As an application, we identify two distinguished points $x_+,x_-\in\Sigma_{s_0}$ and the corresponding ``edge'' integral curves. We show that along these edges the geometry is modeled by $\Bry^3\times\R$, whereas away from the edges it is modeled by $\sph^2\times\R^2$. 

A difficulty in proving Theorem~\ref{Full-linear-bound} is that there is no direct maximum-principle route to the desired estimate. Moreover, Lai's proof of quadratic decay of scalar curvature on collapsed $3$D wings \cite[Theorem~3.20]{Lai25} relies on tools that are intrinsically three-dimensional and do not directly extend to our four-dimensional noncollapsed setting; our approach is therefore different. In Section \ref{sec:Scalar-curvature-along-integral-curves} we show that if $(M,g)$ is sufficiently close to $\sph^2\times\R^2$ at $x$, then 
$$s\,R_{g(-s)}(x)\leq C\qquad \text{for all }s>0$$
where $C$ is a uniform positive constant. We then globalize this bound in Section \ref{sec:proof-of-the-main-result} using distance distortion and a contradiction argument, obtaining the linear decay estimate away from the edges. The same argument also yields the stronger decay when $\lim_{d(x,o)\to\infty}R(x)=0$. 

Finally, in Section \ref{sec:asymptotic-cone-of-the-soliton} we study the relation between the asymptotic cone of $(M,g)$ and the behavior of the scalar curvature at infinity. Under \hyperref[assumption:A1]{(A1)}--\hyperref[assumption:A4]{(A4)} and $R\to 0$ at infinity, we compare distances between geodesic rays and show that the asymptotic cone of $(M,g)$ is a ray. We then prove, under weaker assumptions, the converse implication in Theorem~\ref{thm:converse-to-main-theorem-2-stronger-scalar-bound-4d}. The proof is based on an angle function on the unit sphere in $T_oM$ associated with the asymptotic behavior of geodesic rays relative to $-\nabla f$. 

\subsection{Acknowledgments}
The authors thank Yi Lai, Zilu Ma, Ovidiu Munteanu, and Junming Xie for inspiring discussions. The second named author thanks the NSF for support through grant DMS-2505574. 

\section{Preliminaries}
\label{sec:Preliminaries}
In this section, we collect background material and prior results that will be used throughout the paper.

Let $(M^4,g,f)$ be a steady soliton satisfying \hyperref[assumption:A1]{(A1)} and \hyperref[assumption:A2]{(A2)}. Hamilton \cite{Ham95} showed that the quantity $R+|\nabla f|^2$ is constant on $M^4$. After rescaling $g$ by a positive constant, we assume throughout that
\begin{equation}\label{eq:normalization}
R+|\nabla f|^2=1 \qquad\text{on }M.
\end{equation}
In particular, since $o\in M$ is a critical point of $f$, $R(o)=1$. We use the Laplacian $\Delta=\mathrm{div}\,\nabla$. Tracing the soliton equation yields
\begin{equation}\label{eq:trace}
R+\Delta f=0.
\end{equation}
Moreover, the standard soliton identities give
\begin{align}
\nabla R &= 2\,\Ric(\nabla f), \label{eq:gradR}\\
\Delta R +2|\Ric|^2 &= \langle \nabla R,\nabla f\rangle. \label{eq:scalarPDE}
\end{align}
Under \hyperref[assumption:A1]{(A1)} we have $R>0$, and from \eqref{eq:normalization} it follows that
\begin{equation}\label{eq:bounds}
0<R\le 1,\qquad |\nabla f|^2\leq 1 \qquad\text{on }M.
\end{equation}
Consequently, for any $x,p\in M$,
\begin{equation}\label{eq:Lipschitz}
|f(x)-f(p)|\leq d_g(x,p),
\end{equation}
by integrating $|\nabla f|\leq 1$ along a minimizing geodesic segment joining $p$ and $x$. For further background on Ricci solitons, we refer the reader to \cite{Cho23}. 
\begin{lemma}
\label{lem:Linear-Growth-of-potential}
Let $(M^4,g,f)$ be a complete gradient steady Ricci soliton satisfying \textup{\hyperref[assumption:A1]{(A1)}} and \eqref{eq:normalization}. Then $f$ has at most one critical point. If $o$ is a critical point of $f$, then $o$ is the unique global maximum of $f$. Moreover, setting
$$c_0:= f(o)-\max_{\{y:\,d_g(o,y)=1\}} f(y)>0,$$
one has for every $x\in M$ with $d_g(o,x)\ge 1$,
\begin{equation}\label{eq:linear-lower}
f(o)-f(x)\ge c_0\,d_g(o,x).
\end{equation}
In particular, $f(x)\to -\infty$ along any divergent curve, and $f^{-1}([a,b])$ is compact for all $a\leq b<f(o)$. 
\end{lemma}
\begin{proof}
Using \hyperref[assumption:A1]{(A1)} and the soliton equation, we have $\nabla^2 f<0$ on $M$. Suppose $p,q$ are two distinct critical points of $f$. Let $\sigma:[0,b]\to M$ be a unit-speed geodesic joining $p$ and $q$. Set $h(t)=f(\sigma(t))$ and observe that $h'(0)=0,h'(b)=0$ but $h''(t)=\nabla_{\sigma',\sigma'}f|_{\sigma(t)}<0$ for all $t\in (0,b)$. This is a contradiction. This proves that the critical point of $f$, if it exists, is unique. 

Assume that $o$ is a critical point of $f$. Let $\g:[0,b]\to M$ be any unit-speed minimizing geodesic segment starting at $o$ and ending at $x:=\g(b)$. Suppose $b=d(x,o)>1$, and write $h(s):=f(\g(s))$ for $s\in [0,b]$. Then, $h:[0,b]\to \R$ is concave which implies for $t\in [0,1]$, 
$$
\begin{aligned}
f\left(\g(t)\right)&=h(t)=h\left(\frac{b-1}{b}\cdot 0+\frac{1}{b}\left(b t\right)\right) \geq \frac{b-1}{b} h\left(0\right)+\frac{1}{b} h\left(bt\right).
\end{aligned}
$$
Rearranging the terms, it follows that 
$$h(0)-h(bt)\geq b(h(0)-h(t)).$$
Taking $t=1$, we get the estimate
$$ f(o)-f(x)\geq d(x,o) \left[f(o)-f\left(\g(1)\right)\right]\geq d(x,o)\left[f(o)-\max_{d(y,o)=1}f(y)\right],$$
proving \eqref{eq:linear-lower}. 

The remaining claims follow immediately from \eqref{eq:linear-lower}: if $\gamma$ is divergent then $d_g(o,\gamma(t))\to\infty$, hence
$f(\gamma(t))\leq f(o)-c_0 d_g(o,\gamma(t))\to-\infty$. Finally, for any $a<f(o)$, $\{f\ge a\}\subset \bar{B}_g\!\left[o;\frac{f(o)-a}{c_0}\right]$, which shows that $f^{-1}([a,b])$ is compact for all $a\leq b<f(o)$.
\end{proof}
\begin{remark}
\label{rem:R-strict-less-one}
Under Assumptions \hyperref[assumption:A1]{(A1)} and \hyperref[assumption:A2]{(A2)} it follows from Lemma \ref{lem:Linear-Growth-of-potential} and  \eqref{eq:gradR} that $0<R<R(o)=1$ on $M\setminus \{o\}$. 
\end{remark}
We continue to assume that $(M^4,g,f)$ is a complete steady soliton that satisfies \hyperref[assumption:A1]{(A1)} and \hyperref[assumption:A2]{(A2)}. We denote the level sets of $f$ by
$$\Sigma_s:=f^{-1}(s)\qquad \text{ for }s<f(o),$$
each of which is compact due to Lemma \ref{lem:Linear-Growth-of-potential}. By \cite[Lemma 2.3]{DZ21}, $\Sigma_s$ is diffeomorphic to $\sph^3$. By \cite{CC12}, $M$ is diffeomorphic to $\R^4$. Fix $s_0<f(o)$ and let $\Sigma:=\left\{f=s_0\right\}$. 

\textbf{Integral curves, diffeomorphisms, and notation.} By \cite{Zha09}, $\nabla f$ is a complete vector field on $M$. Let $\Phi_t:M\to M$ be diffeomorphisms such that 
$$\p_t \Phi_t(p)=\nabla f|_{\Phi_t(p)}\qquad \text{for }t\in \R,p\in M,\qquad \text{with }\Phi_0= \operatorname{id}_M.$$
Then $g(t)=\Phi^*_tg$ is called the canonical flow induced by $g$ and satisfies $\Ric_{g(t)}+\nabla^{2,g(t)}{f}_1(t)\equiv 0$ on $M$ where ${f}_1(t):=f\circ \Phi_t$. Throughout, we write $g:=g(0)$ for the time-$0$ soliton metric, and we use $g(t)$ (or $g_t$) to denote the associated canonical Ricci flow. Quantities such as $R(x)$ denote the scalar curvature of $g$. When several metrics are involved, we write $R_{g(t)}$ for the scalar curvature of $g(t)$.

Let $(\chi_s:\Sigma\to \Sigma_s)_{-\infty<s\leq s_0}$ be a family of diffeomorphisms satisfying 
$$
\frac{\p}{\p s}\chi_s(p)=\left.\frac{\nabla f}{|\nabla f|^2}\right|_{\chi_s(p)}\qquad \text{for }s<s_0,p\in \Sigma,\, \qquad \text{with }\chi_{s_0}=\operatorname{id}_{\Sigma}.
$$
Then, for $p\in \Sigma$, along each $s\mapsto \chi_s(p)$ one has $\frac{d}{ds}f(\chi_s(p))=1$, hence $f(\chi_s(p))\equiv s$. For each $q\in M\setminus \{o\}$, let $\G_q:[0,\infty)\to M$ be a smooth curve such that
\begin{equation}
\label{eqn:definition-of-Gamma-p}
\G_q'(t)=-\frac{\nabla f}{|\nabla f|}(\G_q(t))\qquad \text{for }t>0,\qquad \text{with }\G_q(0)=q.
\end{equation}
We note the following monotonicity properties of $f$, $R$, $d(\cdot,o)$, and $|\nabla f|$ along the above curves.  
\begin{lemma}
\label{lem:monotonicity-along-flows}
Let $q\in M\setminus \{o\}$ and $\tau\ge 0$. Along each of the curves
\[
\tau\mapsto \Phi_{-\tau}(q),\qquad \tau\mapsto \Gamma_q(\tau),
\]
the function $f$ is strictly decreasing and satisfies $f\to -\infty$ as $\tau\to\infty$. Moreover, the scalar curvature $R$ is strictly decreasing, $d(o,\cdot)$ is strictly increasing, and $|\nabla f|$ is strictly increasing along each curve. In particular,
\[
\lim_{\tau\to\infty} d_g(\Gamma_q(\tau),o)=\infty,
\]
and likewise for $\Phi_{-\tau}(q)$. If in addition $q\in \Sigma=\Sigma_{s_0}$, then the same conclusions hold along the curve $\tau\mapsto \chi_{s_0-\tau}(q)$. 
\end{lemma}
\begin{proof}
We prove the lemma for $\G_q(\tau)$, the proofs for $\Phi_{-\tau}(q),\chi_{s_0-\tau}(q)$ proceed similarly. Consider the unit vector field $X:=-\nabla f/|\nabla f|$ on $M\setminus \{o\}$. By standard ODE existence theory, there exists an integral curve $\G_q:(a,b)\to M$ of $X$, passing through $q=\G_q(0)$ where $(a,b)$ is the maximal time of existence. Since $\frac{d}{dt}f(\G_q(t))=-|\nabla f||_{\G_q(t)}< 0$ for $t\in (a,b)$, $f$ decreases along $\G_q$. From Lemma~\ref{lem:Linear-Growth-of-potential} and the escape criterion for integral curves, it follows that $b=\infty$, $d(o,\G_q(t))\to \infty$ as $t\to \infty$, $f(\G_q(t))\to -\infty$ as $t\to \infty$, and that $d(\G_q(t),o)\to 0$ as $t\to a$. Using \eqref{eq:gradR}, we have $$\frac{d}{dt}R(\G_q(t))=-\frac{2\Ric(\nabla f,\nabla f)|_{\G_q(t)}}{|\nabla f|(\G_q(t))}< 0,$$ for $t\in (a,b)$, hence $R$ decreases along $\G_q$. Since $R+|\nabla f|^2=1$, it follows that $|\nabla f|$ increases along $\G_q$. 

It remains to show that $d(o,\cdot)$ increases along $\G_q$. Fix $\tau>0$. Given any unit-speed minimizing geodesic $\sigma:[0,b]\to M$ with $\sigma(0)=o$, $\sigma(b)=\G_q(\tau)$, and $b=d(o,\G_q(\tau))$, we have 
$$\left\langle \sigma'(b), \G_q'(\tau)\right\rangle=\left\langle \sigma'(b), \frac{-\nabla f|_{\sigma(b)}}{|\nabla f|(\sigma(b))}\right\rangle.$$
Since $\nabla f(o)=0$, it follows from the soliton equation that 
\begin{equation}
\label{eqn:distance-from-o-increases-monotonicity-along-flow-lemma}
\begin{aligned}
\left\langle \sigma'(b), (-\nabla f|_{\sigma(b)})\right\rangle
 &=-(f\circ \sigma)'(b)=-\int_0^b(f\circ \sigma)''(s)\,ds\\
&=-\int_0^b\nabla^2 _{\sigma',\sigma'}f\,ds=\int_0^b \Ric(\sigma',\sigma')\,ds>0.
\end{aligned}
\end{equation}
As a result, $\left\langle \sigma'(b), \G_q'(\tau)\right\rangle>0$. By the first variation formula for distance, it follows that $\tau\mapsto d(o,\G_q(\tau))$ is increasing. 
\end{proof}
The next lemma records the monotonicity of the intrinsic distance under the maps $\chi_s$. 
\begin{lemma}
\label{lem:variation-of-distance-in-intrinsic-metric}
For $q_1,q_2\in \Sigma_{s_0}$, the function $s\mapsto d_{\Sigma_s}(\chi_s(q_1),\chi_s(q_2))$ is non-increasing. 
\end{lemma}
\begin{proof}
For $s\leq s_0$, we set $\bar{g}_s:=\chi_s^*g|_{\Sigma_s}$ on $\Sigma_{s_0}$. Then, $d_{\Sigma_s}(\chi_s(q_1),\chi_s(q_2))=d_{\bar{g}_s}(q_1,q_2)$. Computing the variation of $\bar{g}_s$, we have 
$$\p_s \bar{g}_s=-\frac{2 \chi_s^*\left(\left.\operatorname{Ric}\right|_{\Sigma_s}\right)}{\chi_s^*|\nabla f|^2}\leq 0\qquad \text{on }\Sigma_{s_0},s\leq s_0,$$
see for instance \cite{MMS26}. By variation of distance, $\p_s^+d_{\bar{g}_s}(q_1,q_2)\leq 0$ for $s\leq s_0$. Integrating this inequality yields the lemma. 
\end{proof}
We now introduce a function $G$ that records the limiting value of $R$ along the curves $\G_q$. The next lemma shows that the scalar curvature vanishes at infinity exactly when $G$ vanishes on $\Sigma$. 
\begin{lemma}
\label{G-on-Level-sets-lemma}
Let $(M^4,g, f)$ be a gradient steady soliton satisfying \textup{\hyperref[assumption:A1]{(A1)}} and \textup{\hyperref[assumption:A2]{(A2)}}. Define $G:M\setminus \{o\}\to [0,1)$ by 
$$G(q):=\lim_{s\to \infty}R(\G_q(s))\qquad \text{for all }q\in M\setminus \{o\}.$$
\begin{enumerate}
    \item $G$ is upper semicontinuous on $M\setminus \{o\}$, i.e., if $q_k\to q$ then $$G(q)\geq \limsup_{k\to \infty}G(q_k).$$
    \item      $G\equiv 0$ on $\Sigma=f^{-1}(s_0)$ if and only if 
$$\lim_{d(x,o)\to \infty}R(x)=0.$$

\end{enumerate}
\end{lemma}
\begin{proof}
(1) The function $G$ is well-defined because for each $q\in M\setminus \{o\}$, $s\mapsto R(\G_q(s))$ is a decreasing function with values in $[0,1)$. Let $q_k\to q$ in $M\setminus\{o\}$. Passing to a subsequence if necessary, we may assume $G(q_k)\to \limsup_{j\to\infty} G(q_j)$. For each fixed $s\ge 0$, the smooth dependence of the flow of the vector field $-\nabla f/|\nabla f|$ on initial data implies
$$\Gamma_{q_k}(s)\to \Gamma_q(s), \quad\quad R(\Gamma_{q_k}(s))\to R(\Gamma_q(s)).$$
Since $R(\Gamma_{q_k}(s))\ge \lim_{t\to\infty}R(\Gamma_{q_k}(t))=G(q_k)$ by monotonicity in $t$, we obtain
$$R(\Gamma_q(s))=\lim_{k\to\infty} R(\Gamma_{q_k}(s))
\ge \lim_{k\to\infty} G(q_k)
= \limsup_{j\to\infty} G(q_j).$$
Finally, letting $s\to\infty$ and using that $R(\Gamma_q(s))$ decreases to $ G(q)$ yields
$$G(q)\ge \limsup_{j\to\infty} G(q_j),$$
as claimed. 

(2) From Lemma \ref{lem:monotonicity-along-flows}, we have $d(\G_q(s),o)\to \infty$ as $s\to \infty$, for every $q\in M\setminus \{o\}$. This shows that if $\lim_{d(x,o)\to \infty}R(x)=0$, then $G(q)=0$ for every $q\in M\setminus \{o\}$.

Assume that $G\equiv 0$ on $\Sigma=f^{-1}(s_0)$. Suppose for a contradiction that there exist $x_i\in M$ and a constant $\eta>0$ such that $d_g(x_i,o)\to \infty$ and $R(x_i)\geq \eta>0$ for all $i$. Choose $q_i\in \Sigma_{s_0}$ and $s_i>0$ such that $\G_{q_i}(s_i)=x_i$. For each $i$,  
$$f(q_i)-f(x_i)=\int_{0}^{s_i}|\nabla f|(\G_{q_i}(t))\,dt\leq s_i .$$
By Lemma \ref{lem:Linear-Growth-of-potential}, it follows that $s_i\to \infty$. Because $\Sigma_{s_0}$ is compact, we may pass to a subsequence to ensure that $q_i\to q$. Using Lemma \ref{lem:monotonicity-along-flows}, we have for all $s\in [0,s_i]$, $$R(\G_{q_i}(s))\geq R(\G_{q_i}(s_i))=R(x_i)\geq \eta.$$
Taking $i\to \infty$ for each fixed $s>0$, we obtain $R(\G_q(s))=\lim_{i\to \infty}R(\G_{q_i}(s))\geq \eta$ for any $s>0$. Taking $s\to \infty$, we obtain $G(q)\geq \eta>0$. This is a contradiction.
\end{proof}
We use the following definition from \cite{Lai25} to measure closeness between pointed manifolds: after identifying large metric balls, the metrics are close in the $C^\ell$ sense. 
\begin{definition}
\label{defn:isom-btw-manifolds}
Let $\epsilon>0$ and $\ell\geq 1$. Let $(M_i^n, g_i,x_i)$, $i=1,2$, be two pointed Riemannian manifolds. We say a smooth map $\phi: B_{g_1}\left[x_1;\epsilon^{-1}\right] \rightarrow M_2, \phi\left(x_1\right)=x_2$, is an $\epsilon$-isometry in the $C^\ell$-norm if it is a diffeomorphism onto the image, and
$$
\sup_{0\leq k\leq \ell}\sup_{\quad B_{g_1}\left[x_1; \epsilon^{-1}\right]}\left|\nabla^k\left(\phi^* g_2-g_1\right)\right| \leq \epsilon,
$$
where the covariant derivatives and norms are taken with respect to $g_1$. In this case, we also say $(M_2, g_2, x_2)$ is $\epsilon$-close to $(M_1, g_1, x_1)$ in the $C^\ell$-norm. In particular, if $\ell=\left[\epsilon^{-1}\right]$, then we say $\left(M_2, g_2, x_2\right)$ is $\epsilon$-close to $\left(M_1, g_1, x_1\right)$ and $\phi$ is an $\epsilon$-isometry. 
\end{definition}
We use the following definition from \cite{CMZ25b}. 
\begin{definition}
Let $k\geq 2,1\leq m\leq k-1$ be integers and fix $\bar{x}\in \sph^m$. Let $\bar{g}$ denote the standard round metric on $\mathbb{S}^m \times \mathbb{R}^{k-m}$ with constant scalar curvature 1. Let $\left(N^k, g\right)$ be a smooth Riemannian manifold. We say that a point $z \in N^k$ is an $(\epsilon, m)$-center in $(N^k,g)$, if $R(z)>0$ and $(N^k,R(z)g,z)$  is $\epsilon$-close to $(\sph^m\times \R^{k-m},\bar{g},(\bar{x},0))$. We call an $(\epsilon,k-1)$-center the center of an $\epsilon$-neck. 
\end{definition}
In dimension $k=4$, an $(\varepsilon,2)$-center corresponds to a region modeled by $\mathbb S^2\times\mathbb R^2$ (sometimes referred to as a \textit{bubble-sheet}), while $\varepsilon$-necks correspond to a region modeled by $\sph^3\times \R$ (sometimes referred to as a \textit{neck}). 

We shall use the following results of \cite{CMZ25b}.
\begin{lemma}
\label{summary-of-CMZ-work}
Let $(M^4,g,f)$ be a gradient steady soliton satisfying \textup{\hyperref[assumption:A1]{(A1)}} and \textup{\hyperref[assumption:A3]{(A3)}}. Let $(M^4,g_t)_{t\leq 1}$ denote the canonical flow of $(M,g)$. The following are true. 
\begin{enumerate}[(\roman{enumi})]
    \item (Perelman's derivative estimates) For $\rho>0,k,l\in \Z,k,l\geq 0$ we have 
$$|\p_t ^l\nabla^k \operatorname{Rm}_{{g_t}}(y)|\leq C({k,\rho,l,\kappa}) R_{{g}_t}(x)^{1+\frac{k}{2}+l}\qquad \text{for all }x\in M,y\in B_{g_t}[x;\rho R_{g_t}(x)^{-1/2}],t\leq 1.$$
\item (Perelman's long-range estimates) Given $A>0$, there exists $C(A,\kappa)>0$ such that 
$$R_{{g}_t}(x)\leq C(A,\kappa)R_{{g}_t}(y),$$ for any $x,y\in M$ and $t\leq 1$ that satisfies $R_{{g}_t}(y)d_{{g}_t}(x,y)^2\leq A$. In this case, $R_{{g}_t}(x)d_{{g}_t}(x,y)^2\leq A C(A,\kappa)$. 
\item (Trace Harnack inequality) For each $x\in M$, $t\mapsto R_{g_t}(x)$ is nondecreasing. 
\end{enumerate}
If we further assume that $(M^4,g,f)$ satisfies \textup{\hyperref[assumption:A4]{(A4)}}, the following properties hold.
\begin{enumerate}[(\roman{enumi})]
  \setcounter{enumi}{3}
    \item (Existence of long necks) For every $\varepsilon>0$, there exists $s_\varepsilon$ such that if $s<s_\varepsilon$, then $\Sigma_{s}$ contains an $(\varepsilon,2)$-center. 
    \item (Dimension reduction) Consider any sequence $x_i\in M$ with $d(x_i,o)\to \infty$. By passing to a subsequence, $$ \left(M^4,R(x_i) g,x_i\right) \to\left(N^{3} \times \mathbb{R}, \tilde{h}_0+d z^2,(\tilde{x}, 0)\right),$$
in the smooth Cheeger--Gromov sense, where $(N^3,\tilde{h}_0)$ is either the three-dimensional Bryant soliton $\Bry^3$, or the round cylinder $\sph^2\times \R$, such that the scalar curvature at $(\tilde{x},0)$ equals $1$. 
\end{enumerate}
\end{lemma}
We note that the trace Harnack inequality follows from \hyperref[assumption:A1]{(A1)} alone. Perelman's long-range estimates let us compare curvatures at points separated by a bounded distance in the scaled metric, and will be used repeatedly. Throughout the paper, whenever we invoke Perelman's long-range estimates under Assumptions \hyperref[assumption:A1]{(A1)}--\hyperref[assumption:A4]{(A4)}, we mean the result (ii) above. Similarly, dimension reduction under these assumptions is understood as result (v).

We now fix the following notation throughout the paper. 
\begin{definition}
\label{def:model-flows}
We fix $\tilde{g}(t)$ (also written $\tilde{g}_t$), $t\in (-\infty,\infty)$, to be an eternal Ricci flow on $\Bry^3$ induced by a soliton metric $\tilde{g}_0$ with scalar curvature $1$ at the tip. Let $\bar{g}_t=2(1-t)g_{\sph^2}+dw^2$, $t\in (-\infty,1)$, be the standard shrinking Ricci flow on $\sph^2\times \R$, where $\bar{g}_0$ has constant scalar curvature $1$. 
\end{definition}
As a consequence of dimension reduction, we have the following corollary. 
\begin{corollary}
\label{cor:close-as-metrics-to-Bry-or-Cyl}
Let $(M^4,g,f)$ be a gradient steady soliton satisfying \textup{\hyperref[assumption:A1]{(A1)}--\hyperref[assumption:A4]{(A4)}}. For every $\varepsilon>0$ there exists $D_\varepsilon>0$ such that if $d(x,o)>D_\varepsilon$, then $(M,R(x)g,x)$ is $\varepsilon$-close to either $(\Bry^3\times \R,R_{\tilde{g}_0}(\tilde{x})\tilde{g}_0+dz^2,(\tilde{x},0))$ or $(\sph^2\times \R\times\R,\bar{g}_0+dz^2,(\tilde{x},0))$. 
\end{corollary}
\begin{proof}
Fix $\varepsilon>0$. Suppose the statement were false. Then there exists a sequence $x_i\in M$ with $d(x_i,o)\to\infty$ such that $(M,R(x_i)g,x_i)$ is not $\varepsilon$-close to either of the two model spaces listed in the conclusion. By (v) of Lemma \ref{summary-of-CMZ-work}, $(M,R(x_i)g,x_i)\to (N^3\times\R,\tilde h_0+dz^2,(\tilde{x},0))$ in the smooth Cheeger--Gromov sense where $(N^3,\tilde h_0)$ is either the Bryant soliton $\Bry^3$ or the round cylinder $\sph^2\times\R$, and so that $R_{\tilde h_0}(\tilde {x})=1$. In either case, smooth pointed convergence implies that for all sufficiently large $i$, the pointed manifold $(M,R(x_i)g,x_i)$ is $\varepsilon$-close to the corresponding limit. This contradicts the choice of the sequence $x_i$. 
\end{proof}

\section{Structure of level sets}
\label{sec:Structure-of-level-sets}
Throughout this section, we assume that $\left(M^4, g, f\right)$ is a complete gradient steady Ricci soliton that satisfies \hyperref[assumption:A1]{(A1)}--\hyperref[assumption:A4]{(A4)}, and $\Sigma_s:=f^{-1}(s),\Sigma:=\Sigma_{s_0}$. We also retain the notation introduced in Section \ref{sec:Preliminaries}. 

In this section, we prove that the function $G$ (defined in Lemma \ref{G-on-Level-sets-lemma}) vanishes everywhere on the level set $\Sigma$ except possibly at two points. This follows from Theorem \ref{Two-Tips} which proves that $\Sigma_s$ contains exactly two tips (Definition \ref{def:tip}) for all $s\ll 0$. 

The following lemma is contained in \cite{CMZ25b}; for the reader's convenience, we provide a complete proof below. 
\begin{lemma}
\label{Level-Set-Convergence}
Consider any sequence $x_i\in M$ with $d(x_i,o)\to \infty$. Set $s_i:=f(x_i)$ so that $ s_i \rightarrow-\infty$. After passing to a subsequence and applying dimension reduction, we may assume that   
$$
\left(M^4,R(x_i) g,x_i\right) \to\left(N^{3} \times \mathbb{R}, g_\infty+dz^2,(\bar{x}, 0)\right).
$$
Suppose further that $R(x_i)\to 0$. Then, after passing to a subsequence, 
$$\sqrt{R(x_i)}(f-f(x_i))\to z,$$
where $z$ denotes the $\R$-coordinate in $N^3\times \R$. As a consequence, 
\begin{equation}
\label{eqn: level-set-convergence}
    (\Sigma_{s_i},R(x_i)g|_{\Sigma_{s_i}},x_i)\to (N^3,g_\infty,\bar{x}),
\end{equation}
in the sense of Cheeger--Gromov, where $\Sigma_{s_i}=f^{-1}(s_i)$ is endowed with the induced metric. 
\end{lemma}
\begin{proof}
Let $f_i:=\sqrt{R(x_i)}(f-f(x_i))$ and $g_i:=R(x_i)g$. For each $k\geq 2$,  $$\begin{aligned}
\operatorname{grad}_{g_i} f_i &= \frac{1}{\sqrt{R(x_i)}} \operatorname{grad}_{g} f, \\
|\operatorname{grad}_{g_i} f_i|_{g_i} &= |\operatorname{grad}_{g} f|_{g}\leq 1,\\
|\nabla^k f_i|_{g_i}^2&=R(x_i)^{1-k}|\nabla^{k-2} \Ric|_{g}^2.
\end{aligned}
$$
We have $f_i(x_i) = 0$, $|\nabla^{g_i} f_i|_{g_i}\leq 1$ on $M$, and using Perelman's derivative estimates (Lemma~\ref{summary-of-CMZ-work} (i)), for each $\rho>0$, and any $y\in B_{{R(x_i)g}}[x_i;\rho]$, 
$$|\nabla^k f_i(y)|_{g_i}^2\leq C_{\rho,k,\kappa} R(x_i)^{1-k}(R(x_i))^{k}=C_{\rho,k,\kappa}R(x_i).$$
After pulling back by the Cheeger--Gromov convergence maps and applying Arzel\`a--Ascoli theorem on compact subsets, it follows that there exists a smooth function $f_\infty:N^3\times \R\to \R$ such that (after passing to a subsequence) $f_i\to f_\infty$ in the smooth Cheeger--Gromov sense (see Definition \ref{defn:c-g-convergence-of-functions-general-manifolds}). Further, $f_\infty(\bar{x},0)=0$, and for each $k\geq 2$,  $$\nabla^k f_\infty\equiv 0\qquad \text{ on } N^3\times \R.$$
In particular, $\nabla f_\infty$ is a parallel vector field and $|\nabla f_\infty|$ is constant.
Since $R(x_i)\to 0$, and $R+|\nabla f|^2=1$ on $M$, we have $|\nabla f|(x_i)\to 1$, hence
$$|\nabla^{g_\infty} f_\infty|(\bar x,0)=\lim_{i\to\infty}|\nabla^{g_i} f_i|_{g_i}(x_i)=1,$$
and therefore $|\nabla f_\infty|\equiv 1$. If $N^3= \Bry^3$, then $N^3$ admits no nontrivial parallel vector fields, so any parallel unit vector field on $N^3\times\mathbb R$ equals $\pm\partial_z$, and therefore $f_\infty=\pm z$ after translation. If $N^3= \mathbb S^2\times\mathbb R_w$, then $N^3\times\mathbb{R}_z= \mathbb S^2\times\mathbb{R}^2_{w,z}$ and $\nabla f_\infty$ is a constant unit vector in the $\mathbb R^2$-factor; hence after an orthogonal linear change of coordinates in $\mathbb R^2$ we may assume $\nabla f_\infty=\partial_z$, which yields $f_\infty=z$.  Finally, each $\Sigma_{s_i}$ is a regular level set of $f_i$ and $|\nabla f_\infty|\neq 0$ on $\{f_\infty=0\}$, so we may apply Lemma \ref{Gen-Level-Set-Convergence} to conclude (\ref{eqn: level-set-convergence}). 
\end{proof}
We now consider the case in which the scalar curvature along the divergent sequence stays uniformly positive. 
\begin{lemma}
\label{Level-Set-Convergence-2}
Consider any sequence $x_i\in M$ with $d(x_i,o)\to \infty$. Set $s_i:=f(x_i)$ so that $ s_i \rightarrow-\infty$. Suppose further that $R(x_i)\to \alpha >0$. Then, after passing to a subsequence, 
$$
\left(M^4, g,x_i\right) \to\left(\Bry^{3} \times \mathbb{R},g_\infty+dz^2,(\tilde{x}, 0)\right),
$$
where $g_\infty:=\alpha^{-1}R_{\tilde{g}_0}(\tilde{x}) \tilde{g}_0$. After passing to a further subsequence, we have 
$$f-f(x_i)\to Az+h,$$
where $A\neq 0,h(\tilde{x})=0$, $|\nabla h|(\tilde{x})^2+A^2=1-\alpha$ and $h$ is a soliton potential for $\Bry^3$. Finally, letting
$$\mathcal S_h:=\{(y,z)\in\Bry^3\times\mathbb R:\ Az+h(y)=0\},$$
one has
\begin{equation}
\label{eqn: level-set-convergence-2}
(\Sigma_{s_i}, g|_{T\Sigma_{s_i}}, x_i) \to (\mathcal S_h,\hat{g}:=(g_\infty+dz^2)|_{\mathcal{S}_h},(\tilde x,0)),
\end{equation}
in the smooth Cheeger--Gromov sense.
\end{lemma}

\begin{proof}
Since $0<\inf_k R(x_k)\leq R(x_i)\leq 1$, dimension reduction implies that, after passing to a subsequence, the unrescaled sequence $(M^4,g,x_i)$ converges to $(N^3\times \R,g_\infty+dz^2,(\tilde{x},0))$. Using Perelman's derivative estimates (Lemma \ref{summary-of-CMZ-work} (i)), we have 
$$|\nabla f|^2\leq 1\qquad |\nabla^k f|=|\nabla^{k-2}\Ric|\leq C_kR^{\frac{k-2}{2}+1}\leq C_k \qquad \text{on }M\text{ for } k\geq 2.$$
Pulling back by the diffeomorphisms defining the convergence and applying Arzel\`a--Ascoli theorem yields, after passing to a subsequence, $f-f(x_i)\to f_\infty$, in the smooth Cheeger--Gromov sense, where $f_\infty\in C^\infty(N^3\times \R)$. It follows that  $\Ric_{g_\infty+dz^2}+\nabla^2 f_\infty=0$ on $N^3\times \R$.  

Suppose for a contradiction that $\left(M^4,g,x_i\right)$ dimension reduces to
$N^3=\sph^2\times \R$. Then $g_\infty=\alpha^{-1}\bar{g}_0$ with constant positive scalar curvature $\alpha$. Since $\Ric_{g_\infty}+\nabla^2 f_\infty=0$ on $\sph^2\times\R$, applying the steady soliton identity
\[
\Delta R_{g_\infty}+2|\Ric_{g_\infty}|^2
=
\langle \nabla R_{g_\infty},\nabla f_\infty\rangle,
\]
and using that $R_{g_\infty}$ is constant, we obtain $2|\Ric_{g_\infty}|^2=0$. Thus $\Ric_{g_\infty}\equiv 0$, which is impossible for the round cylinder. Therefore $N^3\neq \sph^2\times\R$, and hence $N^3=\Bry^3$.

The soliton equation on the product $\Bry^3\times \R$ gives 
$$\nabla^2 f_\infty(\partial_z,\partial_z)=0,\qquad \nabla^2 f_\infty(X,\partial_z)=0\quad\text{for all }X\in T\Bry^3,$$
so $\partial_z f_\infty$ is constant and the $\Bry^3$-gradient of $f_\infty$ is independent of $z$. Hence 
$$f_\infty(y,z)=Az+h(y),$$
for some constant $A\in \R$ and some smooth function $h$ on $\Bry^3$. Substituting back, it follows that $h$ satisfies $\Ric_{g_\infty}+\nabla^2 h=0$, i.e. $h$ is a steady soliton potential on $\Bry^3$, such that $h(\tilde{x})=f_\infty(\tilde{x},0)=0$. Finally, the identity $R+|\nabla f|^2\equiv1$ on $M^4$ gives $|\nabla f|^2(x_i)=1-R(x_i)\to1-\alpha$, and this yields $|\nabla f_\infty|^2(\tilde x,0)=1-\alpha=|\nabla h|^2(\tilde{x})+A^2$. 

We now justify that $A\neq 0$. Let $\bar{x}$ denote the tip of $(\Bry^3,g_\infty)$. Because $h$ is a soliton potential, $R(\tilde{x})+|\nabla h(\tilde{x})|^2=R(\bar{x})+|\nabla h(\bar{x})|^2$ and $\nabla h(\bar{x})=0$, hence $R(\bar{x})=1-A^2$. If $A=0$, then $R(\bar{x})=1$ which implies that there exists $y_k\in M$ such that $d(y_k,o)\to \infty$ and $R(y_k)\to 1$. This is impossible. Indeed, for each $k$, choose $w_k\in \Sigma\sub M,s_k>0$ such that $y_k=\G_{w_k}(s_k)$. Then, by Lemma \ref{lem:monotonicity-along-flows}, $s_k\to \infty$ and $1\geq R(w_k)\geq R(y_k)$ implying that $R(w_k)\to 1$. After passing to a subsequence it follows that there exists $w_0\in \Sigma$ with $w_k\to w_0$ and $R(w_0)=1$, contradicting Remark \ref{rem:R-strict-less-one}.  

Each $\Sigma_{s_i}$ is a regular level set of $f-f(x_i)$ and $|\nabla f_\infty|\neq 0$ on $\{f_\infty=0\}=\mathcal{S}_h$. From Lemma \ref{Gen-Level-Set-Convergence}, (\ref{eqn: level-set-convergence-2}) follows. 
\end{proof}

\noindent \textbf{Notation. }We write $\lambda_1 \leq \cdots \leq \lambda_4$ to denote the eigenvalues of $\Ric_g$ (viewed as a~$(1,1)$-tensor field) with respect to $g$ on $M$. When we wish to emphasize the dependence on the metric, we write $\lambda_{i,g}$.  
\begin{corollary}
\label{Vanishing-lambda-2}
If $x_k\in M$ is a sequence such that $d(x_k,o)\to \infty$ and $ \frac{\lambda_2}{R}(x_k)\to 0$, then $R(x_k)\to 0$. In particular, if the dimension reduction along $x_k$ is $\sph^2\times \R$, then $R(x_k)\to 0$. 
\end{corollary}
\begin{proof}
Arguing by contradiction, suppose the first statement fails. Passing to a subsequence, we have $R\left(x_k\right) \geq \alpha>0$ for all $k$. By Lemma  \ref{Level-Set-Convergence-2}, after passing to a subsequence, $(M^4, g, x_i)$ converges to $(\operatorname{Bry}^3 \times \mathbb{R}, \beta^{-1} R_{\tilde{g}_0}(\tilde{x}) \tilde{g}_0+{d} z^2,(\tilde{x}, 0))$, where $\beta=\lim R\left(x_k\right)$. This implies $\frac{\lambda_2}{R}\left(x_k\right) \rightarrow \frac{\lambda_2}{R}(\tilde{x}) \neq 0$, hence we obtain a contradiction. The second statement immediately follows, since if the dimension reduction along $x_k$ is $\mathbb{S}^2 \times \mathbb{R}$, then the scale-invariant quantity $\frac{\lambda_2}{R}\left(x_k\right) \rightarrow 0$ as $k \rightarrow \infty$.
\end{proof}
We define tips in $\Sigma_s$ according to \cite{MMS26}. 
\begin{definition}
\label{def:tip}
Let $s<f(o)$. We say $x\in \Sigma_s$ is a tip if $\frac{\lambda_2}{R}(x)>\frac{1}{6}$ and $\bar{\nabla}R(x)=0$, where $\bar{\nabla}$ is the connection with respect to metric induced on $\Sigma_s$. 
\end{definition}
The next lemma shows that if the scalar curvature has a positive limit along an integral curve of $-\nabla f/|\nabla f|^2$, then the level sets contain tips near that curve. 
\begin{lemma}
\label{Step-1-of-Proof}
Let $q\in \Sigma$. If $G(q)\neq 0$, then $\lim_{s\to -\infty}\frac{\lambda_2}{R}(\chi_s(q))=\frac{1}{3}.$ Further, there exists $\delta$ such that for every $0<\varepsilon<\delta$, there exists $s_\varepsilon$ such that for all $s<s_\varepsilon$, there exists $q_s\in \Sigma_s$ such that 
$$d_{\Sigma_s}(q_s,\chi_s(q))<\varepsilon,\qquad \bar{\nabla}R(q_s)=0,\qquad \frac{\lambda_2}{R}(q_s)>\frac{1}{6}.$$
\end{lemma}
\begin{proof}
Let $C>0$ such that $|\nabla f|>C^{-1}$ on $M\setminus B_g[o;1]$. Set $\alpha:=G(q)$. Let $s_i\to -\infty$ and $x_i:=\chi_{s_i}(q)$ so that $d(x_i,o)\to \infty$. Using Lemma \ref{Level-Set-Convergence-2}, after passing to a subsequence, 
\begin{equation}
\label{eqn:conv-in-pf-of-lem-step-1-of-pf}
\left(M^4, g,x_i\right) \to\left(\Bry^{3} \times \mathbb{R}, g_\infty+d z^2,(\bar{x}, 0)\right),
\end{equation}
and $f-f(x_i)\to Az+h$, in the smooth Cheeger--Gromov sense, where $g_\infty:=\alpha^{-1}R_{\tilde{g}_0}(\bar{x}) \tilde{g}_0$, $|\nabla h|(\bar{x})^2+A^2=1-\alpha$, $A\neq 0$, and $h\in C^\infty(\Bry^3)$ is a soliton potential for $(\Bry^3,g_\infty)$. 

Under \eqref{eqn:conv-in-pf-of-lem-step-1-of-pf}, we have $\Ric(\nabla f,\nabla f)\to \Ric_{g_\infty}(\nabla h,\nabla h)$. We claim that $\Ric_{g_\infty}(\nabla h,\nabla h)|_{\bar{x}}=0$. Suppose not, i.e. $c:=\Ric_{g_\infty}(\nabla h,\nabla h)|_{\bar{x}}>0$. Then, by continuity and smooth convergence, there exists $\delta_0>0$ such that $\Ric_g(\nabla f,\nabla f)|_{y}>c/2$ for any $y\in B_g[x_i;\delta_0]$ and all large $i$. For all large $i$ and $s\leq s_0$, we have $$d(\chi_s(q),\chi_{s_i}(q))\leq \int_{[\min(s,s_i),\max(s,s_i)]}\frac{1}{|\nabla f|(\chi_t(q))}\,dt\leq C|s-s_i|.$$
It follows that if $\delta_1=(2C)^{-1}\delta_0$, then $\Ric_g(\nabla f,\nabla f)|_{\chi_s(q)}>c/2$ for all $s\in [s_i-\delta_1,s_i+\delta_1]$ and all large $i$. Passing to a subsequence if necessary, we may assume that the intervals $[s_i-\delta_1,s_i+\delta_1]$ are pairwise disjoint. Using \eqref{eq:gradR}, it follows that   $$\frac{d}{ds}R(\chi_s(q))=2\frac{\Ric(\nabla f,\nabla f)|_{\chi_s(q)}}{|\nabla f|^2(\chi_s(q))},$$
for all $s\leq s_0$. As a result, $\frac{d}{ds}R(\chi_s(q))>c>0$  for all $s\in [s_i-\delta_1,s_i+\delta_1]$ and all large $i$. This contradicts the fact that $\frac{d}{ds}R(\chi_s(q))$ is integrable on $(-\infty,s_0)$ since $\int_{-\infty}^{s_0}\frac{d}{ds}R(\chi_s(q))\,ds=R(q)-G(q)$. Therefore, $\Ric_{g_\infty}(\nabla h,\nabla h)|_{\bar{x}}=0$. 

Since $\Bry^3$ has positive Ricci curvature, we have $\nabla h(\bar{x})=0$, showing that $\bar{x}$ is the tip of $\Bry^3$. The preceding arguments apply to any sequence $s_i\to -\infty$. Hence, $\frac{\lambda_j}{R}(\chi_s(q))\to \frac{1}{3}$ for $j=2,3,4$ as $s\to -\infty$. It follows that there exists $\delta>0$ such that for all $s\ll 0$, we have $\frac{\lambda_2}{R}(y)>\frac{1}{6}$ if $y\in M$ and $d_g(y,\chi_s(q))<\delta$. 

We now show the existence of a point $q_{s}$ near $\chi_{s}(q)$ where $R|_{\Sigma_s}$ attains a local maximum whenever $s\ll 0$. Suppose not, i.e. there exist $\varepsilon'\in(0,\delta)$ and a sequence $s_i\to-\infty$ such that if $x_i:=\chi_{s_i}(q)$, the open ball $B_{\Sigma_{s_i}}[x_i;\varepsilon']\sub \Sigma_{s_i}$ doesn't contain a local maximum of the function $R|_{\Sigma_{s_i}}$.

From Lemma~\ref{Level-Set-Convergence-2}, we have  
$$(\Sigma_{s_i},g|_{\Sigma_{s_i}},x_i)\to (\mathcal S_h,(g_\infty+dz^2)|_{\mathcal{S}_h},(\bar x,0)),$$
where $\mathcal S_h=\{Az+h=0\}\subset\Bry^3\times\mathbb R$, and $\bar x$ is the tip of the Bryant soliton. Since the scalar curvature on the product depends only on the Bryant factor and attains maximum at $\bar x$, the restriction $R_{g_\infty}|_{\mathcal S_h}$ has a strict local maximum at $(\bar x,0)$. Hence there exist $\eta>0$ and $\varepsilon''\in(0,\varepsilon')$ such that
$$R_{g_\infty}(\bar x,0)\ \ge\ \sup\Bigl\{R_{g_\infty}(p): p\in\mathcal S_h,\ \varepsilon''/2\leq d_{{\mathcal{S}_h}}(p,(\bar x,0))\le \varepsilon''\Bigr\}+\eta.$$
By smooth convergence of the pointed level sets and of the scalar curvature functions, for all large $i$ we obtain the corresponding inequality on $\Sigma_{s_i}$ with basepoint $x_i$. Therefore $R|_{\Sigma_{s_i}}$ achieves a local maximum at $q_i\in B_{\Sigma_{s_i}}[x_i;\varepsilon'']$, contradicting the assumption for $\varepsilon'$. This proves that given $\varepsilon \in (0,\delta)$, for all sufficiently negative $s$ there exists $q_s\in B_{\Sigma_s}[\chi_s(q);\varepsilon]$ where $R|_{\Sigma_s}$ attains a local maximum. At such a point $q_s$ one has $\bar\nabla R(q_s)=0$. Moreover,
$d_g(q_s,\chi_s(q))\le d_{\Sigma_s}(q_s,\chi_s(q))<\varepsilon<\delta$, so by the choice of $\delta$ we have $\lambda_2/R(q_s)>1/6$. 
\end{proof}
The following was proved in \cite{MMS26} under the assumption that $R\to 0$ at infinity, and we verify that the proof goes through without such an assumption. This shows that $\lambda_2/R$ can be used to detect $(\epsilon,2)$-centers in $M$ and necks in the level sets of $f$. 
\begin{lemma}
\label{Theta}
 For every $\varepsilon>0$, there exist $\theta_\varepsilon\in(0,\tfrac13)$ and $s_1=s_1(\varepsilon)\ll 0$ with the following property: if $s<s_1$ and $x\in\Sigma_s$ satisfies $\frac{\lambda_2}{R}(x)\leq \theta_\varepsilon$, then:
\begin{enumerate}[(i)]
\item $x$ is an $(\varepsilon,2)$-center in $(M^4,g)$;
\item  $x$ is the center of an  $\varepsilon$-neck in $\Sigma_s$.
\end{enumerate}
\end{lemma}
\begin{proof}
(i) We first prove that if $\varepsilon>0$ there exists $\theta>0,s_1\ll 0$ such that if $s<s_1$ and $x\in\Sigma_s$ and $\frac{\lambda_2}{R}(x)<\theta$, then $x$ is an $(\varepsilon,2)$-center in $(M^4,g)$. Assume for a contradiction that the statement does not hold. That is, there exists $\varepsilon>0$ and $x_k\in M,d(x_k,o)\to \infty$ such that $\frac{\lambda_2}{R}(x_k)\to 0$ but $x_k$ is not an $(\varepsilon,2)$-center. Then, by Corollary \ref{Vanishing-lambda-2}, it follows that $R(x_k)\to 0$ and performing dimension reduction, after passing to a subsequence, 
$$
\left(M, R(x_{k})g, x_k\right) \rightarrow (\sph^2\times \R^2,\bar{g}_0+dz^2,(x_\infty,0)),
$$
 in the smooth Cheeger--Gromov sense. As a result, $\left(M, R(x_{k})g, x_k\right)$ is arbitrarily close to $\sph^2\times\R^2$ for all large $k$. This is a contradiction.

\medskip

(ii) Next we prove the following claim. 
\begin{claim}
\label{claim:the-claim-in-lemma-theta}
Given $\varepsilon>0$, there exists $\delta\in (0,\varepsilon)$ and $s_2<0$ such that if $x\in \Sigma_s,s\leq s_2$, and $x$ is a $(\delta,2)$-center, then $x$ is the center of an $\varepsilon$-neck in $\Sigma_s$.
\end{claim}
\begin{proof}
In what follows, $C$ is a constant that may change from line to line. We may assume $|\nabla f|(x)>C^{-1}$ if $d(x,o)\geq 1$. Assume for a contradiction that the claim is false. That is, there exists $\varepsilon>0$, $\delta_i\to 0$ and $x_i\in \Sigma_{s_i}$ with $d(x_i,o)\to \infty,s_i\to-\infty$ such that $x_i$ is a $(\delta_i,2)$-center but $x_i$ is not the center of an $\varepsilon$-neck in $\Sigma_{s_i}$. Then, performing dimension reduction, after passing to a subsequence, $\left(M, R(x_{i})g, x_i\right)$ converges to $\sph^2\times \R^2$, and by Corollary \ref{Vanishing-lambda-2}, $R(x_i)\to 0$. Let $s_i:=f(x_i)\to -\infty$. By Lemma \ref{Level-Set-Convergence}, it follows that  $(\Sigma_{s_i},R(x_i)g|_{\Sigma_{s_i}},x_i)\to (\sph^2\times \R,\bar{g}_0+dz^2,(\bar{x},0))$ in the smooth Cheeger--Gromov sense. We obtain $\bar{R}_{\Sigma_{s_i}}(x_i)/R(x_i)\to 1$, where $\bar{R}_{\Sigma_{s_i}}$ denote the intrinsic scalar curvature on $\Sigma_{s_i}$. Hence, 
$$(\Sigma_{s_i},\bar{R}_{\Sigma_{s_i}}(x_i)g|_{\Sigma_{s_i}},x_i)\to (\sph^2\times \R,\bar{g}_0,\bar{x}),$$
which is a contradiction. This completes the proof of Claim \ref{claim:the-claim-in-lemma-theta}. 
\end{proof}
We now prove the lemma. Let $\varepsilon>0$. From Claim \ref{claim:the-claim-in-lemma-theta}, there exists $\delta \in (0,\varepsilon)$ and $s_2\ll 0$ (depending on $\varepsilon$) such that if $x\in \Sigma_s,s\leq s_2$, and $x$ is a $(\delta,2)$-center, then $x$ is the center of an $\varepsilon$-neck in $\Sigma_s$. From (i), there exists $\theta>0$, $s_1<s_2$ (depending on $\delta$) such that if $s<s_1$ and $x\in\Sigma_s$ and $\frac{\lambda_2}{R}(x)<\theta$, then $x$ is a $(\delta,2)$-center in $(M^4,g)$, which in turn implies that  $x$ is the center of an $\varepsilon$-neck in $\Sigma_s$. This completes the proof.  
\end{proof}
The strategy for showing that $\{x\in \Sigma: G(x)\neq 0\}$ has at most two points is to show that for all large $|s|$, $\Sigma_s$ has two tips and use Lemma \ref{Step-1-of-Proof}. To this end, we adapt the neck-cap decomposition argument of \cite{BDS21}. The proof uses three ingredients: (i) the existence of arbitrarily long necks on $\Sigma_s$ for $s\ll0$ under \hyperref[assumption:A4]{(A4)} (Lemma~\ref{summary-of-CMZ-work} (iv)), (ii) a neck detection lemma in terms of the ratio $\lambda_2/R$ (Lemma~\ref{Theta}), and (iii) Hamilton's foliation of a 3D neck by CMC (constant mean curvature) $2$-spheres \cite{Ham97}. For additional details on the construction of CMC foliations in necks, we refer the reader to \cite[Proposition 34.5]{CCG+15}. 

In \cite{MMS26}, it is proved that the level sets have exactly two tips under the assumption that $\lim_{d(x,o)\to \infty}R(x)=0$. Here, we do not require such an assumption. 
\begin{theorem}
\label{Two-Tips}
For all $s\ll 0$, $\Sigma_s$ has exactly two tips. 
\end{theorem}
\begin{proof}
In the following $\lambda_2$ denotes the second lowest eigenvalue of Ricci curvature of $(M^4,g)$ and $R$ denotes the scalar curvature of $g$.

Let $s_k\to -\infty$. By Lemma \ref{summary-of-CMZ-work} (iv), after passing to a subsequence there exist points $q_k\in \Sigma_{s_k}$ and positive numbers $\varepsilon_k\to 0$ such that $q_k$ is an $(\varepsilon_k,2)$-center in $M$ and the center of an $\varepsilon_k$-neck in $\Sigma_{s_k}$. Thus, 
$$
\left(M, R(q_{k})g, q_{k}\right) \rightarrow (\sph^2\times \R^2,\bar{g}_0+dz^2,(q_\infty,0))
$$
and
$$
\left(\Sigma_{s_k}, \bar{R}(q_k)g|_{\Sigma_{s_k}}, q_k\right) \rightarrow (\mathbb{S}^2 \times \mathbb{R},\bar{g}_0,q_\infty),
$$
in the smooth Cheeger--Gromov sense, where $\bar{R}$ is the intrinsic scalar curvature of the level sets of $f$. 

Let $\mathcal{F}_k$ denote the leaf of Hamilton's CMC foliation, which is an embedded 2-sphere in $\Sigma_{s_k}$ passing through $q_k$. By the Jordan--Brouwer separation theorem, $\Sigma_{s_k} \setminus \mathcal{F}_k$ has exactly two connected components.

Let $\varepsilon>0$ be a small number (independent of $k$) for which Lemma \ref{Theta} applies, and let $\theta=\theta_{\varepsilon}$ be the corresponding constant. Define for all large $k$, 
$$D_k:=\{x\in \Sigma_{s_k}:\ x \text{ is not the center of an }\varepsilon\text{-neck in }\Sigma_{s_k}\}.$$
Then $D_k$ is nonempty (otherwise $\Sigma_{s_k}$ would be covered by $\varepsilon$-necks and hence would be diffeomorphic to $\sph^2\times\sph^1$ which is impossible since $\Sigma_{s_k}$ is diffeomorphic to $\sph^3$). Moreover $D_k$ intersects both components of $\Sigma_{s_k}\setminus \mathcal{F}_k$. By Lemma \ref{Theta}, $\frac{\lambda_2}{R}\geq \theta$ on $D_k$ for all large $k$. Since $q_k$ is the center of an $\varepsilon_k$-neck, we have $\frac{\lambda_2}{R}(q_k)\ll 1$. For all large $k$, we may follow Hamilton's CMC foliation at each side starting at $q_k$ until we reach a point $q_{i,k}\in \Sigma_{s_k}$, and leaf $\mathcal{F}^i_k$ of Hamilton's CMC foliation passing through $q_{i,k}$ such that 
$$\frac{\lambda_2}{R}(q_{i,k})=\frac{2}{3}\theta,$$
 for $i=1,2$ and $q_{1,k},q_{2,k}$ lie in different components of $\Sigma_{s_k}\setminus \mathcal{F}_k$. In particular, $q_{i,k}$ is the center of an $\varepsilon$-neck in $\Sigma_{s_k}$. Let $N_k$ be the connected open (tubular) region of $\Sigma_{s_k}$ between $\mathcal{F}^1_k$ and $\mathcal{F}^2_k$. By construction, we have $\frac{\lambda_2}{R}\le \frac{1}{6}$ on ${\bar{N}_k}$ (after shrinking $\varepsilon$ in Lemma \ref{Theta} if necessary). 

In what follows, $i=1,2$. Passing to a subsequence, either $R(q_{i,k})\to 0$ or $R(q_{i,k})\to \alpha_i>0$. In either case, by Lemma \ref{Level-Set-Convergence} and Lemma \ref{Level-Set-Convergence-2}, we obtain pointed smooth Cheeger--Gromov convergence
$$(M,R(q_{i,k})g,q_{i,k})\to (X_i^3\times \R,g_{i,\infty}+dz^2,(q_{i,\infty},0))$$
and 
\begin{equation}
\label{eqn:convergence-to-Z3}
(\Sigma_{s_k},R(q_{i,k})g|_{\Sigma_{s_k}},q_{i,k})\ \to\ (Z_i^3,\hat g_i,(q_{i,\infty},0)),
\end{equation}
after passing to a subsequence. Moreover, since $\frac{\lambda_2}{R}(q_{i,k})=\frac{2}{3}\theta>0$, the limit $X^3_i$ cannot be the cylinder $\sph^2\times\R$, hence $X_i^3=\Bry^3$. Thus $Z_i^3$ is either $\Bry^3$ (in the case $R(q_{i,k})\to 0$) or a hypersurface $\mathcal S_h=\{Az+h=0\}\subset \Bry^3\times\R$ as in Lemma \ref{Level-Set-Convergence-2} (in the case $R(q_{i,k})\to \alpha_i>0$). In both cases, the scalar curvature of the limit $X_i^3\times \R$ restricted to $Z_i^3$ has only one critical point, which is a nondegenerate global maximum, namely the Bryant tip (in the case $Z_i^3=\Bry^3$) or the unique point lying over the Bryant tip in $\mathcal S_h$ (in the case $Z_i^3=\mathcal{S}_h$). We shall call this point $p_{i,\infty}$ and we have at $p_{i,\infty}$, ${\nabla}^{\hat{g}_i}(R_{g_{i,\infty}+dz^2}|_{Z_i})=0$,  $(\frac{\lambda_2}{R})_{g_{i,\infty}+dz^2}=\frac{1}{3}>\frac{1}{6}$, and $\nabla^{2,\hat{g}_i}(R_{g_{i,\infty}+dz^2}|_{Z_i})<-c_\infty \,\hat{g}_i$, for some $c_\infty>0$. 

For $i=1,2$, let $\Omega^i_k$ be the component of $\Sigma_{s_k}\setminus \mathcal{F}^i_k$ that does not intersect $N_k$. Then
$$\Sigma_{s_k}=\Omega^1_k\ \cup\ {\bar{N}_k}\ \cup\ \Omega^2_k,
\qquad
\Omega^1_k\cap \Omega^2_k=\emptyset,$$
and $\partial\Omega^i_k=\mathcal{F}^i_k$. Let $\vp_{i,k}$ be the embeddings into $\Sigma_{s_k}$ realizing the convergence (\ref{eqn:convergence-to-Z3}) with $\vp_{i,k}(q_{i,\infty})=q_{i,k}$ and $\bar{q}_{i,k}:=\vp_{i,k}(p_{i,\infty})$. Note that Hamilton's CMC foliation $\mathcal{F}_k^i$ has a bounded diameter with respect to the rescaled metric, i.e. there exists $A_0<\infty$ such that
$$\mathcal{F}_k^i\subset B_{R(q_{i,k})g|_{\Sigma_{s_k}}}\big[q_{i,k};A_0\big]
\qquad\text{for all large }k.$$
This implies that there exist connected open subsets $U_{i,k}\sub Z^3_i$ such that $(q_{i,\infty},0)\in \p U_{i,k}$, $ p_{i,\infty}\in U_{i,k}$, with diameters bounded uniformly in $k$, and for all large $k$,  $V_{i,k}:=\vp_{i,k}|_{U_{i,k}}(U_{i,k})\sub \Sigma_{s_k}$ is an open subset with $\p V_{i,k}=\mathcal{F}_k^i$. Hence, $V_{i,k}$ is one of the connected components of $\Sigma_{s_k}\setminus \mathcal{F}_k^i$. Since $V_{i,k}$ contains no $(\eta,2)$-centers for arbitrarily small $\eta$, we obtain $V_{i,k}=\Om^i_k$. This implies that 
$$\sup_k R(q_{i,k})\operatorname{diam}_{\Sigma_{s_k}}(\Om_k^i)^2<\infty.$$
Therefore, under the pointed convergence (\ref{eqn:convergence-to-Z3}) at $q_{i,k}$, we have (after passing to a further subsequence), 
\begin{equation}
\label{eqn:convergence-subsets-of-Z3}
\left(\Om_{k}^i, R(q_{i,k})g|_{\Om_k^i},\bar{q}_{i,k}\right) \rightarrow (U_i,\hat{g}_i,p_{i,\infty}),
\end{equation}
in the smooth Cheeger--Gromov sense, where $U_i\sub Z_i^3$ is a precompact connected open subset of $Z_i^3$ containing $p_{i,\infty}$. 

We claim that for $k$ large, each $\Omega^i_k$ contains exactly one tip. 
Choose $\rho>0$ sufficiently small such that $B_{i,\rho}:=B_{\hat g_i}[p_{i,\infty};\rho]\sub U_i$ is geodesically convex with respect to $\hat{g}_i$, and by smooth convergence, $\vp_{i,k}(B_{i,\rho})$ is geodesically convex with respect to $g|_{\Sigma_{s_k}}$, for all large $k$. 
After shrinking $\rho$ further, we may ensure that 
$R_{g_{i,\infty}+dz^2}|_{B_{i,\rho}}$ has a unique local maximum at $p_{i,\infty}$, 
$(\frac{\lambda_2}{R})_{g_{i,\infty}+dz^2}>\frac{1}{6}$ and $\nabla^{2,\hat{g}_i}(R_{g_{i,\infty}+dz^2}|_{B_{i,\rho}})<-\frac{c_\infty}{2}$ on $B_{i,\rho}$. 
By smooth convergence, for all sufficiently large $k$, 
 $R|_{\Om_k^i}$ has a local maximum inside $\vp_{i,k}(B_{i,\rho})$, 
$\bar{\nabla}^2 (R|_{\Om_k^i})<0$ and $\frac{\lambda_2}{R}>\frac{1}{6}$ on $\vp_{i,k}(B_{i,\rho})$. 
Thus, there is a point $p_{i,k}\in \vp_{i,k}(B_{i,\rho})\sub \Omega^i_k$ with $\bar\nabla R(p_{i,k})=0$ and $\frac{\lambda_2}{R}(p_{i,k})>\frac16$, i.e. $p_{i,k}$ is a tip.
Since $R|_{\Om_k^i}$ is strictly concave on $\vp_{i,k}(B_{i,\rho})$, $p_{i,k}$ is the unique such point in $\vp_{i,k}(B_{i,\rho})$. 
We claim that $p_{i,k}$ is the unique tip in $\Om_k^i$ for all large $k$. Otherwise, we may pass to a subsequence to obtain tips $p_{i,k}'\notin \vp_{i,k}(B_{i,\rho})$, and $p_{i,k}'\to p_{i,\infty}'$ under the convergence  (\ref{eqn:convergence-to-Z3}). This implies that $p_{i,\infty}'\neq p_{i,\infty}$ and $\nabla^{\hat{g}_i}(R|_{Z_i})=0$ at $p_{i,\infty}'$ which is impossible. 
Hence, $p_{i,k}$ is the unique tip in $\Omega^i_k$ for $i=1,2$, and all large $k$. 

As $\frac{\lambda_2}{R}\le \frac{1}{6}$ on ${\bar{N}_k}$, every tip of $\Sigma_{s_k}$ lies in $\Omega^1_k\cup \Omega^2_k$. Therefore $\Sigma_{s_k}$ has exactly two tips, one in each of $\Omega^1_k$ and $\Omega^2_k$. Since the choice of $s_k\to-\infty$ was arbitrary, the same holds for all $s\ll 0$. 
\end{proof}

\begin{remark}
It follows from the proof of Theorem \ref{Two-Tips} that if $p_{1,s},p_{2,s}$ are the two tips in $\Sigma_s$, $R(p_{1,s})d(p_{1,s},p_{2,s})^2\to \infty$ as $s\to -\infty$. 
\end{remark}
Combining Lemma \ref{Step-1-of-Proof} with the previous theorem, we have the following corollary. 
\begin{corollary}
\label{cor:G-neq-0-has-2-points}
$\{x\in \Sigma:G(x)\neq 0\}$ has at most two points. 
\end{corollary}
\begin{proof}
Suppose for a contradiction that there exist three distinct points $P_1,P_2,P_3$ in $\Sigma$ such that $G(P_i)\neq 0$ for $i=1,2,3$. By Lemma~\ref{lem:variation-of-distance-in-intrinsic-metric}, $d_{\Sigma_s}(\chi_s(P_i),\chi_s(P_j))\geq d_{\Sigma}(P_i,P_j)>0$ for all $s\leq s_0$. Thus, the curves $\chi_s(P_i)\in \Sigma_s$ remain uniformly separated for all $s\leq s_0$. 
Choosing $\varepsilon>0$ small enough and applying Lemma \ref{Step-1-of-Proof}, we obtain that for all $s\ll0$, $\Sigma_s$ has at least three distinct tips, one near each $\chi_s(P_i)$, $i=1,2,3$. This contradicts Theorem \ref{Two-Tips}. 
\end{proof}

\section[Stability of (\texorpdfstring{$\epsilon$}{e},2)-centers]{Stability of $ (\epsilon,2)$-centers}

\label{sec:Stability-of-e-2-centers}
We continue to assume that $\left(M^4, g, f\right)$ is a complete steady soliton satisfying \hyperref[assumption:A1]{(A1)}--\hyperref[assumption:A4]{(A4)}.  

The main result of this section, Theorem \ref{Stability-of-bubble-sheet}, shows backward-in-time stability for regions in $M$ that resemble a bubble-sheet: once a point is sufficiently close to $\sph^2\times\R^2$ and is far out on the soliton, it remains close to $\sph^2\times \R^2$ for all earlier times. More precisely, given $\varepsilon>0$, there exists $\delta>0,N>0$ such that if $x$ is a $(\delta,2)$-center and $d(x,o)>N$, then $\Phi_t(x)$ is an $(\varepsilon,2)$-center for all $t\leq 0$. This allows us to propagate geometric properties along integral curves of $-\nabla f$. As a first application, we show that the soliton $M^4$ has two ``edges'' which are integral curves of $-\nabla f$ passing through two special points $x_+,x_-$ (see Theorem \ref{existence-of-x-+-}), such that along these integral curves the manifold resembles $\Bry^3\times \R$. 

Stability properties for necks have been established in several settings; see for instance \cite{KL17,LZ22}. \cite{Lai25} proves a stability statement for three-dimensional steady solitons. The proof here uses the scale-invariant ratio $\lambda_2/R$ along the canonical flow induced by the soliton. 

We begin with the following definition from \cite{Lai25}, which allows us to compare Ricci flows: after identifying large spatial balls at time $0$, the pulled-back metrics remain uniformly close for all $t\in[-\epsilon^{-1},0]$. 
\begin{definition}
\label{defn:closeness-of-flows}
Let $\epsilon>0$. Let $\left(M_i^n, g_i(t),x_i\right)_{t\in \left[-\epsilon^{-1}, 0\right]}$ for $i=1,2$ be two pointed Ricci flows. We say a smooth map $\phi: B_{g_1(0)}\left[x_1; \epsilon^{-1}\right]\rightarrow M_2$ such that $\phi\left(x_1\right)=x_2$ is an $\epsilon$-isometry between pointed Ricci flows if it is a diffeomorphism onto the image, and 
$$
\sup_{0\leq k\leq [\epsilon^{-1}]}\,\,\sup_{B_{g_1(0)}\left[x_1;\epsilon^{-1}\right] \times\left[-\epsilon^{-1}, 0\right]}\left|\nabla^k\left(\phi^* g_2(t)-g_1(t)\right)\right| \leq \epsilon,
$$
where the covariant derivatives and norms are taken with respect to $g_1(0)$. We also say that $(M_2, g_2(t), x_2)$ is $\epsilon$-close to $\left(M_1, g_1(t), x_1\right)$ (as Ricci flows). 
\end{definition}
In the following proposition, we extend the dimension reduction to include spacetime using Lemma \ref{summary-of-CMZ-work} (v). We write $g(t)$ (or $g_t$) to denote the canonical Ricci flow induced by the steady soliton $(M^4,g)$. Also, recall our notation from Definition \ref{def:model-flows}. 
\begin{proposition}
\label{time-dimension-reduction}
Let $d(x_i,o)\to \infty$. Then, after passing to a subsequence, we have 
$$(M,(g_i(t):=R(x_i)g_{t/{R(x_i)}}),x_i)_{t\leq 0}\to (M_\infty\times \R,(g_\infty(t)+dz^2)_{t\leq 0},(x_\infty,0)),$$
where $(M_\infty\times \R,(g_\infty(t)+dz^2)_{t\leq 0},(x_\infty,0))$ is isometric (as Ricci flows) to either $$(\Bry^3\times \R,R_{\tilde{g}_0}(\tilde{x})\tilde{g}_{t/R_{\tilde{g}_0}(\tilde{x})}+dz^2,(\tilde{x},0))\text{ or }(\sph^2\times \R^2,\bar{g}_t+dz^2,(\tilde{x},0)).$$ 
For every $\varepsilon>0$ there exists $D_\varepsilon>0$ such that if $x\in M$ and $d(x,o)>D_\varepsilon$, then the rescaled Ricci flow $(M,R(x)g_{t/R(x)},x)$ is $\varepsilon$-close to either $(\Bry^3\times \R,R_{\tilde{g}_0}(\tilde{x})\tilde{g}_{t/R_{\tilde{g}_0}(\tilde{x})}+dz^2,(\tilde{x},0))$ or $(\sph^2\times \R^2,\bar{g}_t+dz^2,(\tilde{x},0))$.
\end{proposition}

\begin{proof}
By Perelman's compactness theorem and dimension reduction (adapted to weak $\kappa$-solutions \cite[Theorem~1.2]{CMZ25b}), we may pass to a subsequence to obtain 
$$\left(M,(g_i(t):=R(x_i)g_{t/R(x_i)})_{t\leq 0},x_i\right)\to \left(M_\infty\times \R,(g_\infty(t)+dz^2)_{t\leq 0},(x_\infty,0)\right),$$
in the smooth Cheeger--Gromov sense. Here, $(M_\infty,g_\infty(t))$ is a complete, nonflat, ancient solution with $R_{g_\infty(0)}(x_\infty)=1$ that is $\kappa$-noncollapsed at all scales, has $\sec\geq 0$, and satisfies the trace Harnack inequality. By Lemma \ref{summary-of-CMZ-work} (v), it follows that $(M_\infty\times \R,g_\infty(0)+dz^2,(x_\infty,0))$ is either $(\Bry^3\times \R,R_{\tilde{g}_0}(\tilde{x})\tilde{g}_0+dz^2,(\tilde{x},0))$ or $(\sph^2\times \R^2,\bar{g}_0+dz^2,(\tilde{x},0))$. As a result, we have $R_{g_\infty(0)}\leq C$ on $M_\infty$ (where $C$ depends on $\tilde{x}$). From the trace Harnack inequality, it follows that $R_{g_\infty(t)}(p)\leq R_{g_\infty(0)}(p)\leq C$ for all $t\leq 0$ and $p\in M_\infty$. Thus, the curvature of $(M_\infty,g_\infty(t))_{t\leq 0}$ is bounded. Using backward uniqueness \cite{Kot10}, it follows that $(M_\infty,g_\infty(t))$ is isometric (as Ricci flows) to either $(\Bry^3\times \R,R_{\tilde{g}_0}(\tilde{x})\tilde{g}_{t/R_{\tilde{g}_0}(\tilde{x})}+dz^2,(\tilde{x},0))$ or $(\sph^2\times \R\times \R,\bar{g}_t+dz^2,(\tilde{x},0))$. This proves the first part of the proposition.

For the $\varepsilon$-closeness statement, we argue by contradiction. If it failed for some $\varepsilon>0$, one could find a sequence $x_i\in M$ with $d(x_i,o)\to\infty$ such that the rescaled flows based at $x_i$ are not $\varepsilon$-close to either model on $[-\varepsilon^{-1},0]$. Passing to a subsequence and taking the pointed limit yields one of the model flows above. Smooth pointed convergence of Ricci flows implies that for all sufficiently large $i$, the pointed Ricci flow $(M,R(x_i)g_{t/R(x_i)},x_i)_{t\leq 0}$ is $\varepsilon$-close to the corresponding limit. This contradicts the choice of the sequence $x_i$. 
\end{proof}
The next proposition gives a dichotomy for each integral curve of $-\nabla f$ starting from $\Sigma$: either the curve enters a bubble-sheet region at some time (in the sense of an $(\epsilon,2)$-center), or else the manifold dimension reduces to $\Bry^3\times \R$ along the curve with basepoint at the tip. 
\begin{proposition}
\label{prop:either-neck-or-tip}
Let $\epsilon>0$. For every $q\in M\setminus \{o\}$, at least one of the following holds:
\begin{enumerate}
    \item There exists some $z>0$ such that $\Phi_{-z}(q)$ is an $(\epsilon,2)$-center.
    \item If $s_k\to\infty$, then after passing to a subsequence,  
    \begin{equation}
    \label{eqn:0-in-statement-of-prop-either-neck-or-tip}
    (M,R(\Phi_{-s_k}(q))g,\Phi_{-s_k}(q))
    \to
    (\Bry^3\times \R,\tilde g_0+dz^2,(\bar x,0))
    \end{equation}
    in the smooth Cheeger--Gromov sense, where $\bar x$ is the tip of the Bryant soliton.
\end{enumerate}
\end{proposition}

\begin{proof}
Let $\theta_{\epsilon}$ and $s_1(\epsilon)$ be the constants from Lemma \ref{Theta}. Suppose that for each $z>0$, $\Phi_{-z}(q)$ is not an $(\epsilon,2)$-center. Let $s_k\to \infty$ and $x_k:=\Phi_{-s_k}(q)$. Since $q\neq o$, Lemma~\ref{lem:monotonicity-along-flows} implies that $f(x_k)\to -\infty$.

By Proposition \ref{time-dimension-reduction}, after passing to a subsequence, the flows $(M,R(x_k)g_{t/R(x_k)},x_k)$ converge to either $(\sph^2\times\R^2,\ \bar g_t+dz^2,\ (\tilde x,0))$ or $(\Bry^3\times \R,R_{\tilde{g}_0}(\tilde{x})\tilde{g}_{t/R_{\tilde{g}_0}(\tilde{x})}+dz^2,(\tilde{x},0))$ (using the notation from Definition \ref{def:model-flows}). If the limit were $\sph^2\times\R^2$, then by smooth convergence $(M,R(x_k)g,x_k)$ would be $\epsilon$-close to $(\sph^2\times\R^2,\bar g_0+dz^2,(\tilde x,0))$ for all large $k$, i.e., $x_k$ would be an $(\epsilon,2)$-center, contradicting the assumption. Hence the limit is $(\Bry^3\times \R,R_{\tilde{g}_0}(\tilde{x})\tilde{g}_{t/R_{\tilde{g}_0}(\tilde{x})}+dz^2,(\tilde{x},0))$. 

Let us write $R_0:=R_{\tilde{g}_0}(\tilde{x})>0$. For a Riemannian metric $h$, let $\lambda_{1,h}\leq \lambda_{2,h}\leq \cdots$ denote the eigenvalues of $\Ric_h$ viewed as a $(1,1)$-tensor. We write $\left(\frac{\lambda_i}{R}\right)_h$ for the ratio of $\lambda_{i,h}$ to the scalar curvature $R_h$. We also use the scale invariance of $\lambda_i/R$, namely, for every $C>0$, $\left(\frac{\lambda_i}{R}\right)_h=\left(\frac{\lambda_i}{R}\right)_{Ch}.$ 

We claim that $\tilde x$ is the tip of the Bryant soliton. Suppose not. For the product metric $R_0\tilde g_s+dz^2$, we have 
\[
\left(\frac{\lambda_2}{R}\right)_{R_0\tilde g_s+dz^2}(\tilde x,0)
=
\left(\frac{\lambda_1}{R}\right)_{\tilde g_s}(\tilde x)
\qquad\text{for all }s.
\]
Since $\tilde x$ is not the tip, Remark~\ref{rem:Bryant-canonical-flow-monotonicity} implies
\[
\left(\frac{\lambda_2}{R}\right)_{\tilde g_\tau+dz^2}(\tilde x,0)\to 0
\qquad\text{as }\tau\to-\infty.
\]
Hence there exists $T_2>1$ such that
\begin{equation}
\label{eqn:1-in-proof-of-prop-either-neck-or-tip}
\left(\frac{\lambda_2}{R}\right)_{R_0\tilde g_{t/R_0}+dz^2}(\tilde x,0)
<\frac{\theta_\epsilon}{4}
\qquad\text{for all }t\le -T_2R_0.
\end{equation}
Set $t_*:=T_2R_0$. By smooth convergence of the Ricci flows at the fixed time $-t_*$,
\[
\left(\frac{\lambda_2}{R}\right)_{R(x_k)g(-t_*/R(x_k))}(x_k)
\to
\left(\frac{\lambda_2}{R}\right)_{R_0\tilde g_{-t_*/R_0}+dz^2}(\tilde x,0).
\]
Combining this with \eqref{eqn:1-in-proof-of-prop-either-neck-or-tip}, it follows that for all sufficiently large $k$,
\[
\left(\frac{\lambda_2}{R}\right)_{g(-t_*/R(x_k))}(x_k)\le \frac{\theta_\epsilon}{2}.
\]
Define $y_k:=\Phi_{-t_*/R(x_k)}(x_k)=\Phi_{-(s_k+t_*/R(x_k))}(q)=\Phi_{-z_k}(q)$ where $z_k:=s_k+\frac{t_*}{R(x_k)}>0$. Since $\Phi_{-t_*/R(x_k)}^*g=g_{-t_*/R(x_k)}$, we have
\[
\left(\frac{\lambda_2}{R}\right)_g(y_k)
=
\left(\frac{\lambda_2}{R}\right)_{g(-t_*/R(x_k))}(x_k)
\le \frac{\theta_\epsilon}{2}.
\]
Moreover, by Lemma \ref{lem:monotonicity-along-flows}, $f(y_k)\le f(x_k)$. Since $f(x_k)\to -\infty$, it follows that $f(y_k)\to -\infty$. Hence for all large $k$, the point $y_k$ lies on some level set $\Sigma_s$ with $s<s_1(\epsilon)$. Lemma \ref{Theta} therefore applies and shows that $y_k$ is an $(\epsilon,2)$-center for all sufficiently large $k$, contradicting the assumption that no point of $\{\Phi_t(q):t< 0\}$ is an $(\epsilon,2)$-center. This contradiction proves that $\tilde x$ is the tip of the Bryant soliton. In particular, $R_0=1$ and \eqref{eqn:0-in-statement-of-prop-either-neck-or-tip} holds. 
\end{proof}

In the next lemma we propagate a purely spatial bubble-sheet condition at time 0 (being a $(\delta, 2)$-center) backwards for a definite amount of time.
\begin{lemma}
\label{lem:prep-to-neck-stability}
For each $\varepsilon>0$ there exists $\delta_\varepsilon\in (0,\varepsilon)$ and $N_\varepsilon>0$ such that if $x\in M$, $d(x,o)>N_\varepsilon$, and $x$ is a $(\delta_\varepsilon,2)$-center, then there exists $T_{\varepsilon,x}>1$ such that $\Phi_t(x)$ is an $(\varepsilon,2)$-center for all $t\in [-T_{\varepsilon,x},0]$ and $\Phi_{-T_{\varepsilon,x}}(x)$ is a $(\delta_\varepsilon,2)$-center.
\end{lemma}
\begin{proof}
We use the same notation as in Definition \ref{def:model-flows}: $\tilde g_t$ denotes a fixed Bryant soliton Ricci flow on $\Bry^3$ with scalar curvature $1$ at the tip.

We will choose constants
$$\theta_\varepsilon,-s_1(\varepsilon),\delta,-s_1(\delta),\theta_\delta,T_1,N_1$$
in this order, where each constant depends only on the constants to its left. 
\begin{itemize}
    \item Let $\theta_\varepsilon$ and $s_1(\varepsilon)$ be the constants from Lemma \ref{Theta}.
    \item Choose $0<\delta<\min(1,\varepsilon)$ with the following property: if $(\mathcal{N},h)$ is a Riemannian manifold and $p\in \mathcal{N}$ is a $(\delta,2)$-center, then $$\left(\frac{\lambda_2}{R}\right)_h(p)< \min \left(\frac{1}{2}\theta_\varepsilon,\frac{1}{1000}\right).$$
    \item Let $s_1(\delta)$ and $\theta_\delta$ be the constants from Lemma \ref{Theta}. 
    \item Choose $T_1>1$ with the following property: if $z\in \Bry^3$ and $\left(\frac{\lambda_1}{R}\right)_{\tilde g_0}(z)\le \frac{1}{1000}$ then \[
    \left(\frac{\lambda_1}{R}\right)_{\tilde g_t}(z)\le \frac{1}{2}\theta_\delta
    \qquad\text{for all }t\le -T_1.
    \]
    Such a constant exists by Remark~\ref{rem:Bryant-canonical-flow-monotonicity} and the fact that $\left(\frac{\lambda_1}{R}\right)_{\tilde g_0}\to 0$ at infinity on $\Bry^3$. 
    \item Set $N_1:=\max\bigl(-s_1(\varepsilon),-s_1(\delta)\bigr).$ 
\end{itemize}
For each $x\in M$, define $T_{\varepsilon,x}:={T_1}/{R(x)}.$ Since $R(x)\le 1$, we have $T_{\varepsilon,x}\ge T_1>1$. 
\begin{claim}
\label{claim:lambda2-R-prep-to-neck-stability}
There exists $N_\varepsilon>0$ such that if $x\in M$, $d(x,o)>N_\varepsilon$, and $x$ is a $(\delta,2)$-center, then
\begin{equation}
\label{eqn:1-in-claim-lambda2-R-prep-to-neck-stability-fixed}
\left(\frac{\lambda_2}{R}\right)_{g(t)}(x)\le \theta_\varepsilon
\qquad\text{for all }t\in[-T_{\varepsilon,x},0],
\end{equation}
and
\begin{equation}
\label{eqn:2-in-claim-lambda2-R-prep-to-neck-stability-fixed}
\left(\frac{\lambda_2}{R}\right)_{g(-T_{\varepsilon,x})}(x)\le \theta_\delta.
\end{equation}
\end{claim}
\begin{proof}
We argue by contradiction. Suppose no such $N_\varepsilon$ exists. Then there is a sequence $x_j\in M$ such that $d(x_j,o)\to\infty$, $x_j$ is a $(\delta,2)$-center, but for each $j$, at least one of \eqref{eqn:1-in-claim-lambda2-R-prep-to-neck-stability-fixed} or \eqref{eqn:2-in-claim-lambda2-R-prep-to-neck-stability-fixed} fails. Since each $x_j$ is a $(\delta,2)$-center, our choice of $\delta$ ensures that
\begin{equation}
\label{eqn:initial-smallness-at-xj}
\left(\frac{\lambda_2}{R}\right)_g(x_j)\le \min \left(\frac{1}{2}\theta_\varepsilon,\frac{1}{1000}\right)
\qquad\text{for all }j.
\end{equation}
Passing to a subsequence, Proposition \ref{time-dimension-reduction} gives smooth pointed Cheeger--Gromov convergence on $[-T_1,0]$:
\[
\bigl(M,R(x_j)g(t/R(x_j)),x_j\bigr)_{t\in[-T_1,0]}
\to
\bigl(M_\infty\times\R,g_\infty(t)+dz^2,(x_\infty,0)\bigr)_{t\in[-T_1,0]},
\]
where the limit is isometric to one of the following: $(\sph^2\times\R^2,\bar g_t+dz^2,(\tilde x,0))$ or $(\Bry^3\times\R,R_0\tilde g_{t/R_0}+dz^2,(\tilde x,0))$, where $R_0:=R_{\tilde{g}_0}(\tilde{x})$. Because the pointed Ricci flows converge smoothly on the compact time interval $[-T_1,0]$ and $\inf_{[-T_1,0]}R_{{g}_\infty(t)+dz^2}(x_\infty,0)>0$, it follows that
\[
\left(\frac{\lambda_2}{R}\right)_{R(x_j)g(t/R(x_j))}(x_j)
\to
\left(\frac{\lambda_2}{R}\right)_{g_\infty(t)+dz^2}(x_\infty,0),
\]
uniformly in $t\in[-T_1,0]$. By scale invariance,
\begin{equation}
\label{eqn:uniform-convergence-lambda-ratio}
\left(\frac{\lambda_2}{R}\right)_{g(t/R(x_j))}(x_j)
\to
\left(\frac{\lambda_2}{R}\right)_{g_\infty(t)+dz^2}(x_\infty,0),
\end{equation}
uniformly in $t\in[-T_1,0]$.

\noindent \textbf{Case 1.} Suppose the limit is isometric to $(\sph^2\times\R^2,\bar g_t+dz^2,(\tilde x,0)).$ Then $\left(\frac{\lambda_2}{R}\right)_{g_\infty(t)+dz^2}(x_\infty,0)\equiv 0$ for all $t\in[-T_1,0].$ Hence by \eqref{eqn:uniform-convergence-lambda-ratio},
\[
\left(\frac{\lambda_2}{R}\right)_{g(t/R(x_j))}(x_j)\to 0
\]
uniformly in $t\in[-T_1,0]$. Therefore, for all sufficiently large $j$,
\[
\left(\frac{\lambda_2}{R}\right)_{g(t/R(x_j))}(x_j)\le \min(\theta_\varepsilon,\theta_\delta)
\qquad\text{for all }t\in[-T_1,0].
\]
Equivalently, since $T_{\varepsilon,x_j}={T_1}/{R(x_j)}$, 
\[
\left(\frac{\lambda_2}{R}\right)_{g(s)}(x_j)\le \min(\theta_\varepsilon,\theta_\delta)
\qquad\text{for all }s\in\left[-T_{\varepsilon,x_j},0\right].
\]
Thus both \eqref{eqn:1-in-claim-lambda2-R-prep-to-neck-stability-fixed} and \eqref{eqn:2-in-claim-lambda2-R-prep-to-neck-stability-fixed} hold for all large $j$, contradiction. 

\noindent \textbf{Case 2.}
Suppose the limit is isometric to $(\Bry^3\times\R,R_0\tilde g_{t/R_0}+dz^2,(\tilde x,0))$, where $R_0:=R_{\tilde{g}_0}(\tilde{x})$. Note that
\[
\left(\frac{\lambda_2}{R}\right)_{R_0\tilde g_s+dz^2}(\tilde x,0)
=
\left(\frac{\lambda_1}{R}\right)_{\tilde g_s}(\tilde x)
\qquad\text{for all }s.
\]
Taking $t=0$ in \eqref{eqn:uniform-convergence-lambda-ratio} and using \eqref{eqn:initial-smallness-at-xj}, we obtain
\[
\left(\frac{\lambda_1}{R}\right)_{\tilde g_0}(\tilde x)
=
\left(\frac{\lambda_2}{R}\right)_{R_0\tilde g_0+dz^2}(\tilde x,0)
\le \min \left(\frac{1}{2}\theta_\varepsilon,\frac{1}{1000}\right).
\]
By Remark~\ref{rem:Bryant-canonical-flow-monotonicity},
\[
\left(\frac{\lambda_2}{R}\right)_{R_0\tilde g_{t/R_0}+dz^2}(\tilde x,0)
=
\left(\frac{\lambda_1}{R}\right)_{\tilde g_{t/R_0}}(\tilde x)
\le
\left(\frac{\lambda_1}{R}\right)_{\tilde g_0}(\tilde x)
\le \frac{1}{2}\theta_\varepsilon
\]
for all $t\in[-T_1,0]$. Using \eqref{eqn:uniform-convergence-lambda-ratio} again, we conclude that for all sufficiently large $j$,
\[
\left(\frac{\lambda_2}{R}\right)_{g(t/R(x_j))}(x_j)\le \theta_\varepsilon
\qquad\text{for all }t\in[-T_1,0].
\]
Equivalently,
\[
\left(\frac{\lambda_2}{R}\right)_{g(s)}(x_j)\le \theta_\varepsilon
\qquad\text{for all }s\in\left[-T_{\varepsilon,x_j},0\right],
\]
showing \eqref{eqn:1-in-claim-lambda2-R-prep-to-neck-stability-fixed} for $x_j$. It remains to check \eqref{eqn:2-in-claim-lambda2-R-prep-to-neck-stability-fixed}. Since $\tilde g_0$ is normalized to have scalar curvature $1$ at the tip of $\Bry^3$, we have $0<R_0=R_{\tilde g_0}(\tilde x)\le 1$, hence, $-\frac{T_1}{R_0}\le -T_1.$ By the choice of $T_1$,
\[
\left(\frac{\lambda_2}{R}\right)_{R_0\tilde g_{-T_1/R_0}+dz^2}(\tilde x,0)
=
\left(\frac{\lambda_1}{R}\right)_{\tilde g_{-T_1/R_0}}(\tilde x)
\le \frac{1}{2}\theta_\delta.
\]
Evaluating \eqref{eqn:uniform-convergence-lambda-ratio} at $t=-T_1$, we obtain
\[
\left(\frac{\lambda_2}{R}\right)_{g(-T_1/R(x_j))}(x_j)\to
\left(\frac{\lambda_2}{R}\right)_{R_0\tilde g_{-T_1/R_0}+dz^2}(\tilde x,0),
\]
and therefore for all sufficiently large $j$,
\[
\left(\frac{\lambda_2}{R}\right)_{g(-T_1/R(x_j))}(x_j)\le \theta_\delta.
\]
Thus \eqref{eqn:1-in-claim-lambda2-R-prep-to-neck-stability-fixed} and \eqref{eqn:2-in-claim-lambda2-R-prep-to-neck-stability-fixed} hold for all large $j$, contradiction.
\end{proof}
By increasing $N_\varepsilon$ if necessary and using Lemma \ref{lem:Linear-Growth-of-potential}, we may also arrange that
\[
\{x\in M:d(x,o)>N_\varepsilon\}\subset \{x\in M:-f(x)>N_1\}.
\]
Let $x\in M$ satisfy $d(x,o)>N_\varepsilon$, and suppose that $x$ is a $(\delta,2)$-center. By Claim~\ref{claim:lambda2-R-prep-to-neck-stability}, 
$$\left(\frac{\lambda_2}{R}\right)_{g(t)}(x)\leq \theta_\varepsilon\qquad \text{for }t\in [-T_{\varepsilon,x},0]\qquad \text{ and }\qquad \left(\frac{\lambda_2}{R}\right)_{g(-T_{\varepsilon,x})}(x)\leq \theta_\delta,$$
where $T_{\varepsilon,x}>1$. Since $g(t)=\Phi_t^*g$, we have $\left(\frac{\lambda_2}{R}\right)_g(\Phi_t(x))=\left(\frac{\lambda_2}{R}\right)_{g(t)}(x)$ for all $t\leq 0$. Also, by Lemma~\ref{lem:monotonicity-along-flows}, $-f(\Phi_t(x))\ge -f(x)> N_1$ for all $t\leq 0$. Hence Lemma \ref{Theta} applies and shows that $\Phi_t(x)$ is an $(\varepsilon,2)$-center for all $t\in[-T_{\varepsilon,x},0]$, and that $\Phi_{-T_{\varepsilon,x}}(x)$ is a $(\delta,2)$-center. Setting $\delta_\varepsilon:=\delta$, this proves Lemma \ref{lem:prep-to-neck-stability}. 
\end{proof}
We now prove the first main result of this section. 
\begin{theorem}
\label{Stability-of-bubble-sheet}
For each $\varepsilon>0$, there exists $\delta_\varepsilon\in (0,\varepsilon)$, $N_\varepsilon>0$ such that if $x\in M$ is a $(\delta_\varepsilon,2)$-center with $d(x,o)>N_\varepsilon$, then $\Phi_{t}(x)$ is an $(\varepsilon,2)$-center for all $t\leq 0$. 
\end{theorem}
\begin{proof}
The proof is by induction. Let $\delta_\varepsilon,N_\varepsilon$ be from Lemma \ref{lem:prep-to-neck-stability}. Fix $x\in M$ with $d(x,o)>N_\varepsilon$ which is also a $(\delta_\varepsilon,2)$-center. We construct a decreasing sequence of times $\tau_j\to-\infty$ such that $\Phi_{\tau_j}(x)$ is a $(\delta_\varepsilon,2)$-center and $\Phi_t(x)$ is an $(\varepsilon,2)$-center for all $t\in[\tau_j,0]$. Set $\tau_0:=0$ and $x_0:=x$. By Lemma \ref{lem:prep-to-neck-stability} applied to $x_0$, there exists $T_{\varepsilon,x_0}>1$ such that
$$\Phi_t(x_0)\ \text{is an }(\varepsilon,2)\text{-center for }t\in[-T_{\varepsilon,x_0},0],
\qquad
\Phi_{-T_{\varepsilon,x_0}}(x_0)\ \text{is a }(\delta_\varepsilon,2)\text{-center}.$$
Define $\tau_1:=\tau_0-T_{\varepsilon,x_0}=-T_{\varepsilon,x_0}$ and $x_1:=\Phi_{\tau_1}(x)$. Inductively, suppose $\tau_j$ and $x_j=\Phi_{\tau_j}(x)$ have been defined and $x_j$ is a $(\delta_\varepsilon,2)$-center. Since $\tau_j\le 0$, by Lemma \ref{lem:monotonicity-along-flows}, it follows that $d(x_j,o)\ge d(x,o)>N_\varepsilon$. Lemma \ref{lem:prep-to-neck-stability} applies to $x_j$ and yields $T_{\varepsilon,x_j}>1$ such that $\Phi_{-T_{\varepsilon,x_j}}(x_j)$ is a $(\delta_\varepsilon,2)$-center and $\Phi_t(x_j)$ is an $(\varepsilon,2)$-center for $t\in[-T_{\varepsilon,x_j},0]$. Set $\tau_{j+1}:=\tau_j-T_{\varepsilon,x_j}$ and $x_{j+1}:=\Phi_{\tau_{j+1}}(x)$. Since each increment satisfies $-T_{\varepsilon,x_j}<-1$, we have $\tau_j\le -j$ and therefore $\tau_j\to-\infty$. Moreover, for each $j$ and each $s\in[\tau_{j+1},\tau_j]$, we have $s-\tau_j\in[-T_{\varepsilon,x_j},0].$ Since $x_j=\Phi_{\tau_j}(x)$ and $\Phi_s(x)=\Phi_{\,s-\tau_j}(x_j)$, it follows that $\Phi_s(x)$ is an $(\varepsilon,2)$-center. Therefore, $\bigcup_{j=0}^\infty [\tau_{j+1},\tau_j]=(-\infty,0]$ consists entirely of times $t$ such that $\Phi_t(x)$ is an $(\varepsilon,2)$-center. This proves the theorem.
\end{proof}

\noindent \textbf{Notation. }From Theorem \ref{Two-Tips}, if $s_0\ll 0$ (equivalently, $-s_0$ large) then for all $s\le s_0$, $\Sigma_s$ has exactly two tips which we shall denote by $p_{1,s}$ and $p_{2,s}$ from now on. 

We now show that after taking $-s_0$ larger, there exist two points $x_+,x_-\in \Sigma$ such that, after rescaling, $\chi_s(x_\pm)$ remain close to the two tips, and the manifold resembles $\Bry^3\times \R$ near them. Moreover, we show that $G$ can be nonzero on $\Sigma$ only at these two points. Finally, if $M$ is sufficiently close to $\sph^2\times \R^2$ at $q$, then the manifold dimension reduces to $\sph^2\times \R$ along the integral curve of $-\nabla f$ starting at $q$. 
\begin{theorem}
\label{existence-of-x-+-}
There exists $s_0\ll 0$ and two distinct points $x_+,x_-\in \Sigma:=\Sigma_{s_0}$ such that the following holds. 
\begin{enumerate}
    \item Given any $t_k\to -\infty$, we may pass to a subsequence to ensure $(M,R(\chi_{t_k}(x_+))g,\chi_{t_k}(x_+))$ converges to $(\Bry^3\times \R,\tilde{g}_0+dz^2,(\bar{x},0))$ as $k\to \infty$; and $(M,R(\chi_{t_k}(x_-))g,\chi_{t_k}(x_-))$ converges to $(\Bry^3\times \R,\tilde{g}_0+dz^2,(\bar{x},0))$ as $k\to \infty$, where $\bar{x}$ is the tip of the Bryant soliton.
    \item Given any $t_k\to-\infty$, we may pass to a subsequence to ensure
$$
R(\chi_{t_k}(x_+))\,d(p_{1,t_k},\chi_{t_k}(x_+))^2\to 0
\quad\text{and}\quad
R(\chi_{t_k}(x_-))\,d(p_{2,t_k},\chi_{t_k}(x_-))^2\to 0.
$$
    \item $G\equiv 0$ on $\Sigma\setminus \{x_+,x_-\}$.   
    \item   There exists $\epsilon_*>0,N_*>0$ such that if $q\in M$ is an $(\epsilon_*,2)$-center with $d(q,o)>N_*$, then for any $t_k\to -\infty$, we may pass to a subsequence to ensure  $(M,R(\Phi_{t_k}(q))g,\Phi_{t_k}(q))$ converges to $(\sph^2\times \R^2,\bar{g}_0+dz^2,(\tilde{x},0))$ as $k\to \infty$.
\end{enumerate}
\end{theorem}
\begin{proof}
Fix $\epsilon>0$ such that if $(\mathcal{N},h)$ is a Riemannian manifold and $p\in \mathcal{N}$ is an $(\epsilon,2)$-center, then $\left(\frac{\lambda_2}{R}\right)_h(p)< \frac{1}{7}$. 
Let $\delta_\epsilon,N_\epsilon$ be the constants given by Theorem \ref{Stability-of-bubble-sheet}. 
Taking $s_0\ll 0$, we may ensure $\Sigma\sub \{x\in M:d(x,o)>N_\epsilon\}$. Let $x_{i,s}:=\chi_{s}^{-1}(p_{i,s})\in \Sigma$ for $i=1,2$. Let $s_k\to -\infty$. Let $x_+$ be a limit point of $x_{1,s_k}$ and let $x_-$ be a limit point of $x_{2,s_k}$. After passing to a subsequence, we may ensure that $x_{1,s_k}\to x_+,x_{2,s_k}\to x_-$. 

From Proposition \ref{prop:either-neck-or-tip} applied to $\delta_\epsilon/2$, either (i) there exists some $s\ll 0$ such that $\chi_s(x_+)$ is a $(\delta_\epsilon/2,2)$-center or (ii) given any $t_k\to -\infty$, passing to a subsequence, $(M,R(\chi_{t_k}(x_+))g,\chi_{t_k}(x_+))$ converges to $(\Bry^3\times \R,\tilde{g}_0+dz^2,(\bar{x},0))$ where $\bar{x}$ is the tip of $\Bry^3$. We claim that the first case cannot occur. Indeed, if $\chi_{s_1}(x_+)$ is a $(\delta_\epsilon/2,2)$-center for some $s_1\leq s_0$, then as $\chi_{s_1}(x_{1,s_k})\to \chi_{s_1}(x_+)$, it follows that $\chi_{s_1}(x_{1,s_k})$ is a $(\delta_\epsilon,2)$-center for all large $k$. Using Theorem \ref{Stability-of-bubble-sheet}, it follows that $\chi_s(x_{1,s_k})$ is an $(\epsilon,2)$-center for all large $k$ and $s\leq s_1$. In particular, $p_{1,s_k}$ is an $(\epsilon,2)$-center in $M$ and a tip in $\Sigma_{s_k}$. This is a contradiction: by our choice of $\epsilon$, any $(\epsilon,2)$-center $y$ satisfies $\frac{\lambda_2}{R}(y)<\frac{1}{7}$ while each tip $p_{1,s_k}$ satisfies $\frac{\lambda_2}{R}(p_{1,s_k})>\frac{1}{6}$. This proves (1) for $x_+$ and a similar proof holds for $x_-$. 

In order to prove (2), we first claim the following. 
\begin{claim}
\label{claim:nearness-of-tip-edge}
Let $t_k\to -\infty$. For at least one of $i=1,2$, after passing to a subsequence, we may ensure that $R(\chi_{t_k}(x_+))d(p_{i,t_k},\chi_{t_k}(x_+))^2\to 0$. 
\end{claim}
\begin{proof}
From the previous discussion, after passing to a subsequence we have the convergence of  $(M,R(\chi_{t_k}(x_+))g,\chi_{t_k}(x_+))$ to $(\Bry^3\times \R,\tilde{g}_0+dz^2,(\bar{x},0))$ where $\bar{x}$ is the tip of $\Bry^3$. By using Lemma \ref{Level-Set-Convergence} and Lemma \ref{Level-Set-Convergence-2}, we may pass to the level sets: $(\Sigma_{t_k},R(\chi_{t_k}(x_+))g,\chi_{t_k}(x_+))$ converges to the corresponding limit (which is either $\Bry^3$ in the case $R(\chi_{t_k}(x_+))\to 0$ or the hypersurface $\mathcal{S}_h$ in the case $R(\chi_{t_k}(x_+))\to \alpha>0$). In the limit, the point $(\bar x,0)$ is the unique isolated critical point (local maxima) of the scalar curvature restricted to the corresponding level set and $(\frac{\lambda_2}{R})_{\tilde{g}_0+dz^2}=\frac{1}{3}$ at $(\bar{x},0)$. Fix $\beta>0$ small. By smooth convergence, for all sufficiently large $k$, the intrinsic ball $B_{\beta,k}:=B_{R(\chi_{t_k}(x_+))g|_{\Sigma_{t_k}}}[\chi_{t_k}(x_+);\beta]$ is close to the corresponding ball in the limit. This shows that for all sufficiently large $k$ the scalar curvature function $R|_{\Sigma_{t_k}}$ has a local maximum in $B_{\beta,k}$, hence a critical point inside $B_{\beta,k}$. Further, if $\beta$ is small, every point in $B_{\beta,k}$ satisfies $(\frac{\lambda_2}{R})_g>\frac{1}{6}$. By Theorem \ref{Two-Tips}, the only critical points of $R|_{\Sigma_{t_k}}$ satisfying $(\frac{\lambda_2}{R})_g>\frac{1}{6}$ are the two tips $p_{1,t_k}$ and $p_{2,t_k}$. Hence for all large $k$, at least one of $p_{1,t_k}$ or $p_{2,t_k}$ lies in $B_{\beta,k}$. In particular, $\min_{i\in\{1,2\}} R(\chi_{t_k}(x_+))\,d\big(p_{i,t_k},\chi_{t_k}(x_+)\big)^2 \leq \beta^2$. Let $\beta_j\to 0$ and choose a subsequence along which the minimizing index is constant, say $i\in\{1,2\}$. Then along this subsequence, $R(\chi_{t_k}(x_+))\,d\big(p_{i,t_k},\chi_{t_k}(x_+)\big)^2 \to 0,$ which proves Claim \ref{claim:nearness-of-tip-edge}. 
\end{proof}
\begin{remark}
\label{rmk:existence-x-pm-x-+-near-p-1-s}
From Claim \ref{claim:nearness-of-tip-edge}, we may assume $R(\chi_{s_k}(x_+))d(p_{1,s_k},\chi_{s_k}(x_+))^2\to 0$ without loss of generality. 
\end{remark}
In order to show a similar statement for $x_-$, we first establish that $x_+$ and $x_-$ are distinct. 
\begin{claim}
\label{claim:distinct-egdes}
 $x_+\neq x_-$.
\end{claim}
\begin{proof}
By Theorem \ref{Stability-of-bubble-sheet} applied with parameter $\delta_\epsilon$, there exists $\hat\delta\in(0,\delta_\epsilon)$ such that if $p$ is a $(\hat\delta,2)$-center and $f(p)\ll 0$ then $\Phi_{-s}(p)$ is a $(\delta_\epsilon,2)$-center for all $s\geq 0$. Fix $q\in\Sigma$ such that $q$ is a $(\hat{\delta},2)$-center (for $s_0\ll 0$, such points exist by Lemma~\ref{summary-of-CMZ-work}) so that $\chi_s(q)$ is a $(\delta_\epsilon,2)$-center for all $s\le s_0$. 

Suppose for a contradiction that $x_+=x_-$. For each $s\leq s_0$, let $N_s\sub \Sigma_s$ be the leaf of Hamilton's CMC foliation passing through $\chi_s(q)\in \Sigma_s$, which is an embedded $2$-sphere. Since $N_s$ lies inside a $\delta_\epsilon$-neck region every point of $N_s$ is a $(C\delta_\epsilon,2)$-center where $C$ is absolute. After shrinking $\delta_\epsilon$ (and relabeling), we may assume every point of $N_s$ is a $(\delta_\epsilon,2)$-center. 

By the Jordan--Brouwer separation theorem, $\Sigma_s\setminus N_s$ has exactly two components for each $s\leq s_0$. Fix a large $k$. 
Since $x_+=x_-$, the points $x_{+}$ and $x_{2,s_k}$ lie in the same connected component of $\Sigma_{s_0}\setminus N_{s_0}$ for all large $k$. 
By Remark \ref{rmk:existence-x-pm-x-+-near-p-1-s}, $\chi_{s_k}(x_+)$ and $p_{1,s_k}$ are in the same component of  $\Sigma_{s_k}\setminus N_{s_k}$ while from the proof of Theorem \ref{Two-Tips}, $p_{1,s_k}$ and $p_{2,s_k}$ are in different components of  $\Sigma_{s_k}\setminus N_{s_k}$.
As a result, $\chi_{s_k}(x_+)$ and $\chi_{s_k}(x_{2,s_k})=p_{2,s_k}$ are in different components of  $\Sigma_{s_k}\setminus N_{s_k}$.

Since $s\mapsto N_s$ depends smoothly on $s$ (Hamilton CMC foliation) and $s\mapsto \chi_s(x_{2,s_k}),s\in [s_k,s_0]$ is continuous, the component of $\Sigma_s\setminus N_s$ containing $\chi_s(x_{2,s_k})$ can change only when $\chi_s(x_{2,s_k})$ intersects $N_s$. Thus, there exists $z_k$ such that $s_k\leq z_k\leq s_0$ and $\chi_{z_k}(x_{2,s_k})\in N_{z_k}$. Since $\chi_{s_k}(x_{2,s_k})=p_{2,s_k}$ and $s_k\le z_k$, the point $p_{2,s_k}$ is obtained from $\chi_{z_k}(x_{2,s_k})$ by flowing further backward. Therefore, Theorem \ref{Stability-of-bubble-sheet} applied to the $(\delta_\epsilon,2)$-center $\chi_{z_k}(x_{2,s_k})$ implies that $p_{2,s_k}$ is an $(\epsilon,2)$-center, which, as in the proof of (1), is impossible for our choice of $\epsilon$. 
This proves Claim \ref{claim:distinct-egdes}.
\end{proof}
Given $t_k\to -\infty$, we can argue exactly as in the proof of Claim \ref{claim:nearness-of-tip-edge} with $x_+$ replaced by $x_-$. Using the fact that $R(p_{1,s})d(p_{1,s},p_{2,s})^2\to \infty$ as $s\to -\infty$, passing to a subsequence yields $R(\chi_{t_k}(x_-))d(p_{2,t_k},\chi_{t_k}(x_-))^2\to 0$. Together with Claim \ref{claim:distinct-egdes}, this proves parts (1) and (2) of Theorem \ref{existence-of-x-+-}. 

We now show (3). Assume, for the sake of contradiction, that $G(q)\neq 0$ for some
$q\in \Sigma\setminus \{x_+,x_-\}$. 
From Lemma \ref{Step-1-of-Proof} and Theorem \ref{Two-Tips}, it follows that for at least one of $i=1,2$, and $s_k\to -\infty$, we have $d(\chi_{s_k}(q),p_{i,s_k})\to 0.$ Without loss of generality, let $i=1$. Then, by the triangle inequality and part (2), we have $R(\chi_{s_k}(x_+))\,d(\chi_{s_k}(q),\chi_{s_k}(x_+))^2\to 0.$ Using Perelman's long-range estimates (Lemma \ref{summary-of-CMZ-work} (ii)) to compare scalar curvatures, we obtain
\[
C^{-1}R(\chi_{s_k}(q))\le R(\chi_{s_k}(x_+))\le C\,R(\chi_{s_k}(q)),
\]
where $C$ is independent of $k$. Hence $R(\chi_{s_k}(q))\,d(\chi_{s_k}(q),\chi_{s_k}(x_+))^2\to 0.$ Since $G(q)\neq 0$, we obtain $R(\chi_{s_k}(q))\to G(q)>0$, and therefore $d(\chi_{s_k}(q),\chi_{s_k}(x_+))\to 0.$ 
After passing to a subsequence, $R(\chi_{s_k}(x_+))\to \alpha>0.$ Applying Lemma~\ref{Level-Set-Convergence-2} with basepoints $\chi_{s_k}(x_+)$, we may pass to the level sets. Since $d(\chi_{s_k}(q),\chi_{s_k}(x_+))\to 0$, the two points converge to the same point in the limit, and hence $d_{\Sigma_{s_k}}(\chi_{s_k}(q),\chi_{s_k}(x_+))\to 0.$ On the other hand, by Lemma~\ref{lem:variation-of-distance-in-intrinsic-metric}, the function $s\mapsto d_{\Sigma_s}(\chi_s(q),\chi_s(x_+))$ is non-increasing. Since $q\neq x_+$, it follows that for every large $k$, $d_{\Sigma_{s_k}}(\chi_{s_k}(q),\chi_{s_k}(x_+))
\ge d_{\Sigma_{s_0}}(q,x_+)>0,$ a contradiction. Hence $G\equiv 0$ on $\Sigma\setminus \{x_+,x_-\}$.

(4) Let $q\in M$ be a $(\delta_\epsilon,2)$-center with $d(q,o)>N_\epsilon$ and let $t_k\to -\infty$. From Theorem \ref{Stability-of-bubble-sheet}, $\Phi_s(q)$ is an $(\epsilon,2)$-center for all $s\leq 0$. We claim that after passing to a subsequence, $(M,R(\Phi_{t_k}(q))g,\Phi_{t_k}(q))$ converges to $\sph^2\times \R^2$. If not, by Lemma \ref{summary-of-CMZ-work} (v), we may pass to a subsequence to ensure that $(M,R(\Phi_{t_k}(q))g,\Phi_{t_k}(q))$ converges to $(\Bry^3\times \R,R_{\tilde{g}_0}(\tilde{x})\tilde{g}_0+dz^2,(\tilde{x},0))$. 
By our choice of $\epsilon$, $\left(\frac{\lambda_2}{R}\right)_g<\frac{1}{7}$ at $\Phi_{t_k}(q)$ for all $k$, which implies that $\frac{\lambda_2}{R}\leq \frac{1}{7}$ at $(\tilde{x},0)$. 
In particular, $\tilde{x}$ is not the tip of $\Bry^3$. 
Applying Proposition \ref{prop:either-neck-or-tip}, it follows that for each $\eta>0$, there exists $z_\eta$ such that $\Phi_{-z_\eta}(q)$ is a $(\delta_\eta,2)$-center and by Theorem \ref{Stability-of-bubble-sheet}, we obtain that for each $\eta$, $\Phi_{t_k}(q)$ is an $(\eta,2)$-center for all large $k$. As a result, we obtain that $(\tilde{x},0)$ is an $(\eta,2)$-center for all $\eta>0$, which is impossible. 
Setting $\epsilon_*:=\delta_\epsilon$ and $N_*:=N_\epsilon$ completes the proof of part (4). 
\end{proof}
We now define the integral curves through $x_+$ and $x_-$ to be the edges of the soliton. 
Similar integral curves were defined for $3$D flying wings in \cite[Lemma 3.17]{Lai25}. 
\begin{definition}
\label{defn:edges-of-soliton}
For $i=1,2$, let $\Gamma_i:[0,\infty)\to M$ be continuous curves such that
$\Gamma_i(0)=o$, $\Gamma_i((0,\infty))$ is an integral curve of $-\nabla f/|\nabla f|$,
$x_+\in \Gamma_1((0,\infty))$, and $x_-\in \Gamma_2((0,\infty))$. 
We set
\[
\Gamma=\Gamma_1([0,\infty))\cup \Gamma_2([0,\infty)),
\]
and call $\Gamma_1$ and $\Gamma_2$ the edges of the soliton.
\end{definition}
Together with this definition, Theorem \ref{existence-of-x-+-} provides the following picture: the manifold resembles $\Bry^3\times \R$ near the edges, and $\sph^2\times \R^2$ far away from them (see Figure 1).

\section[Scalar curvature along integral curves of -\texorpdfstring{$\nabla f$}{Grad f}]{Scalar curvature along integral curves of $-\nabla f$}

\label{sec:Scalar-curvature-along-integral-curves}
Recall that $\left(M^4, g, f\right)$ satisfies \hyperref[assumption:A1]{(A1)}--\hyperref[assumption:A4]{(A4)}. The main result of this section is Theorem \ref{Upper-bound-on-curves}, which gives a quantitative decay estimate for the scalar curvature along integral curves of $-\nabla f$: once the manifold is sufficiently close to $\sph^2\times\R^2$ at a point $q$ that is at some large distance from $o$, the scalar curvature along the integral curve $\Phi_{-t}(q)$ is bounded above by $Ct^{-1}$ for all $t>0$. Moreover, if dimension reduction along $\Phi_{-t}(q)$ yields $\sph^2\times\R^2$, then the product $t\,R(\Phi_{-t}(q))$ converges to $1$. As a consequence, we show that the volume growth of $M$ is strictly faster than that of $\Bry^4$ (Corollary \ref{Cor:volume-growth}). 

We begin with a preliminary lemma showing that at points where the manifold is sufficiently close to $\sph^2\times \R^2$, the scalar curvature must be small. 
\begin{lemma}\label{small-curv-at-bubble-sheet}
Given $\eta>0$, there exists $\epsilon>0$ such that if $x\in M$ is an $(\epsilon,2)$-center, then $R(x)\le \eta$.
\end{lemma}

\begin{proof}
Suppose not. Then there exist $\eta>0$ and points $x_j\in M$ with $\epsilon_j\to 0$ such that each $x_j$ is an $(\epsilon_j,2)$-center and $R(x_j)\geq \eta$. Since $(M,R(x_j)g,x_j)$ is $\epsilon_j$-close to $(\sph^2\times\R^2,\bar g,(\bar x,0))$ and $\epsilon_j\to 0$, we obtain $(\frac{\lambda_2}{R})(x_j)\to 0$. In particular, $x_j$ cannot remain in a fixed compact set, showing that $d(x_j,o)\to\infty$. From Corollary \ref{Vanishing-lambda-2}, we obtain $R(x_j)\to 0$, which is a contradiction. 
\end{proof}

We now prove the main result of this section. The proof is inspired by \cite{BCDMZ22}. 

\begin{theorem}
\label{Upper-bound-on-curves}
There exist $\varepsilon_1>0,r_1>0$, and $C>0$ such that the following holds. If $q\in M$ is an $(\varepsilon_1,2)$-center such that $d(q,o)>r_1$, then for all $t>0$, 
$$R(\Phi_{-t}(q))\leq \frac{C}{t}.$$
Moreover, if $t_k\to \infty$ and 
$$(M,R(\Phi_{-t_k}(q))g,\Phi_{-t_k}(q))\to (\sph^2\times \R^2,\bar{g}_0+dz^2,(\bar{q},0)),$$
then 
\begin{equation}
\label{eqn:linear-scalar-bound-100}
\lim_{s\to \infty}sR(\Phi_{-s}(q))=1.
\end{equation}

\end{theorem}

\begin{proof}
Fix $r_0>0, C>0$ such that if $d(p,o)>r_0$ then $|\nabla f|^2(p)\geq C^{-1}$ (for example, using Lemma \ref{lem:monotonicity-along-flows}, one may take $C^{-1}=\inf_{f\leq f(o)-10} |\nabla f|^2=\min_{f=f(o)-10} |\nabla f|^2$ and choose $r_0$ using Lemma \ref{lem:Linear-Growth-of-potential}). Let $\eta>0$ be any positive number such that $\eta\leq C^{-1}/8$. 

Let $\epsilon>0$ be a constant, depending on $\eta$, to be chosen shortly. Let $\varepsilon_1:=\delta_\epsilon$ and $N_\epsilon>0$ be the constants from Theorem \ref{Stability-of-bubble-sheet}. Set $r_1:=\max(r_0,N_\epsilon)$. Let $q$ be an $(\varepsilon_1,2)$-center with $d(q,o)>r_1$. Then $\Phi_{-t}(q)$ is an $(\epsilon,2)$-center for all $t\geq 0$. Using Lemma \ref{small-curv-at-bubble-sheet}, we may choose $\epsilon$ small (depending on $\eta$, but independent of $q$) such that for all $t\geq 0$, we have 
$$R(\Phi_{-t}(q))\leq \eta,$$ and $$\frac{|\Delta R|}{R^2}(\Phi_{-t}(q))\leq \eta, \qquad \left|\frac{|\Ric|^2}{R^2}(\Phi_{-t}(q))-\frac{1}{2}\right|\leq \eta.$$
Note that on $\sph^2\times \R^2$ with constant scalar curvature 1, we have $R^{-2}\Delta R\equiv 0$ and $2|\Ric|^2R^{-2}\equiv 1$. 

We obtain the following identity. In \eqref{eqn:linear-scalar-bound-main}, all integrands are evaluated at $\Phi_{-s}(q)$. For each $t>0$,  
\begin{equation}
\label{eqn:linear-scalar-bound-main}
\begin{aligned}
&\left[\frac{1}{R(\Phi_{-t}(q))}+f(\Phi_{-t}(q))\right]-\left[\frac{1}{R(q)}+f(q)\right]\\
&=\int_{0}^t\left\langle -\nabla f,\nabla\left(\frac{1}{R}+f\right)\right\rangle \,ds\\
&=\int_{0}^t-|\nabla f|^2+\frac{1}{R^2}\langle {\nabla f},\nabla R\rangle \,ds\\
&=\int_{0}^t-|\nabla f|^2+\frac{\Delta R+2|\Ric|^2}{R^2} \,ds\\
&=\int_{0}^t\left[R-1+\frac{\Delta R+2|\Ric|^2}{R^2} \right]\,ds\\
&=\int_{0}^t\left[2\left(\frac{|\Ric|^2}{R^2}-\frac{1}{2}\right)+R+\frac{\Delta R}{R^2}\right ]\,ds,
\end{aligned}
\end{equation}
where we used the soliton identities \eqref{eq:scalarPDE} and \eqref{eq:normalization} in the third and fourth equalities, respectively. By the choice of $\epsilon$ and the scale-invariant bounds along $\Phi_{-s}(q)$, we have for all $s\ge 0$,
$$\left|2\Big(\frac{|\Ric|^2}{R^2}-\frac12\Big)+R+\frac{\Delta R}{R^2}\right|(\Phi_{-s}(q))\le 4\eta.$$
Therefore, (\ref{eqn:linear-scalar-bound-main}) implies that for all $t>0$, 
\begin{equation}
\label{eqn:linear-scalar-bound-4}
\left|\Big(\frac{1}{R(\Phi_{-t}(q))}+f(\Phi_{-t}(q))\Big)-\Big(\frac{1}{R(q)}+f(q)\Big)\right|
\leq \int_{0}^t4\eta\,ds,
\end{equation}
in particular, 
$$\begin{aligned}
&\left[\frac{1}{R(q)}+f(q)\right]-\left[\frac{1}{R(\Phi_{-t}(q))}+f(\Phi_{-t}(q))\right]\leq 4\eta t.
\end{aligned}$$
Rearranging this, we obtain for all $t>0$, 
\begin{equation}
\label{eqn:linear-scalar-bound-2}
 \frac{1}{tR(q)}+\frac{f(q)-f(\Phi_{-t}(q))}{t}-4\eta\leq \frac{1}{tR(\Phi_{-t}(q))}.
\end{equation}
Since $d(o,\Phi_{-t}(q))\geq d(o,q)\geq r_1$, we have $C^{-1}\leq |\nabla f|^2(\Phi_{-t}(q))\leq 1$ for all $t\geq 0$. It follows that 
$$C^{-1}t\leq \int_0^t|\nabla f|^2(\Phi_{-s}(q))\,ds
=f(q)-f(\Phi_{-t}(q))$$
Combining this with \eqref{eqn:linear-scalar-bound-2}, and using the choice of $\eta$, we obtain
$$(2C)^{-1}\leq C^{-1}-4\eta\leq \frac{1}{tR(\Phi_{-t}(q))}.$$
We conclude that the estimate $t\,R(\Phi_{-t}(q))\le 2C$ holds for all $t>0$. This proves the first part of the theorem. Choose, for instance, $\eta=C^{-1}/16$. 

Suppose that $(M,R(\Phi_{-t_k}(q))g,\Phi_{-t_k}(q))\to (\sph^2\times \R^2,\bar{g}_0+dz^2,(\bar{q},0))$ for some $t_k\to \infty$. From Corollary \ref{Vanishing-lambda-2}, we obtain that $R(\Phi_{-t_k}(q))\to 0$, hence $G(q)=0$. Let $s_k\to \infty$ and let $\eta \in (0,C^{-1}/8)$. 
Let $\varepsilon_1,r_1$ be the corresponding constants from the previous part of the proof. 
For all sufficiently large $k$, the point $\Phi_{-t_k}(q)$ is an $(\varepsilon_1,2)$-center. Hence there exists $t_\eta>0$ such that $q_\eta:=\Phi_{-t_\eta}(q)$ is an $(\varepsilon_1,2)$-center with $d(q_\eta,o)>r_1$. 
Applying the previous argument to $q_\eta$, we obtain the analogue of \eqref{eqn:linear-scalar-bound-4}: for every $t>0$,
\begin{equation}
\label{eqn:linear-scalar-bound-10}
\left|\Big(\frac{1}{R(\Phi_{-t}(q_\eta))}+f(\Phi_{-t}(q_\eta))\Big)-\Big(\frac{1}{R(q_\eta)}+f(q_\eta)\Big)\right|
\leq 4\eta t.
\end{equation}
Substituting $t=s_k-t_\eta$ and using $\Phi_{-(s_k-t_\eta)}(q_\eta)=\Phi_{-s_k}(q)$, we obtain for all large $k$, 
$$\left|\frac{1}{R(\Phi_{-s_k}(q))}+f(\Phi_{-s_k}(q))-c_q\right|\leq 4\eta s_k,$$
where $c_q:=\frac{1}{R(q_\eta)}+f(q_\eta)$. Since $c_q/s_k\to 0$, we obtain for all large $k$, 
$$\left|\frac{1}{R(\Phi_{-s_k}(q))}+f(\Phi_{-s_k}(q))\right|\leq 5\eta s_k.$$
From the identity  
\begin{equation}
\label{eqn:linear-scalar-bound-8}
f(q)-f\left(\Phi_{-t}(q)\right)=\int_0^t|\nabla f|^2 d s=\int_0^t(1-R) d s\qquad \text{for all }t>0,
\end{equation}
it follows that $\frac{-f(\Phi_{-t}(q))}{t}\to 1$ as $t\to \infty$. Combining this with the previous estimate, we obtain for all large $k$,
\begin{equation}
\label{eqn:linear-scalar-bound-7}
\left|\frac{1}{s_kR(\Phi_{-s_k}(q))}-1\right|\leq 6\eta.
\end{equation}
Since $\eta$ was arbitrary, we conclude that $s_kR(\Phi_{-s_k}(q))\to 1$. Because $s_k\to\infty$ was arbitrary, this proves \eqref{eqn:linear-scalar-bound-100}. 
\end{proof}
In \cite{BCMZ23}, it was proved that for $\kappa$-noncollapsed steady solitons in dimension 4 with bounded curvature, the volume growth must $\gtrsim r^{5/2}$, where $r$ is the distance from a fixed basepoint. The volume growth on $\Bry^4$ is $\sim r^{5/2}$. We now show that under Assumptions \hyperref[assumption:A1]{(A1)}--\hyperref[assumption:A4]{(A4)}, the growth must be strictly faster than that of $\Bry^4$. In particular, this gives some information on the volume growth on the flying wings constructed by Lai \cite{Lai24}.
\begin{corollary}
\label{Cor:volume-growth}
We have 
$$\lim_{r\to \infty}\frac{\operatorname{Vol}_g(B_g[o;r])}{r^{5/2}}=\infty.$$
\end{corollary}
\begin{proof}
 Let $s_k\to -\infty$ and $L>0$. In what follows, $C$ denotes a positive constant that may change from line to line, but remains independent of $k$ and $L$. Let $q\in \Sigma$ such that after passing to a subsequence, $(M,R(\chi_{s_k}(q))g,\chi_{s_k}(q))$ converges to $(\sph^2\times \R^2,\bar{g}_0+dz^2,(\bar{x},0))$ as $k\to \infty$. Such points exist by Theorem \ref{existence-of-x-+-}. From Corollary \ref{Vanishing-lambda-2}, we obtain that $R(\chi_{s_k}(q))\to 0$. Using Lemma \ref{Level-Set-Convergence}, $(\Sigma_{s_k},R(\chi_{s_k}(q))g,\chi_{s_k}(q))$ converges to $(\sph^2\times \R,\bar{g}_0,\bar{x})$. For all large $k$, 
$$\operatorname{Vol}_{R(\chi_{s_k}(q))g}(\Sigma_{s_k})\geq \operatorname{Vol}(B_{\bar{g}_0}[\bar{x};L])\geq CL.$$
Let $t_k>0$ be such that $\chi_{s_k}(q)=\Phi_{-t_k}(q)$. Arguing as in \eqref{eqn:linear-scalar-bound-8}, we conclude that $\frac{-s_k}{t_k}=\frac{-f(\Phi_{-t_k}(q))}{t_k}\to 1$. Using Theorem \ref{Upper-bound-on-curves}, it follows that $(-s_k)R(\chi_{s_k}(q))\to 1$. Therefore, for all large $k$, 
$$(-s_k)^{-3/2}\operatorname{Vol}(\Sigma_{s_k})\geq CL.$$
Since $s_k$ and $L$ were arbitrary, this proves that
\begin{equation}
\label{eqn:volume-growth-level-sets-cor}
\lim_{s\to -\infty}\frac{\operatorname{Vol}(\Sigma_s)}{|s|^{3/2}}=\infty.
\end{equation}
 Using the co-area formula, we have for every $s<s_1\leq f(o)$,  
$$\int_{s\leq f\leq s_1}|\nabla f|=\int_{s}^{s_1}\int_{\Sigma_t} \,d\mu_t\,dt,$$
where $d\mu_t$ is the volume form induced by the metric $g|_{\Sigma_t}$. Let $L>0$. From $|\nabla f|\leq 1$ and \eqref{eqn:volume-growth-level-sets-cor}, it follows that there exists $s_L\ll 0$ such that for all $s\leq s_L$,  
$$\operatorname{Vol}(\{s\leq f\leq s_L\})\geq 100L\int_s^{s_L}(-t)^{3/2}\,dt.$$
For $s\ll s_L$, we conclude that 
\begin{equation}
\label{cor:vol-growth-eqn-1}
(-s)^{-5/2}\operatorname{Vol}(\{s\leq f\leq s_L\})\geq L.
\end{equation}
From Lemma \ref{lem:Linear-Growth-of-potential}, we have $f(o)-f(x)\geq c_0d(x,o)$ for $d(x,o)\geq 1$, where $c_0$ is a positive constant. This implies $\{s\leq f\leq s_L\}\sub B[o;c_0^{-1}(|s|+f(o))+1]$ for all $s\ll 0$. Using $r=|s|$ in \eqref{cor:vol-growth-eqn-1}, it follows for all large $r$,  
$$r^{-5/2}\operatorname{Vol}(B[o;c_0^{-1}(r+|f(o)|)+1])\geq L.$$
Letting $r\to \infty$, we obtain $\liminf_{r\to \infty}\frac{\operatorname{Vol}(B[o;r])}{r^{5/2}}\geq c_0^{5/2}L$. As $L$ was arbitrary, this finishes the proof of Corollary~\ref{Cor:volume-growth}. 
\end{proof}
Corollary \ref{Cor:volume-growth} implies the following rigidity for the 4D Bryant soliton based on volume growth. 
\begin{corollary}
Let $(\mathcal{N}^4,g_{\mathcal{N}})$ be a $\kappa$-noncollapsed steady soliton with nonnegative sectional curvature and positive Ricci curvature. If the potential function has a critical point and the volume growth of $\mathcal{N}^4$ is of the order $\lesssim r^{5/2}$, then $(\mathcal{N}^4,g_{\mathcal{N}})$ must be isometric to the Bryant soliton $\Bry^4$. 
\end{corollary}
\begin{proof}
By \cite[Corollary 5.4]{CMZ25b} the tangent flow at infinity of $\mathcal{N}^4$ is either $\sph^2\times \R^2$ or $\sph^3\times \R$. If it is $\sph^3\times \R$ then $(\mathcal{N}^4,g_\mathcal{N})$ is isometric to $\Bry^4$ by \cite[Theorem 5.2]{CMZ25b}. If it is $\sph^2\times \R^2$, then $(\mathcal{N}^4,g_\mathcal{N})$ satisfies the Assumptions \hyperref[assumption:A1]{(A1)}--\hyperref[assumption:A4]{(A4)} and by Corollary \ref{Cor:volume-growth} we obtain a contradiction to the assumed volume growth of $\mathcal{N}^4$.
\end{proof}

\section{Proof of the main result}
\label{sec:proof-of-the-main-result}
We continue to assume that $\left(M^4, g, f\right)$ satisfies \hyperref[assumption:A1]{(A1)}--\hyperref[assumption:A4]{(A4)}. In this section, we prove the main result of our paper (Theorem \ref{Full-linear-bound} and the first half of Theorem \ref{thm:main-theorem-2-stronger-scalar-bound}). We first prove a lemma that shows that the dimension reduction along a sequence of points is $\Bry^3\times \R$ if and only if the points remain at uniformly bounded rescaled distance from one of the two edges. 
\begin{lemma}
\label{Obtaining-points-from-Bryant-like-sequences}
Suppose that $x_k\in M$ with $d(x_k,o)\to \infty$. Then,  
\begin{equation}
\label{eqn:obtaining-points-1}
(M,R(x_k)g,x_k)\to (\Bry^3\times \R,R_{\tilde{g}_0}(\tilde{x})\tilde{g}_0+dz^2,(\tilde{x},0)),
\end{equation}
in the smooth Cheeger--Gromov sense if and only if there exists $s_k\to \infty$ such that for at least one of $x_+$ or $x_-$, we have 
\begin{equation}
\label{nearness-to-tip-3}
\sup_k R(x_k)d(\Phi_{-s_k}(x_\pm),x_k)^2<\infty.
\end{equation}
\end{lemma}
\begin{proof}
In the proof, $C(\cdot),C'(\cdot)$ are constants depending on their parameters and on $\kappa$, which may change from line to line. 
 
Assume that (\ref{nearness-to-tip-3}) holds for $x_+$.  Write $z_k:=\Phi_{-s_k}(x_+)$ and recall that $s\mapsto \Phi_{-s}(x_+)$ and $s\mapsto \chi_{-(s-s_0)}(x_+)$ parameterize the same curve for $s\geq 0$. From Theorem \ref{existence-of-x-+-}, we have 
$$(M,R(z_k)g,z_k)\to (\Bry^3\times \R,\tilde{g}_0+dz^2,(\bar{x},0))$$
where $\bar{x}$ is the tip of $\Bry^3$. The estimate \eqref{nearness-to-tip-3} implies that $d_{R(z_k)g}(x_k,z_k)$ is uniformly bounded. Hence, under the above pointed convergence, after passing to a further subsequence, the points $x_k$ converge to some point $(\tilde x,c)\in\Bry^3\times\R$. Therefore,
$$(M, R(z_k)g,x_k)\to (\Bry^3\times\R,\tilde g_0+dz^2,(\tilde x,c)).$$
Finally, since $R(x_k)/R(z_k)\to R_{\tilde{g}_0}(\tilde{x}) \in (0,1]$, we may replace $R(z_k)$ by $R(x_k)$ to obtain \eqref{eqn:obtaining-points-1}. 

Conversely, assume that \eqref{eqn:obtaining-points-1} holds. Let $\bar x$ denote the tip of $\Bry^3$ and choose $D>0$ such that
$$d_{R_{\tilde g_0}(\tilde x)\tilde g_0+dz^2}\big((\bar x,0),(\tilde x,0)\big)\le D.$$
Let $\varphi_k$ denote the Cheeger--Gromov convergence maps defining \eqref{eqn:obtaining-points-1}. Define $y_k:=\varphi_k(\bar x,0)$. Then $y_k\to(\bar x,0)$ under the convergence (\ref{eqn:obtaining-points-1}) in the Cheeger--Gromov sense. It follows that $R(x_k)d_g(x_k,y_k)^2\leq 2D^2$ for all large $k$, and $\frac{\lambda_2}{R}(y_k)\to \frac{1}{3}$. From Perelman's long-range estimates (Lemma \ref{summary-of-CMZ-work} (ii)), it follows that
$$R(y_k)\,d_g(x_k,y_k)^2\le C(D)\,R(x_k)\,d_g(x_k,y_k)^2\le C'(D),$$
and $y_k\to\infty$. Let $t_k:=f(y_k)$, so that $y_k\in \Sigma_{t_k}$. Since $y_k\to\infty$, we have $t_k\to-\infty$. Let $p_{1,t_k},p_{2,t_k}$ be the two tips on $\Sigma_{t_k}$ given by Theorem \ref{Two-Tips}.

Arguing exactly as in Claim \ref{claim:nearness-of-tip-edge} of Theorem \ref{existence-of-x-+-}, we conclude that $R(y_k)d(p_{i,t_k},y_k)^2\to 0$ as $k\to \infty$ for at least one of $i=1,2$. Without loss of generality, $R(y_k)d(p_{1,t_k},y_k)^2\to 0$ as $k\to \infty$. 

By Theorem \ref{existence-of-x-+-}, after passing to a subsequence we have
$$R(p_{1,t_k})\,d\big(p_{1,t_k},\chi_{t_k}(x_+)\big)^2\to 0.$$
Combining this with $R(y_k)\,d(p_{1,t_k},y_k)^2\to 0$ and the boundedness of $R(x_k)d(x_k,y_k)^2$, and using Perelman's long-range estimate, we obtain
$$R(x_k)\le C(D)\,R(y_k)\le C'(D)\,R(p_{1,t_k}).$$
Using these estimates and the triangle inequality, 
$$\begin{aligned}
&R(x_k)\,d\big(\chi_{t_k}(x_+),x_k\big)^2\\
&\le C(D)\,R(y_k)\Big(d(\chi_{t_k}(x_+),p_{1,t_k})^2+d(p_{1,t_k},y_k)^2+d(y_k,x_k)^2\Big)\\
&\le C(D)\Big(R(p_{1,t_k})d(\chi_{t_k}(x_+),p_{1,t_k})^2+R(y_k)d(p_{1,t_k},y_k)^2+R(y_k)d(y_k,x_k)^2\Big),
\end{aligned}$$
and each term on the right-hand side is bounded uniformly in $k$. This proves \eqref{nearness-to-tip-3} after choosing $s_k\to \infty$ such that $\Phi_{-s_k}(x_+)=\chi_{t_k}(x_+)$ for all large $k$. 
\end{proof}
\begin{remark}
\label{rem:polynomial-decay-not-possible}
It follows from Lemma \ref{Obtaining-points-from-Bryant-like-sequences} that if $d(x_k,o)\to \infty$ and $M$ dimension reduces to $\sph^2\times \R$ along $x_k$, then $R(x_k)d(x_k,\G)^2\to \infty$. As a consequence, $\sup_{x\in M} R(x)d(x,\G)^2=+\infty$.
\end{remark}
We now prove Theorem \ref{Full-linear-bound} and the first half of Theorem \ref{thm:main-theorem-2-stronger-scalar-bound}. Recall from Definition \ref{defn:edges-of-soliton} that $\G$ is the union of two integral curves of $-\nabla f/|\nabla f|$ passing through $x_+,x_-$.
\begin{theorem}
\label{thm:full-linear-bound-proof}
There exists $C>0$ depending on $\kappa$ and the manifold such that 
$$R(x)d(x,\G)\leq C\qquad \text{ for all }x\in M.$$
Here and below, $d(x,\G)$ denotes the distance from $x$ to the set $\G$. Further, if 
\begin{equation}
\label{eqn:vanishing-scal-at-infinity}
\tag{*}
\lim_{d(x,o)\to \infty}R(x)=0,
\end{equation}
then the following stronger bound holds:
$$\lim_{d(x,o)\to \infty}R(x)d(x,\G)=0.$$
\end{theorem}
\begin{proof}
We will show the following claim. 
\begin{claim}
\label{claim:scalar-bound-for-sequence}
Let $x_k$ be any sequence in $M$ such that $d(x_k,o)\to \infty$. After passing to a subsequence of $(x_k)$, we have   
$$\sup_k R(x_k)d(x_k,\G)<\infty.$$
Further, if (\ref{eqn:vanishing-scal-at-infinity}) holds, then after passing to a subsequence we have $R(x_k)d(x_k,\G)\to 0$. 
\end{claim}
Proving Claim \ref{claim:scalar-bound-for-sequence} is enough to show the theorem: if the theorem were false, we could choose a sequence $x_k\in M$ with $d(x_k,o)\to \infty$ and $R(x_k)d(x_k,\G)\to\infty$, contradicting the first part of the claim. Similarly, under \eqref{eqn:vanishing-scal-at-infinity}, if $\limsup_{d(x,o)\to\infty}R(x)d(x,\G)>0$ then we could extract a sequence $x_k\in M$ with $d(x_k,o)\to \infty$ and $R(x_k)d(x_k,\G)\ge c>0$ for all $k$, contradicting the second part of the claim. 

Therefore, we proceed to prove Claim \ref{claim:scalar-bound-for-sequence}. Let $x_k\in M$ with $d(x_k,o)\to \infty$. In what follows, $C$ is a positive constant depending on $\kappa$ and the manifold $(M,g)$, which may change from line to line, but will always stay independent of $k$. We consider two cases.

\noindent \textbf{Case 1.} Suppose $(M,R(x_k)g,x_k)\to (\Bry^3\times \R,R_{\tilde{g}_0}(\tilde{x})\tilde{g}_0+dz^2,(\tilde x,0))$, after passing to a subsequence. By Lemma \ref{Obtaining-points-from-Bryant-like-sequences}, after passing to a subsequence there exist $s_k\to\infty$ and a choice of sign $\pm$ such that
$$A_1:=\sup_k R(x_k)\,d\left(x_k,\Phi_{-s_k}(x_\pm)\right)^2<\infty.$$
Since $\Phi_{-s_k}(x_\pm)\in\G$, we have $d(x_k,\G)\le d(x_k,\Phi_{-s_k}(x_\pm))$, hence
$$R(x_k)d(x_k,\G)\le R(x_k)\,d\left(x_k,\Phi_{-s_k}(x_\pm)\right)
\le \sqrt{R(x_k)}\cdot \Big(R(x_k)d\left(x_k,\Phi_{-s_k}(x_\pm)\right)^2\Big)^{1/2}.$$
which implies 
$$R(x_k)d(x_k,\G)\leq A_1^{1/2}\sqrt{R(x_k)}.$$
This proves Claim \ref{claim:scalar-bound-for-sequence} in this case. 

\noindent \textbf{Case 2.} Suppose that $(M,R(x_k)g,x_k)\to (\sph^2\times \R^2,\bar{g}_0+dz^2,(\bar{x},0))$, after passing to a subsequence. Let $\varepsilon_1,r_1$ be constants from Theorem \ref{Upper-bound-on-curves}, and let $\epsilon\in (0,\varepsilon_1)$ (independent of $k$). Choose $s_1\ll 0$ such that $\Sigma_{s_1}\subset \{d(\cdot,o)>r_1\}$. Choose $p_k\in \Sigma_{s_1}$ and $t_k>0$ such that $\Phi_{-t_k}(p_k)=x_k$. As $d(x_k,o)\to \infty$, we have $t_k\to \infty$. Indeed, $f(p_k)=s_1$ is fixed while $f(x_k)\to-\infty$, and along $\Phi_{-t}$ we have $-\frac{d}{dt}f(\Phi_{-t}(p_k))=|\nabla f|^2\leq 1$ for $t\geq 0$, and integrating gives $t_k\to\infty$.

Define
$$\hat{t}_k^1:=\inf\{t\in[0,t_k]:\ \Phi_{-t}(p_k)\ \text{is an }(\epsilon,2)\text{-center}\}.$$
Such $\hat{t}_k^1$ exists because $x_k$ is an $(\epsilon,2)$-center for all large $k$ and $\hat{t}_k^1<t_k$. Choose $ t_k^1\in(\hat{t}_k^1,\hat{t}_k^1+\tfrac{1}{k})$ so that $q_k^1:=\Phi_{- t_k^1}(p_k)$ is an $(\epsilon,2)$-center. Note that $d(q_k^1,o)\geq d(p_k,o)> r_1$ for all large $k$. We may ensure that $t_k^1\leq t_k$, and hence $R(q_k^1)\geq R(x_k)$. By Theorem \ref{Upper-bound-on-curves} applied to $q_k^1$, we have for all large $k$, 
\begin{equation}
\label{eqn:first-linear-bd-t-k-1}
R(x_k)\leq \frac{C}{t_k-t^1_k}.
\end{equation}
We consider two sub-cases based on whether $(t_k^1)_k$ is bounded or not. 

\noindent \textbf{Case 2a.} Suppose that $a:=\sup_k t_k^1<\infty$. In this case, from (\ref{eqn:first-linear-bd-t-k-1}) it follows that for all large $k$, 
$$t_kR(x_k)\leq C.$$
Define for $s\in[0,t_k-a]$, 
$$d_{\pm,k}(s):=d\left(\Phi_{-(s+a)}(p_k),\,\Phi_{-s}(x_\pm)\right)
=d_{g(-s)}\left(\Phi_{-a}(p_k),\,x_\pm\right).$$
Note that $d_{\pm,k}(0)$ is bounded above uniformly in $k$ since $\Phi_{-a}(\Sigma_{s_1})$ is compact. We obtain 
\begin{equation}
\label{eqn:case-2a-bound-on-scalar-along-x_k}
\begin{aligned}
&R(x_k)d(x_k,\G)\leq C\frac{d\left(x_k,\Phi_{-(t_k-a)}(x_\pm)\right)}{t_k}=C\frac{d_{\pm,k}(t_k-a)}{t_k}.\\
\end{aligned}
\end{equation}
By Lemmas \ref{lem:variation-of-distance} and \ref{stability-inequalities} (cf. Remark \ref{rem:stability-inequalities-application-remark}), we obtain for all $s\geq 0$, 
\begin{equation}
\label{eqn:case-2a-bound-on-partial-d-pm}
\partial_s^+ d_{\pm,k}(s)\le C\max\Big(\sqrt{R(\Phi_{-(s+a)}(p_k))},\,\sqrt{R(\Phi_{-s}(x_\pm))}\Big).
\end{equation}
Integrating (\ref{eqn:case-2a-bound-on-partial-d-pm}) on $[0,t_k-a]$, and using (\ref{eqn:case-2a-bound-on-scalar-along-x_k}) it follows that $\sup_k R(x_k)d(x_k,\G)<\infty$. This proves the first part of Claim \ref{claim:scalar-bound-for-sequence}. 

Assume vanishing scalar curvature at infinity (\ref{eqn:vanishing-scal-at-infinity}). We have $\inf_{p\in\Sigma_{s_1}\cup \Sigma_{s_0}} d(o,\Phi_{-s}(p))\to\infty$ as $s\to\infty$. Together with \eqref{eqn:vanishing-scal-at-infinity} it follows that for every $\eta>0$, there exists $b>0$ such that $R(\Phi_{-s}(p))\le \eta^2$ for all $s\ge b$ and all $p\in\Sigma_{s_1}$. With \eqref{eqn:case-2a-bound-on-scalar-along-x_k} and integrating \eqref{eqn:case-2a-bound-on-partial-d-pm} on $[b,t_k-a]$, we obtain for all large $k$, 
$$R(x_k)d(x_k,\G)\leq C\frac{d_{\pm,k}(t_k-a)}{t_k}\leq \frac{C}{t_k}(d_{\pm,k}(b)+\eta \cdot(t_k-a-b)).$$
Since $\sup_k d_{\pm,k}(b)<\infty$, we obtain $R(x_k)d(x_k,\G)\leq 2C\eta$ for all large $k$. This proves the second part of Claim \ref{claim:scalar-bound-for-sequence} for this case. 

\noindent \textbf{Case 2b.} Suppose that $a=\sup_k t_k^1=\infty$. After passing to a subsequence, $t_k^1\to \infty$. 

Since $q_k^1=\Phi_{-t_k^1}(p_k)$ and $t_k^1\to \infty$, we have $d(q_k^1,o)\to \infty$.
Consider the pointed sequence $(M,R(q_k^1)g,q_k^1)$. By dimension reduction, after passing to a subsequence, it converges smoothly to either $\sph^2\times\R^2$ or $\Bry^3\times\R$. 
The limit cannot be $\sph^2\times\R^2$: if it were, the sequence $(M,R(\Phi_{1}q_k^1)g,\Phi_{1}q_k^1)$ also converges to $\sph^2\times \R^2$ since 
$$\sqrt{R(q_k^1)}d_g(q_k^1,\Phi_{1}q_k^1)\leq \int_{0}^{1}|\nabla f|(\Phi_{s}q_k^1)\,ds\leq 1,$$
for all $k$. This implies that $\Phi_{-(t_k^1-1)}p_k=\Phi_{1}q^1_k$ is an $(\epsilon,2)$-center for all large $k$, contradicting the choice of $t_k^1$ and $\hat{t}_k^1$. Hence
$$(M,R(q_k^1)g,q^1_k)\to (\Bry^3\times\R,R_{\tilde{g}_0}(\tilde{x})\tilde{g}_0+dz^2,(\tilde x,0)).$$
By Lemma \ref{Obtaining-points-from-Bryant-like-sequences}, without loss of generality, there exists $s_k^1\to \infty$ such that 
$$A_2:=\sup_{k}R(q_k^1)d(q_k^1,\Phi_{-s_k^1}(x_+))^2<\infty.$$
In particular, we have for all $k$,  
\begin{equation}
\label{eqn:nearness-to-tip}
R(q_k^1)\,d\left(q_k^1,\Phi_{-s_k^1}(x_+)\right)\leq \sqrt{R(q_k^1)}\,\sqrt{A_2}.
\end{equation}
Set $d_k(s)=d(\Phi_{-(s+t_k^1)}(p_k),\Phi_{-(s+s_k^1)}(x_+))$. By Lemmas \ref{lem:variation-of-distance} and \ref{stability-inequalities}, we obtain for $s\geq 0$, 
$$\partial_s^+ d_k(s)\leq C\max\Big(\sqrt{R(\Phi_{-(s+t_k^1)}(p_k))},\,\sqrt{R(\Phi_{-(s+s_k^1)}(x_+))}\Big).$$
Integrating and using that $s\mapsto R(\Phi_{-s}(x))$ is non-increasing for any $x\in M$ (Lemma \ref{lem:monotonicity-along-flows}), 
$$d_k(t_k- t_k^1)\leq d_k(0)+m_k(t_k-t_k^1),$$
where $m_k:=\max\Big(\sqrt{R(\Phi_{-t_k^1}(p_k))},\,\sqrt{R(\Phi_{-s_k^1}(x_+))}\Big)\in (0,1]$. We obtain 
\begin{equation}
\label{eqn:distance-bound-t-k-1}
\begin{aligned}
d(x_k,\G)&\leq d\left(x_k,\Phi_{-(t_k-t_k^1+s_k^1)}(x_+)\right)\\
&=d_k(t_k-t_k^1)\\
&\leq d\left(q_k^1,\Phi_{-s_k^1}(x_+)\right)+Cm_k(t_k-t_k^1).
\end{aligned}
\end{equation}
Combining with (\ref{eqn:first-linear-bd-t-k-1}), we obtain 
\begin{equation}
\label{eqn:second-bound-on-x-k}
\begin{aligned}
&R(x_k)d(x_k,\G)\leq R(x_k)\,d\left(q^1_k,\Phi_{-s^1_k}(x_+)\right)+Cm_k.
\end{aligned}
\end{equation}
Using (\ref{eqn:nearness-to-tip}), (\ref{eqn:second-bound-on-x-k}), and $R(x_k)\leq R(q_k^1)$, we obtain 
\begin{equation}
\label{eqn:first-bound-on-x-k}
\begin{aligned}
&R(x_k)d(x_k,\G)\leq \sqrt{R(q_k^1)}\,\sqrt{A_2}+Cm_k,
\end{aligned}
\end{equation}
uniformly in $k$. This estimate proves the first part of Claim \ref{claim:scalar-bound-for-sequence}, and as (\ref{eqn:vanishing-scal-at-infinity}) implies $m_k\to 0,R(q_k^1)\to 0$, we obtain the second part of Claim \ref{claim:scalar-bound-for-sequence}. 

This finishes the proof of Claim \ref{claim:scalar-bound-for-sequence}, completing the proof of Theorem \ref{thm:full-linear-bound-proof}. 
\end{proof}

\section{Asymptotic cone of the soliton}
\label{sec:asymptotic-cone-of-the-soliton}
In this section, we analyze the relationship between the asymptotic cone of the soliton and the behavior of the scalar curvature at infinity.
Recall the function $G$ defined in Lemma \ref{G-on-Level-sets-lemma} which computes the limit of scalar curvature along integral curves of $-\nabla f/|\nabla f|$. We begin with the following lemma that shows how the distance of $\Phi_{-s}(p)$ from $o$ behaves for $s\gg 0$. 
\begin{lemma}
\label{lem:distance-p-Phi-from-o-asymptotic}
Let $(M^4,g,f)$ satisfy \textup{\hyperref[assumption:A1]{(A1)}, \hyperref[assumption:A2]{(A2)}} and let $p\in M\setminus \{o\}$. Let $G$ be the function defined in Lemma \ref{G-on-Level-sets-lemma}. Then, $$\lim_{s\to \infty}\frac{d(\Phi_{-s}(p),o)}{s}=\sqrt{1-G(p)}.$$
\end{lemma}
\begin{proof}
Recall that $\G_p$ was defined to be the integral curve of $-\nabla f/|\nabla f|$ starting at $p$. From \cite[Lemma 3.2]{Lai24}, we have 
\begin{equation}
\label{eqn:distance-p-Phi-from-o-1}
\lim_{s\to \infty}\frac{d(\G_p(s),o)}{s}=1.
\end{equation}
For each $s>0$, let $\beta(s)>0$ be such that $\Phi_{-s}(p)=\G_p(\beta(s))$ with $\beta (0)=0$. Then, $\beta'(s)=|\nabla f|(\Phi_{-s}(p))$ for $s>0$. From Lemma \ref{lem:monotonicity-along-flows}, we have $|\nabla f|(\Phi_{-s}(p))\geq |\nabla f|(p)>0$ for all $s>0$, it follows that $\beta(s)\to \infty$ as $s\to \infty$. Combining this with \eqref{eqn:distance-p-Phi-from-o-1}, we obtain 
$$\lim_{s\to \infty}\frac{d(\Phi_{-s}(p),o)}{s}=\lim_{s\to \infty}\frac{d(\G_p(\beta(s)),o)}{s}=\lim_{s\to \infty}\frac{\beta(s)}{s}.$$
Using $|\nabla f|^2(\Phi_{-s}(p))=1-R(\Phi_{-s}(p))\to 1-G(p)$ as $s\to \infty$, we obtain 
$$\lim_{s\to \infty}\frac{d(\Phi_{-s}(p),o)}{s}=\lim_{s\to \infty}\frac{\beta(s)}{s}=\lim_{s\to \infty}\frac{1}{s}\int_0^s \beta'(t)\,dt=\lim_{s\to \infty}\beta'(s)=\sqrt{1-G(p)},$$
completing the proof. 
\end{proof}
We now prove the second half of Theorem \ref{thm:main-theorem-2-stronger-scalar-bound} using Theorem \ref{thm:full-linear-bound-proof}. Recall from Theorem \ref{existence-of-x-+-} that the scalar curvature vanishes at infinity if and only if $G(x_+)=0=G(x_-)$. 
\begin{theorem}
\label{thm:asymptotic-cone-ray-proof}
 Assume $(M^4,g,f)$ is a steady soliton that satisfies \textup{\hyperref[assumption:A1]{(A1)}--\hyperref[assumption:A4]{(A4)}}. Suppose that $\lim_{d(x,o)\to \infty}R(x)=0$. Then, the asymptotic cone of $(M,g)$ is a ray. 
\end{theorem}
\begin{proof}
Let $v,w\in T_oM$ be two unit vectors. Let $\g_v,\g_w:[0,\infty)\to M$ denote two minimizing geodesic rays emanating from $o$ with initial velocities $v,w$ respectively. 

By \cite[Theorem 2.6]{CCMZ23}, the scalar curvature on $M$ satisfies the lower bound 
$$R(x)\geq \frac{C^{-1}}{1+d(x,o)}\qquad \text{ for all }x\in M,$$
 where $C$ is a positive constant. As a result, $\inf_{s\geq 1} sR(\g_v(s))>0$ and $\inf_{s\geq 1} sR(\g_w(s))>0$. Using Theorem \ref{thm:full-linear-bound-proof}, we have $R(\g_v(s))d(\G,\g_v(s))\to 0$ as $s\to \infty$ and $R(\g_w(s))d(\G,\g_w(s))\to 0$ as $s\to \infty$. Combining these facts, it follows that  
\begin{equation}
\label{eqn:asymptotic-cone-thm-1}
\lim_{s\to \infty}\frac{d(\G,\g_v(s))}{s}=0=\lim_{s\to \infty}\frac{d(\G,\g_w(s))}{s}.
\end{equation}
Let $s_k\to \infty$ be a sequence of positive numbers. Recall that we defined $x_\pm$ in Theorem \ref{existence-of-x-+-}. After passing to a subsequence, choose $r_k,l_k$ such that, for each $k$,
$$d(\G,\g_v(s_k))=d(\Phi_{-r_k}(x_{e(v)}),\g_v(s_k)),\qquad d(\G,\g_w(s_k))=d(\Phi_{-l_k}(x_{e(w)}),\g_w(s_k)),$$
 where $e(v),e(w)\in \{+,-\}$ allowing for the possibility that $e(v)\neq e(w)$. By the triangle inequality, we have for all $k$, 
$$\begin{aligned}
s_k-d(\Phi_{-r_k}(x_{e(v)}),\g_v(s_k))\leq d(\Phi_{-r_k}(x_{e(v)}),o)\leq d(\Phi_{-r_k}(x_{e(v)}),\g_v(s_k))+s_k.
\end{aligned}$$
Dividing by $s_k$ and taking $k\to \infty$, it follows from \eqref{eqn:asymptotic-cone-thm-1} that
\begin{equation}
\label{eqn:asymptotic-cone-eqn-3}
\lim_{k\to \infty}\frac{d(\Phi_{-r_k}(x_{e(v)}),o)}{s_k}=1.
\end{equation}
Similarly,  
\begin{equation}
\label{eqn:asymptotic-cone-eqn-4}
\lim_{k\to \infty}\frac{d(\Phi_{-l_k}(x_{e(w)}),o)}{s_k}=1.
\end{equation}
This implies that $r_k\to \infty,l_k\to \infty$ as $k\to \infty$. Since we assumed  $\lim_{d(x,o)\to \infty}R(x)=0$, we have $G(x_+)=0=G(x_-)$ by Lemma \ref{G-on-Level-sets-lemma}. It follows from Lemma \ref{lem:distance-p-Phi-from-o-asymptotic} that 
\begin{equation}
\label{eqn:asymptotic-cone-eqn-5}
\lim_{k\to \infty}\frac{d(\Phi_{-r_k}(x_{e(v)}),o)}{r_k}=1=\lim_{k\to \infty}\frac{d(\Phi_{-l_k}(x_{e(w)}),o)}{l_k}.
\end{equation}
Combining \eqref{eqn:asymptotic-cone-eqn-3}, \eqref{eqn:asymptotic-cone-eqn-4} and \eqref{eqn:asymptotic-cone-eqn-5} we obtain that 
\begin{equation}
\label{eqn:asymptotic-cone-eqn-6}
\lim_{k\to \infty}\frac{s_k}{r_k}=1=\lim_{k\to \infty}\frac{s_k}{l_k}.
\end{equation}
By triangle inequality, 
$$\begin{aligned}
\frac{d(\g_{v}(s_k),\g_w(s_k))}{s_k}&\leq \frac{d(\g_{v}(s_k),\Phi_{-r_k}(x_{e(v)}))}{s_k}+\frac{d(\Phi_{-l_k}(x_{e(w)}),\g_w(s_k))}{s_k}\\
&\quad +\frac{d(\Phi_{-r_k}(x_{e(v)}),\Phi_{-l_k}(x_{e(w)}))}{s_k}.\\
\end{aligned}$$
From \eqref{eqn:asymptotic-cone-thm-1} it follows that 
\begin{equation}
\label{eqn:asymptotic-cone-eqn-2}
\begin{aligned}
&\limsup_{k\to \infty}\frac{d(\g_{v}(s_k),\g_w(s_k))}{s_k}\leq \limsup_{k\to \infty}\frac{d(\Phi_{-r_k}(x_{e(v)}),\Phi_{-l_k}(x_{e(w)}))}{s_k}.
\end{aligned}
\end{equation}

\begin{claim}
\label{claim:asymptotic-cone-1}
The right-hand side of \eqref{eqn:asymptotic-cone-eqn-2} is zero. 
\end{claim}
The proof of Claim \ref{claim:asymptotic-cone-1}  proceeds by considering two cases, depending on whether $e(v),e(w)$ coincide or are distinct.

\noindent \textbf{Case 1.} Suppose that $e(v)=e(w)$. Without loss of generality, let $e(v)=e(w)=+$. Using $|\nabla f|^2\leq 1$ on $M$, we obtain for each $k$, 
$$d(\Phi_{-r_k}(x_+),\Phi_{-l_k}(x_+))\leq \int_{[\min(r_k,l_k),\max(r_k,l_k)]}|\nabla f|(\Phi_{-t}(x_+))\,dt\leq |r_k-l_k|.$$
Combining this with \eqref{eqn:asymptotic-cone-eqn-6}, we obtain 
$$\lim_{k\to \infty}\frac{d(\Phi_{-r_k}(x_+),\Phi_{-l_k}(x_+))}{s_k}\leq \lim_{k\to \infty}\frac{|r_k-l_k|}{s_k}=0.$$
The proof of Claim \ref{claim:asymptotic-cone-1} in this case is complete. 

\noindent \textbf{Case 2.} Suppose $e(v)\neq e(w)$. After relabelling if necessary, we may assume $e(v)=+,e(w)=-$. From triangle inequality, 
\begin{equation}
\label{eqn:asymptotic-cone-eqn-9}
\frac{d(\Phi_{-r_k}(x_+),\Phi_{-l_k}(x_-))}{s_k}\leq \frac{d(\Phi_{-r_k}(x_+),\Phi_{-r_k}(x_-))}{s_k}+\frac{d(\Phi_{-r_k}(x_-),\Phi_{-l_k}(x_-))}{s_k}.
\end{equation}
Following the same reasoning as in Case 1, we find that $\lim_{k\to \infty}\frac{d(\Phi_{-r_k}(x_-),\Phi_{-l_k}(x_-))}{s_k}=0$. Using Lemmas \ref{lem:variation-of-distance} and \ref{stability-inequalities} (cf. Remark \ref{rem:stability-inequalities-application-remark}), we have for all $s>0$, 
\begin{equation}
\label{eqn:asymptotic-cone-eqn-7}
\p_s^+ d(\Phi_{-s}(x_+),\Phi_{-s}(x_-))\leq C\max(R(\Phi_{-s}(x_+))^{1/2},R(\Phi_{-s}(x_-))^{1/2}).
\end{equation}
Let $\eta>0$. Since $G(x_+)=0$ and $G(x_-)=0$, there exists $s_1$ such that the right-hand side of \eqref{eqn:asymptotic-cone-eqn-7} is bounded above by $\eta$ for all $s\geq s_1$. Integrating \eqref{eqn:asymptotic-cone-eqn-7} we obtain for all large $k$,  
$$d(\Phi_{-r_k}(x_+),\Phi_{-r_k}(x_-))\leq d(\Phi_{-s_1}(x_+),\Phi_{-s_1}(x_-))+\eta (r_k-s_1).$$
Dividing by $s_k$, taking $k\to \infty$, and using \eqref{eqn:asymptotic-cone-eqn-6}, we obtain 
$$\limsup_{k\to \infty}\frac{d(\Phi_{-r_k}(x_+),\Phi_{-r_k}(x_-))}{s_k}\leq \eta.$$
Combining this with \eqref{eqn:asymptotic-cone-eqn-9}, since $\eta>0$ was arbitrary, 
$$\lim_{k\to \infty}\frac{d(\Phi_{-r_k}(x_+),\Phi_{-l_k}(x_-))}{s_k}=0.$$
This completes the proof of Claim \ref{claim:asymptotic-cone-1}. From Claim \ref{claim:asymptotic-cone-1}, it follows that given any $s_k\to \infty$, we may pass to a subsequence in $(s_k)_k$ to ensure that 
$$\lim_{k\to \infty}\frac{d(\g_{v}(s_k),\g_w(s_k))}{s_k}=0.$$
By Toponogov's theorem $s\mapsto \frac{d(\g_{v}(s),\g_w(s))}{s}$ is nonincreasing. We conclude that for every pair of minimizing geodesic rays $\g_v,\g_w$ starting at $o$, we have $\lim_{s\to \infty}\frac{d(\g_{v}(s),\g_w(s))}{s}= 0$. This shows that the asymptotic cone of $(M,g)$ is a ray, completing the proof of Theorem~\ref{thm:asymptotic-cone-ray-proof}. 
\end{proof}

\subsection{\texorpdfstring{Proof of Theorem \ref{thm:converse-to-main-theorem-2-stronger-scalar-bound-4d}}{Proof of Theorem}}
  
From now on, we assume that $(M^4,g,f)$ is a complete gradient steady Ricci soliton satisfying
\begin{equation}
\label{assumptions-to-converse-of-a-main-theorem-final-section}
\Ric>0,\qquad \nabla f(o)=0,\qquad \sec\ge 0 \text{ on } M.
\end{equation}
Accordingly, we retain \hyperref[assumption:A1]{(A1)} and \hyperref[assumption:A2]{(A2)}, and from \hyperref[assumption:A3]{(A3)} we keep only the condition $\sec\geq 0$. We do not assume $\kappa$-noncollapsedness, and we do not assume \hyperref[assumption:A4]{(A4)}. 

We next define a function measuring the asymptotic radial component of $-\nabla f$ along geodesics emanating from $o$. 
\begin{definition}
\label{defn:defining-angle-function}
Let $UT_oM:=\{v\in T_oM:|v|_g=1\}$. For each $v\in UT_oM$, let $\gamma_v:[0,\infty)\to M$ denote the unit-speed geodesic with $\gamma_v(0)=o, \gamma_v'(0)=v$. Define
\begin{equation}
\label{eqn:defn-of-angle-function-alpha-1}
\alpha(v):=\lim_{r\to\infty}\big\langle \gamma_v'(r),-\nabla f|_{\g_v(r)}\big\rangle_{g}.
\end{equation}
\end{definition}
The limit in \eqref{eqn:defn-of-angle-function-alpha-1} exists for every $v\in UT_oM$, hence $\alpha:UT_oM\to[0,1]$ is well-defined. Indeed, define for $r\geq 0$, $\beta_v(r):=\big\langle \gamma_v'(r),-\nabla f|_{\g_v(r)}\big\rangle_{g}$. Since $\nabla f(o)=0$, we have $\beta_v(0)=0$. Using the soliton equation, 
\begin{equation}
\label{eqn:defn-of-angle-function-alpha-0}
\beta_v'(r)
= -\nabla^2 f(\gamma_v'(r),\gamma_v'(r))
= \Ric(\gamma_v'(r),\gamma_v'(r)) > 0.
\end{equation}
Hence $\beta_v$ is increasing on $[0,\infty)$. Since $|\nabla f|\le1$ on $M$,
\[
0\le \beta_v(r)\le |\nabla f|(\gamma_v(r))\le 1
\qquad \text{for all } r\ge0.
\]
Therefore $\beta_v(r)$ converges as $r\to\infty$, showing that $\alpha(v)$ is well-defined and $0\le \alpha(v)\le 1$. Since $\beta_v=-(f\circ \g_v)'$, it follows that
\begin{equation}
\label{eqn:defn-of-angle-function-alpha-2}
\begin{aligned}
\alpha(v)&=\lim_{r\to \infty}\beta_v(r)=\lim_{r\to\infty}\frac1r\int_0^r \beta_v(s)\,ds=\lim_{r\to \infty}\frac{f(o)-f(\gamma_v(r))}{r}.
\end{aligned}
\end{equation}
We define $Z:=\{v\in UT_oM:\gamma_v \text{ is a minimizing geodesic ray}\}$. We say that $\gamma_v$ is a minimizing geodesic ray if
\[
d(\gamma_v(t),\gamma_v(s))=t-s
\qquad\text{for all }0\le s\le t<\infty.
\]
Then $Z$ is a closed subset of $UT_oM$. Indeed, if $v_k\to v$ in $UT_oM$ and $v_k\in Z$ for all $k$, then for each $0\le s\le t<\infty$, $d(\gamma_{v_k}(t),\gamma_{v_k}(s))=t-s$. Since geodesics depend smoothly on their initial data, $\gamma_{v_k}(t)\to \gamma_v(t)$ and $\gamma_{v_k}(s)\to \gamma_v(s)$ as $k\to\infty$. By continuity of the distance function, $d(\gamma_v(t),\gamma_v(s))=t-s$, showing that $v\in Z$. 
\begin{lemma}
The function $\alpha:Z\to[0,1]$ is $1$-Lipschitz with respect to the angular distance on $UT_oM$.
\end{lemma}
\begin{proof}
Let $v_1,v_2\in Z$, and let $\theta:=\angle(v_1,v_2)\in [0,\pi]$ so that $\cos \theta=\langle v_1,v_2\rangle$. Fix $r\ge0$. Since $v_1,v_2\in Z$, the geodesic segments $\gamma_{v_i}|_{[0,r]}$ are minimizing. Hence, by Toponogov's theorem (see \cite[Theorem~1.171]{CLN06}), 
\[
d(\gamma_{v_1}(r),\gamma_{v_2}(r))
\le 2r\sin(\theta/2)\le r\theta.
\]
Using $|\nabla f|\le1$, we obtain
\[
\left|
\frac{f(o)-f(\gamma_{v_1}(r))}{r}
-
\frac{f(o)-f(\gamma_{v_2}(r))}{r}
\right|
\le
\frac{d(\gamma_{v_1}(r),\gamma_{v_2}(r))}{r}
\le \theta.
\]
Letting $r\to\infty$ and using \eqref{eqn:defn-of-angle-function-alpha-2}, we conclude that
\[
|\alpha(v_1)-\alpha(v_2)|\le \angle(v_1,v_2).
\]
Therefore $\alpha$ is $1$-Lipschitz on $Z$.
\end{proof}
We now obtain minimizing geodesic rays from integral curves of $-\nabla f/|\nabla f|$. Fix $p\in M\setminus\{o\}$, and recall the definition of $\G_p$ from \eqref{eqn:definition-of-Gamma-p}. From Lemma~\ref{lem:monotonicity-along-flows}, we have $d(\Gamma_p(t),o)\to\infty$ as $t\to\infty$. For each $t>0$, let $\sigma_t^{(p)}:[0,d(\Gamma_p(t),o)]\to M$ be a unit-speed minimizing geodesic such that
\[
\sigma_t^{(p)}(0)=o,
\qquad
\sigma_t^{(p)}(d(\Gamma_p(t),o))=\Gamma_p(t).
\]
Then there exists $v_t^{(p)}\in UT_oM$ such that
\[
\sigma_t^{(p)}=\gamma_{v_t^{(p)}}\big|_{[0,d(\Gamma_p(t),o)]}.
\]
By compactness of $UT_oM$ there exist integers $t_{k}\to\infty$ and $v^{(p)}\in UT_o M$, such that
\[
v_{t_{k}}^{(p)}\to v^{(p)}.
\]
Since geodesics depend smoothly on their initial data, it follows that $\sigma_{t_{k}}^{(p)}\to \gamma_{v^{(p)}}$ uniformly on compact subintervals of $[0,\infty)$. Set $\sigma_p:=\gamma_{v^{(p)}}$. 
\begin{lemma}
\label{lem:obtained-ray-stays-asymptotically-close-to-curve}
Let $p\in M\setminus \{o\}$. $\sigma_p$ is a minimizing geodesic ray; equivalently, $\sigma_p'(0)\in Z$. There exists a sequence $s_k\to\infty$ such that
\begin{equation}
\label{eqn:obtained-ray-stays-asymptotically-close-to-curve}
\lim_{k\to\infty}\frac{d_g\bigl(\sigma_p(s_k),\Gamma_p\bigr)}{s_k}=0.
\end{equation}
Here, $d_g\bigl(q,\Gamma_p\bigr)$ denotes the distance between $q$ and $\G_p([0,\infty))$ for $q\in M$. 
\end{lemma}
\begin{proof}
Let $t_k\to\infty$ be the sequence used to define $\sigma_p$. Let $v^{(p)}:=\sigma_p'(0), v_{t_k}^{(p)}:=(\sigma_{t_k}^{(p)})'(0)$. Then $v_{t_k}^{(p)}\to v^{(p)}$ and $\sigma_{t_k}^{(p)}\to \sigma_p$ uniformly on compact intervals. Fix $0\le s\le u<\infty$. For all sufficiently large $k$, we have $u\le d(\Gamma_p(t_{k}),o)$ so the restriction $\sigma_{t_{k}}^{(p)}|_{[s,u]}$ is minimizing, and hence
\[
d\bigl(\sigma_{t_{k}}^{(p)}(s),\sigma_{t_{k}}^{(p)}(u)\bigr)=u-s.
\]
Letting $k\to\infty$ and using continuity of the distance function, we obtain $d(\sigma_p(s),\sigma_p(u))=u-s.$ Therefore $\sigma_p$ is a minimizing geodesic ray, and in particular $\sigma_p'(0)\in Z.$ 

Define $s_k:=d(o,\Gamma_p(t_k))$ so that $s_k\to\infty$. By construction,
\[
\sigma_{t_k}^{(p)}(s_k)=\Gamma_p(t_k)=\gamma_{v_{t_k}^{(p)}}(s_k),
\qquad
\sigma_p(s_k)=\gamma_{v^{(p)}}(s_k).
\]
Moreover, $\gamma_{v_{t_k}^{(p)}}|_{[0,s_k]}=\sigma_{t_k}^{(p)}$ is minimizing by definition, and $\gamma_{v^{(p)}}|_{[0,s_k]}=\sigma_p|_{[0,s_k]}$ is minimizing because $\sigma_p$ is a minimizing geodesic ray. Therefore, by Toponogov's theorem, 
\[
d\bigl(\Gamma_p(t_k),\sigma_p(s_k)\bigr)
=
d\bigl(\gamma_{v_{t_k}^{(p)}}(s_k),\gamma_{v^{(p)}}(s_k)\bigr)
\le
2s_k\sin\!\left(\frac{\angle(v_{t_k}^{(p)},v^{(p)})}{2}\right).
\]
Dividing by $s_k$ and using $\Gamma_p(t_k)\in \Gamma_p([0,\infty))$, we obtain
\[
\frac{d\bigl(\sigma_p(s_k),\Gamma_p\bigr)}{s_k}\le \frac{d\bigl(\Gamma_p(t_k),\sigma_p(s_k)\bigr)}{s_k}
\le
2\sin\!\left(\frac{\angle(v_{t_k}^{(p)},v^{(p)})}{2}\right)\to0.
\]
This proves \eqref{eqn:obtained-ray-stays-asymptotically-close-to-curve}.
\end{proof}

\begin{lemma}
\label{lem:value-of-angle-function-on-subsequential-limit-rays}
Let $p\in M\setminus\{o\}$ and $v^{(p)}:=\sigma_p'(0)$. Then,
\[
\alpha(v^{(p)})=\sqrt{1-G(p)}.
\]
\end{lemma}
\begin{proof}
Set $v:=v^{(p)}$, so that $\sigma_p=\gamma_v$. Let $s_k\to\infty$ be such that \eqref{eqn:obtained-ray-stays-asymptotically-close-to-curve} holds. Choose $q_k\in \Gamma_p([0,\infty))$ such that $d\bigl(q_k,\gamma_v(s_k)\bigr)=d\bigl(\gamma_v(s_k),\Gamma_p\bigr)$ and write $q_k=\Gamma_p(l_k)$ for some $l_k\ge0$. Then
\begin{equation}
\label{eqn:obtained-ray-stays-asymptotically-close-to-curve-2}
\lim_{k\to\infty}\frac{d\bigl(q_k,\gamma_v(s_k)\bigr)}{s_k}= 0.
\end{equation}
Since $|\nabla f|\le1$,
\[
\left|
\frac{f(o)-f(\gamma_v(s_k))}{s_k}
-
\frac{f(o)-f(q_k)}{s_k}
\right|
\le
\frac{d\bigl(q_k,\gamma_v(s_k)\bigr)}{s_k}\to0.
\]
Combined with \eqref{eqn:defn-of-angle-function-alpha-2}, it follows that
\[
\alpha(v)=\lim_{k\to \infty}\frac{f(o)-f(\gamma_v(s_k))}{s_k}=\lim_{k\to \infty}\frac{f(o)-f(q_k)}{s_k}.
\]
Since $\gamma_v$ is a minimizing ray, $d(o,\gamma_v(s_k))=s_k$. By the triangle inequality, \[ s_k-d(q_k,\gamma_v(s_k)) \le d(o,q_k)\le s_k+d(q_k,\gamma_v(s_k)). \] Dividing by $s_k$ and letting $k\to\infty$, it follows from \eqref{eqn:obtained-ray-stays-asymptotically-close-to-curve-2} that 
$$ \lim_{k\to \infty}\frac{d(o,q_k)}{s_k}= 1. $$
In particular, $d(o,q_k)\to\infty$, hence $l_k\to\infty$. Using \eqref{eqn:distance-p-Phi-from-o-1}, $\frac{d(o,q_k)}{l_k}=\frac{d(o,\G_p(l_k))}{l_k}\to 1$ and we obtain
\[
\frac{l_k}{s_k}
=
\frac{l_k}{d(o,q_k)}\cdot \frac{d(o,q_k)}{s_k}
\to 1.
\]
It follows that 
$$\begin{aligned}
\alpha(v)&=\lim_{k\to \infty}\frac{f(o)-f(q_k)}{s_k}=\lim_{k\to \infty}\frac{f(o)-f(q_k)}{l_k}\\
&=\lim_{k\to \infty}\frac{f(o)-f(\G_p(l_k))}{l_k}\\
&=\lim_{k\to \infty}\frac{1}{l_k}\int_0^{l_k}|\nabla f|(\G_p(s))\,ds=\lim_{s\to \infty}|\nabla f|(\G_p(s))=\sqrt{1-G(p)}.
\end{aligned}$$
This completes the proof. 
\end{proof}
We now prove Theorem~\ref{thm:converse-to-main-theorem-2-stronger-scalar-bound-4d}. 
\begin{theorem}
\label{thm:converse-ray-implies-Rto0}
If the asymptotic cone of $(M,g)$ is a ray, then
\[
\lim_{d(x,o)\to\infty}R(x)=0.
\]
\end{theorem}
\begin{proof}
Let $v_1,v_2\in Z$. Since the asymptotic cone of $(M,g)$ is a ray and $\gamma_{v_1},\gamma_{v_2}$ are minimizing rays, we obtain $\lim_{r\to\infty}\frac{d(\gamma_{v_1}(r),\gamma_{v_2}(r))}{r}=0.$ Since $|\nabla f|\le1$, we have
\[
\left|
\frac{f(o)-f(\gamma_{v_1}(r))}{r}
-
\frac{f(o)-f(\gamma_{v_2}(r))}{r}
\right|
\le
\frac{d(\gamma_{v_1}(r),\gamma_{v_2}(r))}{r}.
\]
Letting $r\to\infty$ and using \eqref{eqn:defn-of-angle-function-alpha-2}, we obtain $\alpha(v_1)=\alpha(v_2)$. Thus $\alpha$ is constant on $Z$.

For each $p\in \Sigma_{s_0}=\Sigma$, set $v_p:=\sigma_p'(0).$ Then, by Lemma~\ref{lem:obtained-ray-stays-asymptotically-close-to-curve} and Lemma~\ref{lem:value-of-angle-function-on-subsequential-limit-rays}, $v_p\in Z$ and $\alpha(v_p)=\sqrt{1-G(p)}.$ Since $\alpha$ is constant on $Z$, it follows that $G$ is constant on $\Sigma$; write $G\equiv \eta$ on $\Sigma$. 

We claim that $\eta=0$. Suppose instead that $\eta>0$. It follows from Lemma~\ref{lem:monotonicity-along-flows} that 
\[
R\ge \eta\qquad\text{on }\Sigma_s\quad\text{for all }s\le s_0.
\]
For $s\le s_0$, define
\[
A(s):=\int_{\Sigma_s} |\nabla f|\,d\mu_s,
\]
where $d\mu_s$ is the volume form induced by the metric $g|_{\Sigma_s}$. The first variation formula gives
\[
A'(s)=-\int_{\Sigma_s}\frac{R}{|\nabla f|}\,d\mu_s.
\]
Using $R\ge \eta$ on $\Sigma_s$ and $|\nabla f|\le1$, we obtain
\[
A'(s)\le -\eta A(s)\qquad\text{for all }s\le s_0.
\]
Hence
\[
 \operatorname{Vol}(\Sigma_s)\geq A(s)\ge A(s_0)e^{\eta (s_0-s)}\qquad\text{for all }s\le s_0.
\]
The coarea formula and $|\nabla f|\le1$ give
\begin{equation}
\label{eqn:1-in-proof-thm:converse-ray-implies-Rto0}
\begin{aligned}
\operatorname{Vol}\bigl(\{s\leq f\leq s_0\}\bigr)
&=\int_s^{s_0}\int_{\Sigma_u}\frac{1}{|\nabla f|}\,d\mu_u\,du\\
&\ge \int_s^{s_0}\operatorname{Vol}(\Sigma_u)\,du\\
&\ge A(s_0)\eta^{-1}(e^{\eta(s_0-s)}-1),
\end{aligned}
\end{equation}
for all $s\leq s_0$. On the other hand, by Lemma~\ref{lem:Linear-Growth-of-potential} there exist constants $c_0>0$ and $b\in\mathbb R$ such that $f(o)-f(x)\ge c_0\,d(o,x)-b$ for all $x\in M$. Therefore $\{s\le f\le s_0\}
\subset
B\!\left[o;c_0^{-1}(f(o)-s+b)\right]$. Combining this with \eqref{eqn:1-in-proof-thm:converse-ray-implies-Rto0}, we obtain constants $c,c''>0$ such that 
\[
\operatorname{Vol}(B_g[o;r])\ge c\,e^{c''r},
\]
for all sufficiently large $r$. This contradicts Bishop--Gromov volume comparison. Thus $\eta=0$, so $G\equiv0$ on $\Sigma$. By Lemma~\ref{G-on-Level-sets-lemma}, this is equivalent to
\[
\lim_{d(x,o)\to\infty}R(x)=0.
\]
\end{proof}
None of the arguments in the proof of Theorem~\ref{thm:converse-ray-implies-Rto0} use the assumption $n=4$. Hence the same argument yields the following statement. 
\begin{theorem}\label{thm:dimension-free-converse-to-main-theorem-2}
Let $(\mathcal N^n,g_{\mathcal N},f)$, $n\ge 2$, be a complete noncompact gradient steady
Ricci soliton. Assume that $\sec_{g_{\mathcal N}}\ge 0, \Ric_{g_{\mathcal N}}>0$ on $\mathcal{N}$ and that the potential function $f$ has a critical point $o\in \mathcal N$. If the asymptotic cone of $(\mathcal N,g_{\mathcal N})$ is a ray, then
\[
\lim_{d(x,o)\to\infty} R(x)=0.
\]
\end{theorem}
We end this section by proving Corollary~\ref{cor:ray-iff-R-vanishes}.
\begin{proof}[Proof of Corollary \textup{\ref{cor:ray-iff-R-vanishes}}]
(1)$\Rightarrow$(2) follows from Theorem~\ref{thm:full-linear-bound-proof}.

\noindent (2)$\Rightarrow$(1). If $\lim_{d(x,o)\to \infty}R(x)d(x,\G)=0$, choose points $x_k$ where $d(x_k,o)\to \infty$ and $d(x_k,\G_1)=1$ to obtain that $R(x_k)\to 0$. By Perelman's long-range estimates, $G(x_+)=0$. Similarly $G(x_-)=0$. By Theorem~\ref{existence-of-x-+-} and Lemma~\ref{G-on-Level-sets-lemma}, $\lim_{d(x,o)\to \infty}R(x)=0$. 

\noindent (1)$\Rightarrow$(3) follows from Theorem~\ref{thm:asymptotic-cone-ray-proof}.

\noindent (3)$\Rightarrow$(1) follows from Theorem~\ref{thm:converse-ray-implies-Rto0}. 
\end{proof}

\appendix

\section{Some auxiliary results}
\label{sec:appendix-Some-auxiliary-results}

\setcounter{theorem}{0}
\renewcommand{\thetheorem}{\thesection.\arabic{theorem}}
In this section, we collect several lemmas used throughout the paper. We begin by recalling the definition of convergence of functions and points under the Cheeger--Gromov convergence of manifolds. 
\begin{definition}
\label{defn:c-g-convergence-of-functions-general-manifolds}
Let $(M_i,g_i,x_i)_{i\geq 1}$ and $(N,g,x_0)$ be pointed complete Riemannian manifolds with $(M_i,g_i,x_i)\to(N,g,x_0)$ in the smooth Cheeger--Gromov sense, i.e., for all large $i$, there exist domains $U_i\subset N$ exhausting $N$ and smooth embeddings $\Phi_i:U_i\to M_i$ with $\Phi_i(x_0)=x_i$ such that $\Phi_i^*g_i\to g$ in $C^\infty_{\operatorname{loc}}(N)$. Let $f_i\in C^\infty(M_i)$ and $f\in C^\infty(N)$. We say $f_i\to f$ in the sense of smooth Cheeger--Gromov if 
$$f_i\circ \Phi_i=\Phi_i^*f_i\to f \quad\text{in } C^\infty_{\operatorname{loc}}(N).$$
We also say $y_i\in M_i$ converges to $y_0\in N$ if $y_i\in \Phi_i(U_i)$ for all large $i$ and $\Phi_i^{-1}(y_i)\to y_0$ in $(N,g)$. 
\end{definition}
The next lemma shows the convergence of regular level sets under the Cheeger--Gromov convergence of the underlying manifolds. The proof follows from the implicit function theorem and the details are omitted. 
\begin{lemma}
\label{Gen-Level-Set-Convergence}
Let $(M_i,g_i,x_i)_{i\geq 1},(N,g,x_0)$ be pointed complete Riemannian manifolds such that $(M_i,g_i,x_i)\to (N,{g},x_0)$ as $i\to \infty$, in the smooth Cheeger--Gromov sense. Suppose that the following holds: 
\begin{enumerate}
    \item $f_i\in C^\infty(M_i)$, $f_i(x_i)=0$,
    \item $f\in C^\infty(N),f(x_0)=0,f_i\to f$ in the smooth Cheeger--Gromov sense, 
    \item $\Sigma_{i}=f_i^{-1}(0)$ is a regular connected hypersurface in $M_i$, and
    \item $\Sigma=f^{-1}(0)$ is also a regular connected hypersurface with $\nabla f\neq 0$ on $\Sigma$.
\end{enumerate}
 Then, $(\Sigma_i, g_i|_{\Sigma_i}, x_i)\to (\Sigma, g|_{\Sigma}, x_0)$ in the smooth Cheeger--Gromov sense. Moreover, if $\tilde u_i\in C^\infty(M_i)$ and $\tilde u\in C^\infty(N)$ satisfy $\tilde{u}_i\to \tilde{u}$ in the smooth Cheeger--Gromov sense, then $\tilde u_i|_{\Sigma_i}\to \tilde u|_{\Sigma}$ in the smooth Cheeger--Gromov sense.
\end{lemma}
We now prove a monotonicity property of the Bryant soliton: the normalized smallest Ricci eigenvalue strictly decreases along the radial direction. 
\begin{lemma}
\label{lemma:Bryant-monotonicity}
Let $\left(\Bry^n, g, f\right)$ be the $n \geq 3$ Bryant soliton, with tip $\bar{x}$ and $R(\bar{x})=1$ at the tip. Let $\lambda_{\text {min }}(r)$ be the smallest Ricci eigenvalue at distance $r$ from the tip. Then 
$$
h(r):=\frac{\lambda_{\min }(r)}{R(r)}
$$
is strictly decreasing for $r>0$. In particular, $h(r)<1 / n$ for $r>0$, and $h(0)=1 / n$. 
\end{lemma}
\begin{proof}
Write the rotationally symmetric Bryant soliton metric as
$$
g=d r^2+w(r)^2 g_{\mathbb{S}^{n-1}}, \quad a:=\frac{w^{\prime}}{w},
$$
where $r>0$ is the radial coordinate. Assume the steady soliton convention $\Ric+\nabla^2 f=0$ and $R+|\nabla f|^2=1$ where $f$ is a smooth function of $r$. As $\Ric>0$ on $\Bry^n$, we have $f''<0$ for $r>0$. Hence, $f'(r)<f'(0)=0$ for $r>0$. Set $u:=-f^{\prime}>0$ for $r>0$ so that  $u'=-f''>0$ for $r>0$. For the warped product metric, the Ricci tensor is 
$$\Ric_g=-(n-1) \frac{w^{\prime \prime}}{w} d r^2+\left((n-2)\left(1-\left(w^{\prime}\right)^2\right)-w w^{\prime \prime}\right) g_{\sph^{n-1}},$$
and the soliton equation gives  
\begin{equation}
\label{eqn:Bry-monotone-0}
\begin{cases}f^{\prime \prime} & =(n-1) w^{\prime \prime} / w \\ 
w w^{\prime} f^{\prime}& =w w^{\prime \prime}+(n-2)\left(\left(w^{\prime}\right)^2-1\right).\end{cases}
\end{equation}
See, for example, equations (3) and (7) in \cite{Kot08}. This implies that the eigenvalues of the Ricci tensor are given by 
$$
\lambda_r=-f^{\prime \prime}=u^{\prime}, \quad \lambda_t=-a f^{\prime}=a u,
$$
and the scalar curvature is given by  
\begin{equation}
\label{eqn:Bry-monotone-1}
R=\lambda_r+(n-1) \lambda_t=u^{\prime}+(n-1) a u .
\end{equation}
In particular, $a>0$. We first show that $\lambda_r$ is the smallest eigenvalue for all $r>0$. Let
$$
D:=\lambda_t-\lambda_r=a u-u^{\prime} .
$$
Using $\nabla R=2\Ric(\nabla f)$, with $\nabla f=f^{\prime} \partial_r=-u \partial_r$,
$$
R^{\prime}=2 f^{\prime} \lambda_r=-2 u u^{\prime} .
$$
Differentiating (\ref{eqn:Bry-monotone-1}) and using the above equation, we obtain 
$$
-2uu'=R^{\prime}=u^{\prime \prime}+(n-1)\left(a^{\prime} u+a u^{\prime}\right) .
$$
Therefore, 
\begin{equation}
\label{eqn:Bry-monotone-2}
u^{\prime \prime}=-2 u u^{\prime}-(n-1)\left(a^{\prime} u+a u^{\prime}\right) .
\end{equation}
Next, use (\ref{eqn:Bry-monotone-0}) to see that 
\begin{equation}
\label{eqn:Bry-monotone-3}
a^{\prime}=\left(\frac{w^{\prime}}{w}\right)^{\prime}=\frac{w^{\prime \prime}}{w}-a^2=-\frac{u^{\prime}}{n-1}-a^2<0 .
\end{equation}
We compute $D^{\prime}$:
$$
D^{\prime}=a^{\prime} u+a u^{\prime}-u^{\prime \prime} .
$$
Using  (\ref{eqn:Bry-monotone-2}), (\ref{eqn:Bry-monotone-3}) we obtain the ODE 
$$
D^{\prime}+n a D=\frac{n-2}{n-1} u u^{\prime}>0 .
$$
At the tip, $\lambda_r(0)=\lambda_t(0)$, hence $D(0)=0$. Since $a>0$ for $r>0$, we obtain $(Dw^n)'=w^n(D'+naD)>0$ for $r>0$, which shows $D(r)>0$ for all $r>0$. Therefore, 
$$
\lambda_t>\lambda_r \quad(r>0),
$$
so the smallest Ricci eigenvalue is
$$
\lambda_{\min }=\lambda_r=u^{\prime} .
$$
We now derive an ODE for $h:=\lambda_{\min } / R=u^{\prime} / R$. Differentiate to obtain  
$$
h^{\prime}=\frac{u^{\prime \prime}}{R}-\frac{u^{\prime} R^{\prime}}{R^2} .
$$
An algebraic simplification using \eqref{eqn:Bry-monotone-1} and \eqref{eqn:Bry-monotone-2} gives the ODE
\begin{equation}
\label{eqn:Bryant-monotone-4}
h^{\prime}=a(1-n h)+u h(2 h-1) .
\end{equation}
Also, since $\lambda_t>\lambda_r$, we have $1>R\geq n\lambda_r$ for $r>0$ which implies 
$$
0<h(r)<\frac{1}{n} \quad \text{ if } r>0, \quad h(0)=\frac{1}{n} .
$$
If $r_0>0$ with $h^{\prime}\left(r_0\right)=0$, then differentiating \eqref{eqn:Bryant-monotone-4} gives
$$
h^{\prime \prime}\left(r_0\right)=a^{\prime}(1-n h)+u^{\prime} h(2 h-1) .
$$
For $r_0>0$ we have $h>0$, $a^{\prime}<0$, $1-n h>0$, $u'>0$, and $2 h-1<0$. Hence $h^{\prime \prime}\left(r_0\right)<0$. Thus, every critical point of $h$ in $(0,\infty)$ satisfies $h''<0$.

We now analyze the behavior of $h$ near $r=0$. At the tip $r=0$, $f'(0)=0$ and $f''(0)=-\frac{1}{n}$, $f'''(0)=0$ hence, near $r=0$, 
$$u(r)=\frac{1}{n} r+O\left(r^3\right).$$
Near $r=0$, completeness gives 
$$w(r)=r+O\left(r^3\right),$$ 
which implies 
$$a(r)=\frac{1}{r}+O(r).$$
As $R(0)=1$ and $R'(0)=0$, we obtain near $r=0$,  
$$
R(r)=1+O\left(r^2\right),
$$
so there exists a constant $c$ such that 
$$
h(r)=\frac{u^{\prime}(r)}{R(r)}=\frac{1}{n}+c r^2+O\left(r^4\right) .
$$
Substituting this expansion into (\ref{eqn:Bryant-monotone-4}), we obtain 
$$
\begin{gathered}
1-n h=-n c r^2+O\left(r^4\right), \quad a(1-n h)=-(n c) r+O\left(r^3\right), \\
u h(2 h-1)=\left(\frac{r}{n} \right)\left(\frac{1}{n}\right)\left(\frac{2}{n}-1\right)+O\left(r^3\right)=\frac{(2-n)}{n^3} r+O\left(r^3\right).
\end{gathered}
$$
Equating this with $h'=2cr+O\left(r^3\right)$ implies $c=-\frac{(n-2)}{n^3(n+2)}<0$. 
This proves that $h'(r)<0$ for all sufficiently small $r>0$. 
We claim that $h'<0$ on $(0,\infty)$. 
Otherwise, the set $S:=\{r\in (0,\infty):h'(r)\geq 0\}$ is nonempty and satisfies $r_*:=\inf S\in (0,\infty)$. 
Then $h'\leq 0$ on $\left(0, r_*\right]$ and $h'(r_*)\geq 0$. 
This implies $h'(r_*)=0$ and $h''(r_*)\geq 0$, which is impossible. 
Therefore, $h$ is strictly decreasing on $(0,\infty)$, completing the proof of Lemma \ref{lemma:Bryant-monotonicity}. 
\end{proof}

\begin{remark}
\label{rem:Bryant-canonical-flow-monotonicity}
Recall the notation of Definition~\ref{def:model-flows}. Let $\tilde x$ denote the tip of $\Bry^3$, and let $\Psi_t$ denote the diffeomorphisms generating the canonical flow on the Bryant soliton, so that $\tilde g_t=\Psi_t^*\tilde g_0$. Then for every $w\in \Bry^3$ and every $t\le 0$, 
\[
\left(\frac{\lambda_1}{R}\right)_{\tilde g_t}(w)
=
\left(\frac{\lambda_1}{R}\right)_{\tilde g_0}(\Psi_t(w)).
\]
Let $x\neq \tilde{x}$. The function $t\mapsto d_{\tilde g_0}(\Psi_t(x),\tilde x)$ is strictly decreasing. By Lemma \ref{lemma:Bryant-monotonicity}, $t\mapsto \left(\frac{\lambda_1}{R}\right)_{\tilde g_t}(x)$ is strictly increasing on $(-\infty,0]$. Moreover, $\left(\frac{\lambda_1}{R}\right)_{\tilde g_t}(x)\to 0$ as $t\to -\infty$. In particular, 
\[
\left(\frac{\lambda_1}{R}\right)_{\tilde g_t}(w)
\le
\left(\frac{\lambda_1}{R}\right)_{\tilde g_0}(w)
\qquad\text{for all }w\in \Bry^3,\ t\le 0.
\]
\end{remark}
We shall use the following variation of distance formula for the Ricci flow on a steady soliton. It is obtained by applying the standard distance distortion estimate \cite[Lemma 18.1]{CCG+10} for the backward Ricci flow.   
\begin{lemma}
\label{lem:variation-of-distance}
Let $(M^4,g,f)$ be a steady soliton and let $\Phi_t$ denote the diffeomorphisms generated by $\nabla f$ with $\Phi_0=\operatorname{id}_M$, so that $g(t)=\Phi_t^*g$ for all $t\in \R$. For any $x,y\in M$ and $s\geq 0$, let $\mathcal{Z}(x,y,-s)$ denote the set of all unit-speed minimizing geodesics joining $x$ and $y$, with respect to $g(-s)$. Then, 
$$\p_s^+ d_g(\Phi_{-s}(x),\Phi_{-s}(y))=\p_s^+ d_{g(-s)}(x,y)=\max_{\mathcal{Z}(x,y,-s)}\int_\g \Ric_{g(-s)}(\g',\g'),$$
where $\p_s^+$ takes limsup of backward difference quotients. In particular, if $\Ric\geq 0$ on $M$, then $s\mapsto d_g(\Phi_{-s}(x),\Phi_{-s}(y))$ is nondecreasing. 
\end{lemma}
We also need the following stability inequality. 
\begin{lemma}
\label{stability-inequalities}
Let $(\mathcal{N}^n,g_\mathcal{N}),n\geq 2$, be a Riemannian manifold with $\Ric\geq 0,R>0$ and let $x_1,x_2\in \mathcal{N}$, $A\geq 1$.  For $i=1,2$ assume 
$$R(y)\leq A\, R(x_i)\qquad \text{for all } y\in B_{\mathcal{N}}[x_i;R(x_i)^{-1/2}].$$
If $\g:[0,L]\to \mathcal{N}$ is a unit-speed minimizing $g_\mathcal{N}$-geodesic such that $\g(0)=x_1,\g(L)=x_2$, then
\begin{equation}
\label{eqn:stability-estimate-lemma}
\int_{\g}\Ric(\g',\g')\leq 4(n-1)\sqrt{A}\max\left(\sqrt{R(x_1)},\sqrt{R(x_2)}\right).
\end{equation}
\end{lemma}
\begin{proof}
Let $B:=\max(R(x_1),R(x_2))$ and set $K:=AB$. Since $\Ric\ge 0$, we have $\Ric\le Rg_\mathcal{N}$ pointwise. Thus, for $i=1,2$ and every $y\in B_{R(x_i)g_\mathcal{N}}[x_i;1]=B_{\mathcal{N}}[x_i;R(x_i)^{-1/2}]$,
\[
\Ric(y)\le R(y)g_\mathcal{N} \le A\,R(x_i)g_\mathcal{N} \le AB\,g_\mathcal{N} = K\,g_\mathcal{N}.
\]
Since $K=AB\ge B\ge R(x_i)$, we have $B_{g_\mathcal{N}}\!\left[x_i;K^{-1/2}\right]\subset B_{R(x_i)g_\mathcal{N}}[x_i;1]$ for $i=1,2$. Hence
\[
\Ric(y)\le K\,g_\mathcal{N} \le (n-1)K\,g_\mathcal{N}
\qquad\text{for all }y\in B_{g_\mathcal{N}}\!\left[x_1;K^{-1/2}\right]\cup B_{g_\mathcal{N}}\!\left[x_2;K^{-1/2}\right].
\]
Since $\g$ is minimizing, it is stable. Hence by \cite[Proposition 1.94]{CLN06},
\[
\int_\gamma \Ric(\gamma',\gamma')\,ds \le 4(n-1)\sqrt K
=4(n-1)\sqrt{AB}.
\]
This proves the lemma.
\end{proof}

\begin{remark}
\label{rem:stability-inequalities-application-remark}
In this paper, Lemma \ref{stability-inequalities} is applied together with Lemma \ref{lem:variation-of-distance}. Suppose $(M^4,g,f)$ satisfies \hyperref[assumption:A1]{(A1)}--\hyperref[assumption:A3]{(A3)}. Use the notation of Section~\ref{sec:Preliminaries}. By Perelman's long-range estimates (Lemma \ref{summary-of-CMZ-work} (ii)) there exists a uniform constant $C$ such that $R_{g_t}(y)\leq CR_{g_t}(x)$ whenever $x,y\in M,t\leq 0$ and $R_{g_t}(x)d_{g_t}(x,y)^2\leq 1$. This implies that the hypotheses of Lemma \ref{stability-inequalities} are satisfied on $(M^4,g_t),t\leq 0$ for any points $x_1,x_2$ and $g_t$-geodesics $\g$ with a uniform constant $A$. From Lemma \ref{lem:variation-of-distance} and Lemma \ref{stability-inequalities} we conclude that there exists a positive constant $C$ such that 
$$\p_s^+d_g(\Phi_{-s}(x),\Phi_{-s}(y))\leq C\max\left(\sqrt{R_g(\Phi_{-s}(x))},\sqrt{R_g(\Phi_{-s}(y))}\right),$$ for all $x,y\in M$ and $s>0$. 
\end{remark}

\end{document}